\documentclass[11pt,reqno]{amsart} %GD command
\usepackage{bbm, dsfont}
\usepackage{enumerate,ulem}
\usepackage{mdframed}
\usepackage{verbatim}
\usepackage{paralist}
\usepackage{mathrsfs, stmaryrd}
\usepackage{mathrsfs}
\usepackage[left=2.5cm,right=2.5cm, top=3.25cm, bottom=3.25cm]{geometry}
\usepackage{lipsum}
\usepackage{amsfonts, amsmath, amsthm}
\usepackage{mathrsfs}
\usepackage{graphicx, xcolor,fancyhdr}
\usepackage{epstopdf}
\usepackage{algorithmic}
\ifpdf
\DeclareGraphicsExtensions{.eps,.pdf,.png,.jpg}
\else
\DeclareGraphicsExtensions{.eps}
\fi

\definecolor{darkblue}{rgb}{0.1,0.1,0.9}%%%COLOR FOR SIDE COMMENT

\newcommand{\hdoubtz}[1]{\marginpar{\raggedright\scriptsize{\textcolor{red}{#1}}}}

\newcommand\dela[1]{}

\newcommand\cnew[1]{\textcolor{purple}{#1}}
\newcommand\deln[1]{}

\newcommand{\embed}{\hookrightarrow }
\usepackage{fouridx}

\numberwithin{equation}{section}

\newcommand{\me}{\mathbb{E}}
\newcommand\del[1]{}

\newcommand{\rA}{\mathrm{A}}
\newcommand{\rrA}{\mathrm{A}_1}
\newcommand{\hrrA}{\hat{\mathrm{A}}_1}
\newcommand{\trrA}{\tilde{\mathrm{A}}_1}

\newcommand{\rH}{\mathrm{ H}}
\newcommand{\bH}{\mathbf{H}}
\newcommand{\rK}{\mathrm{K}}

\newcommand{\divv}{\mathrm{div }\;}

\newcommand{\err}{\mathbb{R}}

\newcommand\toup{\nearrow}

\renewcommand{\v}{\bv}

\newcommand{\ba}{\mathbf{a}}

\newcommand{\bx}{\mathbf{x}}

\newcommand{\bX}{\mathbf{X}}

\newcommand{\db}{\bar{\d}}
\newcommand{\A}{\mathbf{A}}
\newcommand{\f}{\mathbf{F}}

\newcommand{\bw}{\mathbb{W}}
\newcommand{\bW}{\mathbf{W}}
\newcommand{\y}{\mathbf{y}}
\newcommand{\z}{\mathbf{z}}

\renewcommand{\d}{\mathbf{n}}
\newcommand{\MO}{\mathcal{O}}

\newcommand{\w}{\mathbf{w}}

\newcommand{\bv}{\mathbf{v}}

\newcommand{\bd}{\mathbf{n}}

\newcommand{\el}{\mathrm{L}}
\newcommand{\elb}{\mathbf{L}}
\newcommand{\ve}{\mathrm{V}}
\newcommand{\h}{\mathrm{H}}
\newcommand{\mo}{\mathcal{O}}

\newcommand{\bn}{\boldsymbol{\nu}}
\newcommand{\bu}{\mathbf{u}}

\newcommand{\eps}{\varepsilon}
\newcommand{\rve}{\rVert}
\newcommand{\lve}{\lVert}

\newcommand{\E}{\mathbb{E}}

\newcommand{\bh}{\mathbf{h}}
\DeclareMathOperator{\Div}{div}
\theoremstyle{plain}
\newtheorem{condition}{Condition}[section]
\newtheorem{assum}[condition]{Assumption}

\newtheorem{lem}{Lemma}[section]
\newtheorem{thm}[lem]{Theorem}
\newtheorem{prop}[lem]{Proposition}
\newtheorem{cor}[lem]{Corollary}
\theoremstyle{definition}
\newtheorem{Def}[lem]{Definition}
\newtheorem{Rem}[lem]{Remark}
\title[Stochastic nematic liquid crystals with multiplicative noise]{Strong solution to stochastic  penalised nematic liquid crystals model driven by multiplicative Gaussian noise}%{Local and global strong solution to a stochastic PDEs  for penalised nematic liquid crystals driven by multiplicative noise}

\author[Z. Brze\'zniak]{Zdzis{\l}aw Brze{\'z}niak}
\address{Department of Mathematics \\
	University of York, Heslington, York YO10
	5DD, UK} \email{zdzislaw.brzezniak@york.ac.uk}

\author[E. Hausenblas]{Erika Hausenblas}

\address{Department of Mathematics and Information Technology\\  Montanuniversity of Leoben,
	Franz Josef Strasse 18, 8700 Leoben, Austria} \email{
	erika.hausenblas@unileoben.ac.at}

\author[P. Razafimandimby]{Paul Andr\'e  {Razafimandimby}}
\address{Department of Mathematics\\University of York, Heslington, York YO10
	5DD, UK} \email{paul.razafimandimby@york.ac.uk}
\thanks{This article is part of a project that is currently funded by the  European Union's Horizon 2020 research and innovation programme under the Marie Sk\l{}odowska-Curie grant agreement No. 791735 ``SELEs".}
\begin{document}
	\maketitle
	\begin{abstract}
		In this paper, we prove the existence of a unique maximal local strong solutions to a stochastic system for both 2D and 3D penalised nematic liquid crystals driven by multiplicative Gaussian noise.   In the 2D case, we show that this solution is global.
		%%%%%
%	\red{	In contrast to several works in the deterministic setting we replace the Ginzburg-Landau function
%		$\mathds{1}_{\lvert \bd\rvert\le 1}(\lvert \d\rvert^2-1)\d$ by an appropriate polynomial $f(\d)$ and we give sufficient conditions on the polynomial $f$ for the aforementioned results to hold.}
	%%%%
		As a by-product of our investigation, but of independent interest,  we present a general method based on fixed point arguments to establish the existence and uniqueness of a maximal local solution of an abstract stochastic evolution equations with coefficients satisfying local Lipschitz condition involving the norms of two different Banach spaces.
	\end{abstract}
	\section{Introduction}
	%%%%%%%%%%%%%%%%%%%%%%%%%%%%%%%%%%%%%%%%%%%%%%%%%%%%%%%%%%%%%%%%%%%%%%
	%%%%%%%%%%%%%%%%%%%%%%%%%%%%%%%%%%%%%%%%%%%%%%%%%%%%%%%%%%%%%%%%%%%%%%%%%
	Nematic liquid crystal (NLC) is a liquid crystal phase  with has rod-shaped molecules which tend to align along a particular direction denoted by a unit vector $\mathbf{n}$, called the optical director axis. In addition to $\bd:\mathbb{R}^d \to \mathbb{R}^3$ , the hydrodynamic of an isothermal and incompressible NLC is also described by its pressure $p:\mathbb{R}^d\to \mathbb{R}$ and velocity $\bv:\mathbb{R}^d\to \mathbb{R}^d$.
We refer to \cite{Chandrasekhar}
and \cite{Gennes} for a comprehensive treatment of the physics of
liquid crystals.
	
Using the  Ericksen and Leslie continuum theory for liquid crystals, see \cite{Ericksen} and
Leslie \cite{Leslie},  F. Lin and C. Liu
\cite{Lin-Liu} derived the most basic and simplest form of the dynamical system  modeling the motion of a nematic liquid crystal (NLC)
flowing in $\mathbb{R}^d (d=2,3)$. This system is given by 
		\begin{align}
	&	d\bv+\biggl[(\bv\cdot\nabla)\bv-\Delta \bv+\nabla p\biggr]dt=  -\nabla\cdot(\nabla \bd \odot \nabla \bd)dt+ \mathbf{f} dt,\label{eqn-SELE-v} \\
	&	\divv \bv=  0,\label{eqn-SELE-div} \\
	&	d\bd+(\bv\cdot\nabla)\bd dt =  \biggl[\Delta
	\bd+ \lvert \nabla \d \rvert^2 \d \biggr]dt+\mathbf{g} dt , \label{eqn-SELE-d}\\
	& \lvert \bd \rvert^2= 1,\label{eqn-SELE-sphere}
	\end{align}
	where $\mathbf{f}$ and $\mathbf{g}$ are forcings acting on the system.
	 The entries of the matrix $\nabla \bd \odot 
\nabla \bd$ are defined by 
\begin{equation*}
[\nabla \bd \odot \nabla \bd]_{i,j}=\sum_{k=1}^3 \frac{\partial
	\bd^{(k)}}{\partial x_i}\frac{\partial \bd^{(k)}}{\partial x_j},\;\;
\mbox{ } i,j=1,\dots, d.
\end{equation*}
Before proceeding further, we should mention that \eqref{eqn-SELE-v}-\eqref{eqn-SELE-sphere} is obtained by neglecting several terms such as the viscous Leslie stress tensor in the equation for $\bv$, the stretching and rotational effects for $\mathbf{d}$. Thus, it is not known whether the models \eqref{eqn-SELE-v}-\eqref{eqn-SELE-sphere} and \eqref{eqn-SLQE-v}-\eqref{eqn-SLQE-d} are thermodynamically stable or consistent. However, these models still retain many mathematical and essential features of the dynamics for NLCs. In  the recent papers \cite{MH+JP-2017} \cite{MH+JP-2018}, \cite{MH et al-2014}, \cite{Lin+Wang-2014} and \cite{Sun+Liu} several thermodynamically consistent and stable models of NLC have been developed and analysed.
	
	In this paper, we fix a bounded domain $\MO\subset  \mathbb{R}^d$, $d=12,3$ with smooth boundary  and we consider the following stochastic system
	\begin{align}
	&	d\bv+\Bigl[(\bv\cdot\nabla)\bv-\Delta \bv+\nabla p\Bigr]dt=  -\nabla\cdot(\nabla \bd \odot \nabla \bd)dt+ S(\bv) dW_1,\label{eqn-SLQE-v} \\
	&	\divv \bv= 0,\label{eqn-SLQE-div} \\
	&	d\bd+(\bv\cdot\nabla)\bd dt = \bigl[\Delta
	\bd+f(\d)\bigr]dt+(\bd\times \bh)\circ
	dW_2, \label{eqn-SLQE-d}\\
	&	\bv=0 \text{ and } \frac{\partial \bd}{\partial \bn}=0 \text{ on }
	\partial \MO,\label{BC}\\
& \bv(0)=\bv_0 \text{ and } \bd(0)=\bd_0, \label{Init-Cond}
	\end{align}
	where $\bv_0: \MO \to \mathbb{R}^d$, $\bd_0:\MO\to \mathbb{R}^3$  are given mappings, $\bn$  is the unit outward normal to $\partial \MO$, $f$ is a polynomial function satisfying some conditions to be fixed later. Here, $W_1$ and $W_2$ are respectively independent cylindrical Wiener process and standard Brownian motion, $(\bd\times \bh)\circ dW_2$ is understood in
	the Stratonovich sense.

	The Fr\'eedericksz transition, which is produced by applying a sufficiently strong external perturbation (e.g. magnetic or electric fields) to  an undistorted NLC, and its behaviour under random perturbation have been extensively studied in several physics papers, see \cite{Horsthemke+Lefever-1984, San Miguel-1985,FS+MSanM}, all of which neglected the fluid velocity. However, it is pointed out in \cite[Chapter 5]{Gennes} that the fluid flow disturbs the alignment and conversely a change in the alignment will induce a flow in the nematic liquid crystal. It is this gap in knowledge that is the motivation for our mathematical study which was initiated in the old unpublished preprints \cite{BHP13} and \cite{BHP-arxiv}, see also the recent papers \cite{BHP18} and \cite{BHP19}.

	In this paper, we mainly prove the existence and uniqueness of a maximal local strong solution which is
understood in the sense of stochastic calculus and PDEs. This result is a corollary of several abstract results  which are proved in Section \ref{ABST-STRONG} and are of independent interest. 
 In the case $d=2$, we show the non-explosion of the maximal solution by an adaptation and combination of Khashminskii test for non-explosions and an idea of Schmalfu{\ss} elaborated in \cite{Bjorn}, see Section \ref{SLC-Sect4} for more details. Our novelty is the extension of the Schmalfu{\ss} idea, which has been used so far to prove the uniqueness of solutions of Stochastic Navier-Stokes equations and related problems, to the proof of the global existence of a strong solution to the problem  \eqref{eqn-SLQE-v}-\eqref{Init-Cond}.  In particular, we give another proof of the global existence of 2D stochastic Navier-Stokes equations with multiplicative noise and for initial data with finite enstrophy. Thus, our paper can also be seen as a generalization of the results for the existence and uniqueness of maximal local and global solutions of strong solutions of stochastic Navier-Stokes proved in \cite{Nathan1}, \cite{Nathan2} and \cite{Mikulevicius}.

We should notice that some of the arguments elaborated in Section \ref{ABST-STRONG} have been already used in  \cite{BHR-2014}  and \cite{BHP19} which respectively studied  the strong solution of some stochastic hydrodynamic equations (NSEs, MHD and 3D Leray $\alpha$-models) driven by L\'evy noise,  and  the existence and uniqueness of a maximal local smooth solution to the stochastic Ericksen-Leslie system \eqref{eqn-SELE-v}-\eqref{eqn-SELE-sphere} on the $d$-dimensional torus. We are also strongly convinced that with these general results it is possible, although it  has not been done in detail, to prove the existence of strong solution of several stochastic hydrodynamical models such as the NSEs, MHD equations, $\alpha$-models for Navier-Stokes and related problems.

	While the deterministic version of \eqref{eqn-SLQE-v}-\eqref{eqn-SLQE-d} has been the subject of intensive mathematical studies, see \cite{Rojas-Medar, Hong,Lin-Liu,Lin-Wang,Lin-Wang-2016, Shkoller,Dai+Schonbeck_2014} and \cite{Cavaterra,Hong-2014,Hong-2012,Huang-2014, Wang}, there are fewer results related to the stochastic system \eqref{eqn-SLQE-v}-\eqref{eqn-SLQE-d}. The unpublished paper \cite{BHP13} proved the existence and uniqueness of mxaimal local strong solution to the system \eqref{eqn-SLQE-v}-\eqref{eqn-SLQE-d} with a bounded nonlinear term $f(\d)=\mathds{1}_{\lvert \d \rvert\le 1} (1-\lvert \d\rvert^2)\d $.    The paper \cite{BHP18} deals only with weak (both in PDEs and stochastic calculus sense) solutions and the maximum principle. Some of the results  in \cite{BHP18} and the current paper have already been used in several papers such as  \cite{ZB+UM+AP}, \cite{ZB+UM+AP-2019}, \cite{RZ+GZ}, \cite{BGuo+GZhou-2019},  \cite{GZhou-2019} and  \cite{Wang+Wu+Zhou-2019}.
	%%%%%%%%%%%%%%%%%%%%%%%%%%%%%%
	 %%%%%%%%%%%%%%%%
	  Very recently we have become aware of a recent paper by Feireisl and Petcu \cite{FeirPetcu_2019}, in which they proved the existence of a \textit{dissipative martingale}, as well as the existence of a local strong solution and weak-strong uniqueness of the solution of the stochastic Navier-Stokes Allen-Cahn Equations. Note that in \cite{FeirPetcu_2019} the second unknown $\bd$ is a scalar field, the nonlinear term $f(\cdot)$ is globally Lipschitz and the derivative of a double-well potential $F(\cdot)$, and the coefficient of the noise entering the equations for $\bd$ is bounded.
	The paper \cite{BHP19} is the first paper to deal with the the stochastic counterpart of the  Ericksen-Leslie equations \eqref{eqn-SELE-v}-\eqref{eqn-SELE-sphere}. The results in present manuscript is not covered in \cite{BHP19} because in contrast to our framework which  considers initial condition  $(\v_0, \d_0)\in \h^1 \times \bH^{2}$,  the initial data in \cite{BHP19} satisfies $(\v_0, \d_0)\in \h^{\alpha}\times \h^{\alpha+1}$ for $\alpha>\frac d2$, where $d=2,3$ is the space dimension.  There is also the papers \cite{Medjo1}  which seeks for a special solution $(\bv,\bd)$ with the unknown $\bd$ is replaced by an angle $\theta$ such that $\bd=(\cos\theta, \sin\theta)$. This model reduction considerably simplify the mathematical analysis of \eqref{eqn-SELE-v}-\eqref{eqn-SELE-d}.

	To close this introduction, we emphasize that the analysis in the present paper might also be of great interest in the numerical study of stochastic Ericksen-Leslie system. In fact, on the one hand our assumptions on the polynomial $f(\d)$ enable us to consider the typical Ginzburg-Landau function $f_\eps(\d)=\frac{1}{\eps^2}(1 -\lvert \d \rvert^2) \d$, see Assumption \ref{eqn-f} and Remark \ref{Rem:Ginz-Land-General}. In the numerical context, handling the constraint $\lvert \d\rvert=1$ in the Ericksen-Leslie is a rather challenging task and to overcome this difficulty, one usually use the Ginzburg-Landau approximation, see \cite{Walkington}. On the other hand, to get convergence and a rate of convergence of a space discretization for parabolic SPDEs, one often has to consider a regular solution in favour of weak solutions. 
	%%%%%%%%%%%%%%%%%%%%%%%%%%%%%%%%%%%%%%%%%%%
	%%%%%%%%%%%%%%%%%%%%%%%%%%%%%%%%%%%%%%%%%%%%%%%%%%%%
	%%%%%%%%%%%%%%%%%%%%%%%%%%%%%%%%%%%%%%%%%%%%%%%%%%%%%
	%%%%%%%%%%%%%%%%%%%%%%%%%%%%%%%%%%%%%%%%%%%%%%%%%%%%%
	%%%%%%%%%%%%%%%%%%%%%%%%%%%%%%%%%%%%%%%%%%%%%%%%%%%%%
	%%%%%%%%%%%%%%%%%%%%%%%%%%%%%%%%%%%%%%%%%%%%%%%%%%%%%%%
	%%%%%%%%%%%%%%%%%%%%%%%%%%%%%%%%%%%%%%%%%%%%%%%%%%%%%%%%%%%
	%%%%%%%%%%%%%%%%%%%%%%%%%%%%%%%%%%%%%%%%%%%%%%%%%%%%%%%%%%% %%%%%%%%%%%%%%%%%%%%%%%%%%
	%%%%%%%%%%%%%%%%%%%%%%%%%%%%%%%%%%%%%%%%%%%%%%%%%%%%%%%%%%%%%%%%%%%%%%%%%%%%%%%%%%%%%%%
	%%%%%%%%%%%%%%%%%%%%%%%%%%%%%%%%%%%%%%%%%%%%%%%%%%%%%%%%%%%%%%%%%%%%%%%%%%%%%%%%%%%%%%%%%%%%
	%%%%%%%%%%%%%%%%%%%%%%%%%%%%%%%%%%%%%%%%%%%%%%%%%%%%%%%%%%%%%%%%%%%%%%%%%%%%%%%%%
	%%%%%%%%%%%%%%%%%%%%%%%%%%
	%%%%%%%%%%%%%%%%%%%%%%%%%%
	\section{Preliminary results and notations}\label{sec-spaces-2}
	%%%%%%%%%%%%%%%%%%%%%%%%%%%%%%%%%%%%%%%%%%%%%%%%%%%%%%%%%%%%%%%%%%%
	\subsection{Functional spaces and linear operators}
	Following  \cite{BHP18} we introduce in this section various notations and results that are frequently used in this paper.
	
	 For  two topological spaces $X$ and $Y$ the symbol $X\hookrightarrow Y$ means that the embedding $X$ is continuously embedded in  $Y$. 
	
	\noindent Let $d\in \{2,3\}$ and assume that $\MO \subset \mathbb{R}^d$ is a bounded domain with boundary $\partial \MO$ of class $\mathcal{C}^\infty$.
	For $p\in [1,\infty)$ and $k\in \mathbb{N}$ the symbols $\el^p(\MO)$ or by $\mathbb{W}^{k,p}(\MO)$ (resp. by $\elb^p(\MO)$   or by $\mathbf{W}^{k,p}(\MO)$) respectively denote the Lebesgue and Sobolev spaces of functions $\bv:\mathbb{R}^d\to \mathbb{R}^d$ (resp. $\bd:\mathbb{R}^d \to \mathbb{R}^3$)  
	For $p=2$ the function spaces $\mathbb{W}^{k,2}(\MO)$ and  $\mathbf{W}^{k,p}$  are respectively denoted by $\h^k$ ($\bH^k$) and their norms are 
	denoted by $\lve \cdot
	\rve_k$. The scalar products on $\el^2$ and $\elb^2$ are denoted by the same symbol $\langle
	u,v\rangle$ for $u,v\in \el^2$ (resp. $u,v\in \elb^2$ ) and their associated norms is denoted by $\lVert
	u\rVert$, $u\in \el^2$ (resp. $u\in \el^2$).
	By $\h^1_0$ and $\mathrm{W}^{1,r}_0$, $r>2$,   we mean the spaces of functions in $\h^1$ and $\mathrm{W}^{1,r}$
	that vanish on the boundary on $\MO$.  
	 It is well known that if $a=\frac{d}4$, then there exists a constants $c>0$ such that 
	\begin{equation} \label{GAG-l4}
	\lvert \bu\rvert_{\el^4}\le c \begin{cases}
 \lvert \bu\rvert_{\el^2}^{1-a} \lvert \nabla \bu
\rvert^a_{\el^2} 	\text{ if } \bu\in \mathbb{H}^{1}_{0}\\
\lvert \bu\rvert_{\el^2}^{1-a} \lvert  \bu
\rvert^a_{\mathbb{H}^1} 	\text{ if } \bu\in \mathbb{H}^{1},
	\end{cases}
	\end{equation}
		\begin{equation} \label{GAG-LInf}
	\lvert \bu\rvert_{\el^\infty}\le c \begin{cases}
	\lvert \bu\rvert_{\el^4}^{1-a} \lvert \nabla \bu
	\rvert^a_{\el^4} 	\text{ if } \bu\in \mathbb{W}^{1,4}_{0}\\
	\lvert \bu\rvert_{\el^4}^{1-a} \lvert  \bu
	\rvert^a_{\mathbb{W}^{1,4}} 	\text{ if } \bu\in \mathbb{W}^{1,4},
	\end{cases}
	\end{equation}
	and since $\mathrm{H}^{2} \hookrightarrow \mathrm{W}^{1,4}$, we have 
	\begin{equation}
	\lvert \bu \rvert_{\el^\infty}\le \lve \bu \rve_1 ^{1-a}\lve
	\bu \rve_{2}^a, \; \bu\in \h^2  .\label{GAG-LInf-2}
	\end{equation}
	
	We now introduce the following spaces
	\begin{align*}
	\mathcal{V}& =\left\{ \bu\in \mathcal{C}_{c}^{\infty }(\MO,\mathbb{R}^d)\,\,\text{such that}%
	\,\,\Div \bu=0\right\} \\
	\mathrm{V}& =\,\,\text{closure of $\mathcal{V}$ in }\,\,\mathrm{H}_0^{1}(\MO) \\
	\mathrm{H}& =\,\,\text{closure of $\mathcal{V}$ in
	}\,\,\mathrm{L}^{2}(\MO).
	\end{align*}
	As usual, we endow $\mathrm{H}$ with the scalar product and norm of $\el^2$, and  we equip the space $\ve$ with he scalar product
	$\langle \nabla \bu, \nabla \bv\rangle$ which is equivalent to the
	$\h^1(\MO)$-scalar product on $\ve$.
	
	Let $\Pi: \el^2 \rightarrow \h$ be the Helmholtz-Leray projection
	from $\el^2$ onto $\h$. We denote by $\rA=-\Pi\Delta$ the
	Stokes operator with domain $D(\rA)=\mathrm{V}\cap \h^2$, see for instance \cite[Chapter I, Section 2.6]{Temam} .  It is well-known that the spaces $\ve_\beta:=D(\rA^\beta)$, $\beta \in \mathbb{R}$, are Hilbert spaces   when endowed with the graph inner product and $\ve_\frac12 = \ve$.  It is also well-known that the map $\rA^\delta: \ve_\beta \to \ve_{\beta-\delta}$, $\beta,\delta \in \mathbb{R}$,  is a linear isomorphism. For all these facts we refer, for instance, to \cite{PC+CF}.

 The Neumann Laplacian acting on $\mathbb{R}^d$-valued
	function will be  denoted by $\rrA$, that is,
	\begin{equation*}
	\begin{split}
	D(\rrA)&:=\biggl\{\bu\in \bH^2: \frac{\partial \bu}{\partial \bn}=0 \text{ on } \partial \MO\biggr\},\\
	\rrA\bu&:=-\sum_{i=1}^d \frac{\partial^2 \bu}{\partial
		x_i^2},\;\;\; \bu \in D(\rrA).
	\end{split}
	\end{equation*}
	Notice that the Neumann Laplacian $\rrA$  can be viewed as a linear map $\rrA:\bH^1 \to (\bH^1)^\ast$  satisfying 	
	\begin{equation}\label{Eq:WeakNeumannLaplace}
	{}_{(\bH^1)^\ast}\langle \rrA \bu, \bd\rangle_{\bH^1}=\langle \nabla \bu, \nabla \bd \rangle,\text{ for all } \bu,\bd \in \bH^1.
	\end{equation}
	Thanks to \cite[Theorem 5.31]{Haroske} one can define and characterize in standard way the spaces $\bX_\alpha=D(\hrrA^{\alpha})$, $\alpha \in [0,\infty)$, where $\hrrA=I+\rrA$.  Also, it can be shown that  $\bX_\alpha \hookrightarrow \bH^{2\alpha}$, for all $\alpha \geq 0$  and
	$\bX:=\bX_{\frac12}=\bH^1$,  see, for instance, \cite[Sections 4.3.3 \& 4.9.2]{Triebel}.

	Now, let  $\bh\in\elb^\infty$  be fixed and define a linear bounded operator $G$ from $\elb^2$ into itself by
\begin{equation}\label{Eq:LinG}
	 G: \elb^2 \ni \bd \mapsto \bd\times \bh\in \elb^2.
\end{equation}It is straightforward to check that there exists a constant $C>0$ such that
	$$ \lve G(\bd)\rve \le C \lve \bh\rve_{\elb^\infty} \lvert \bd \rvert_{\elb^2}, \text{ for all } \bd \in \elb^2.$$
	
	\subsection{The nonlinear terms}
	
	Throughout this paper $\mathbf{B}^\ast$ denotes the dual space of
	a Banach space $\mathbf{B}$. We also denote by $\langle \Psi,
	\mathbf{b}\rangle_{\mathbf{B}^\ast,\mathbf{B}}$ the value of $\Psi\in \mathbf{B}^\ast$ on
	$\mathbf{b}\in \mathbf{B}$. Throughout,  $\partial_{x_i}=\frac{\partial}{\partial x_i} $ and {$\phi^{(i)}$}
	is the $i$-th entry of any vector-valued $\phi$.
	
	Let $p,q,r \in [1,\infty]$ such that $\frac 1p +\frac 1q +\frac 1r\le 1.$ Then,  we define a trilinear form $b(\cdot, \cdot,\cdot)$ by
	\begin{equation*}
	b(\bu,\bv,\w)=\sum_{i,j=1}^d \int_\mathcal{O}\bu^{(i)}\frac{\partial
		\bv^{(j)}}{\partial x_i}\w^{(j)} dx,\,\, \bu\in \el^p ,\bv\in
	\mathbb{W}^{1,q}, \text{ and } \w\in \el^r,
	\end{equation*}
	Note if  $\bv\in \mathbf{W}^{1,q}$ and $\w \in \elb^r$, then we have to take the sum over $j$ from $j=1$ to $j=3$.
	
	It is well known, see \cite[Section II.1.2]{Temam}, that there is a bilinear  map $B: \ve \times \ve \to \ve^\ast$ such that
	\begin{equation}\label{DEF-B1}
	\langle
	B(\bu,\bv),\w\rangle_{\ve^\ast,\ve}=b(\bu,\bv,\w)\text{ for }  \w\in \ve,\text{ and } \bu, \bv\in \ve.
	\end{equation}
	In a similar way,  there is also a bilinear map$\tilde{B}:\ve \times \bH^1\to (\bH^1)^\ast$
	such that
	\begin{equation}\label{DEF-B2}
	\langle
	\tilde{B}(\bu,\bv),\w\rangle_{(\bH^1)^\ast,\bH^1}=b(\bu,\bv,\w)\,\, \text{ for all }
	\bu \in \ve, \,\, \bv, \,\,\w\in \bH^1.
	\end{equation}
	The following lemma was proved in \cite[Section II.1.2]{Temam}. 
	\begin{lem}\label{LEM-B}
		The bilinear map $B(\cdot, \cdot)$ maps continuously  $\ve\times
		\h^1$ into $\ve^\ast$ and
		\begin{align}
		&\langle B(\bu,\bv),\w\rangle_{\ve^\ast,\ve}=b(\bu,\bv,\w), \text{ for all }
		\bu\in \ve, \bv\in \h^1,\w \in \ve,\label{B1}\\
		&\langle B(\bu,\bv),\w\rangle_{\ve^\ast,\ve}=-b(\bu,\w,\bv)
		\text{ for all } \bu\in \ve, \bv\in \h^1,\w \in \ve,\label{B2}\\
		&\langle B(\bu,\bv),\bv\rangle_{\ve^\ast,\ve}=0 \,\, \text{ for all } \bu\in \ve,
		\bv\in \ve,\label{B3}\\
		&\lvert B(\bu,\bv)\rvert_{\ve^\ast}\le C_0 \lvert \bu\rvert_{\el^2}^{1-\frac d4}
		\lvert\nabla \bu \rvert_{\el^2}^{\frac d4} \lvert \bv\rvert_{\el^2}^{1-\frac d4}\rvert_{\el^2}
		\nabla \bv\rvert_{\el^2}^\frac d4, \text{ for all } \bu\in \ve, \bv\in
		\h^1.\label{B4}
		\end{align}
	\end{lem}
	\begin{lem}\label{LEM-G1}
	There exists a constant $C_1>0$
		such that
		\begin{align}
	&	\lve \tilde{B} (\bv, \bd)\lve \le C_1  \lve \bv\rve_{\el^2}^{1-\frac
			d4}\rve \nabla \bv\rve_{\el^2}^\frac d4 \lve \bd\rve_{\mathbf{H}^1}^{1-\frac d4}
		\lve \bd \rve_{\mathbf{H}^2}^{\frac d4}, \text{ for all } \bv \in \ve,
		\bd\in \h^2,\label{EST-G1}\\
	&	\langle \tilde{B}(\bv,\bd),\bd\rangle=0,\text{
			for
			any } \bv\in \ve, \bd \in \h^2.\label{tild-b-0}
		\end{align}
	\end{lem}
	\begin{proof}
	The estimate \eqref{EST-G1} follows from \cite[Lemma 6.2]{BHP18}  and the Gagliardo-Nirenberg estimate \eqref{GAG-l4}. The proof of \eqref{tild-b-0} and \eqref{B3} are the same, see \cite[Section II.1.2]{Temam}. 
	\end{proof}

	Let $r,\,p,\, q\in(1,\infty)$ such that $ \frac 1p
	+\frac 1q+\frac 1r\le 1$. For \text{ $\d_1\in \bW^{1,p}$, $\d_2\in \bW^{1,q}$ and $\bu\in
		\bw^{1,r}$ } we set 
	\begin{equation}\label{INT-md}
	\mathfrak{m}(\d_1, \d_2,\bu)= -\sum_{i,j=1}^d\sum_{k=1}^3 \int_\MO
	\partial_{x_i}\bd_1^{(k)} \partial_{x_j}\bd_2^{(k)} \partial_{x_j}\bu^{(i)}\,dx. 
	\end{equation}
	Since $d\le 4$,
	the integral in \eqref{INT-md} is also well defined for $\bd_1,
	\d_2 \in \bH^2$ and $\bu\in \ve$.

	We recall the following proposition which can be found in \cite[Proposition 2.2 \& Remark 2.3]{BHP18}.
	\begin{prop}\label{LEM-M}
		Let $d\in [1,4]$. There exists a bilinear map $M:\bH^2\times \bH^2 \to \ve^\ast $ such that
		\begin{align}\label{def-Md}
		&\langle M(\bd_1,\d_2), \bu\rangle_{\ve^\ast,\ve}= \mathfrak{m}(\d_1,\d_2,\bu), \;\text{$\d_1, \, \d_2 \in \bH^2$, } \bu \in \ve,\\	
		&	\langle M(\mathbf{f}, \mathbf{g}), \bv \rangle_{\ve^\ast, \ve}= \langle \Pi[\Div (\nabla \mathbf{f}
		\odot \nabla \mathbf{g})], \bv \rangle \text{ for all } \mathbf{f}, \mathbf{g}\in \bX_{1} \text{ and } \bv \in \mathrm{H},\label{Eq:Identity-M-L2}\\
		&	\langle \tilde{B}(\bv,\bd), \rrA \bd\rangle+\langle
		M(\bd,\d), \bv\rangle_{\ve^\ast,\ve}=0, \text{ for all } \bv\in \ve, \bd \in \bX_1,\label{G1-eq-Md}.
		\end{align}
	\end{prop}
	In some places in this manuscript we use the following shorthand notations:
	$$ B(\bu):= B(\bu,\bu) \text{ and } M(\bd):= M(\bd,\bd),$$ for all $\bu$ and $\bd$ such that the above quantities are meaningful.

	We now fix the standing assumptions on the function $f(\cdot)$.
	\begin{assum}\label{eqn-f}
		Let $I_d$ be the set defined by
		\begin{equation}\label{DEG-POL}
		I_d=\begin{cases}
		\mathbb{N}:=\{1, 2, 3, \ldots\} \text{ if } d=2,\\
		\{1\}, \text{ if } d=3.
		\end{cases}
		\end{equation}
		We fix $N\in I_d
		$  and  $a_k\in \mathbb{R}$, $k=0,\ldots, N$,
		with $a_N<0$.  We define a function $\tilde{f}:[0,\infty) \rightarrow \mathbb{R}$ by
		$$ \tilde{f}(r)=\sum_{k=0}^N a_k r^k, \text{ for all } r\in \err_+.$$
		We define a map $f:\mathbb{R}^3\rightarrow \mathbb{R}^3$ by $f(\d)=\tilde{f}(\vert \d\vert^2)\d$ where $\tilde{f}$ is as above.
		
		We now assume that there exists $F: \mathbb{R}^3 \rightarrow \mathbb{R}$  a Fr\'echet differentiable map such that  $$ F^\prime(\d)[\mathbf{g}]= f(\d)\cdot \mathbf{g},
	\text{ $\d\in \mathbb{R}^d$, $\mathbf{g}\in \mathbb{R}^d$}.$$ Note that if  $\tilde{F}$ is a map such that $\tilde{F}^\prime= \tilde{f}$ and  $\tilde{F}(0)=0$, then, there  is $U$ is a polynomial function  with $\text{deg}(U)\le N$  and $a_{N+1}<0$ such that 
	$ \tilde{F}(r)=a_{N+1}r^{N+1}+U(r).$
	\end{assum}

	\begin{Rem}\label{REM-H2}
		\begin{enumerate}[(i)]
		\item \label{Itemi:REM-H1 } There exists a constant $\ell_3>0$ such that
			\begin{align}
			\lvert \tilde{f}^{\prime\prime}(r) \rvert\le \ell_3(1+r^{N-2}), \,\, r>0.\label{ST6-B-2}
			\end{align}
		\item \label{Itemi:REM-H2}  From \eqref{ST6-B-2}, we infer that there exist $c_0,c_1,c_3>0$ such that
			$$ \lvert f(\bd)\rvert \le c_0 (1+\vert \d\vert^{2N+1}), \text{    } \lvert f^{\prime}(\d)\rvert\le
			c_1(1+\vert \bd \vert^{2N}) \text{ and }    \lvert f^{\prime \prime}(\d)\rvert\le c_2 (1+\vert \bd \vert^{2N-1}) \text{ for all } \bd\in \err^n.$$
			\item  Let $\tilde{q}=4N+2$. It is easy to show that there exists $C>0$ such that \mbox{   for all } $\d\in \h^2$
			\begin{align}
			\lve \rrA \d\rve^2=&\lve \rrA \d +f(\d)-f(\d)\rve^2
			\le  2 \lve \rrA \d -f(\d)\rve^2+2 \lve f(\d)\rve^2,\nonumber \\
			\le & 2 \lve \rrA \d -f(\d)\rve^2+C \lve \d\rve^{\tilde{q}}_{\mathbf{L}^{\tilde{q}}}+C.\label{bigdandel}
			\end{align}

			\item Since the norm $\lVert \cdot \rVert_2$ is equivalent to $\lVert \cdot \rVert + \lVert \rrA \cdot \rVert $ on $D(\rrA)$,  there exists $C>0$
			such that
			\begin{equation}\label{bigdanh2}
			\lve \d\rve^2_2\le C (\lve \rrA \d -f(\d)\rve^2
			+\lve \d\rve^{\tilde{q}}_{\mathbf{L}^{\tilde{q}}}+1), \mbox{   for all }\d\in D(\rrA).
			\end{equation}
			
			\item\label{Rem-Linf-H1Delta}
			Since $\bH^1\hookrightarrow \elb^{4N+2}$,  $N \in I_d$, we infer from \eqref{bigdanh2} that  $\bd\in \bH^2$ if $\bd \in \bH^1$ and $\rrA \bd -f(\bd)\in \elb^2$.
		\end{enumerate}
	\end{Rem}
		\begin{Rem}\label{Rem:Ginz-Land-General}
		Let $\eps>0$ and $\tilde{f}_\eps(r):=\frac{1}{\eps^2}(-r+1)$, $r\in [0,\infty) $. Examples of maps $f$ and $F$ satisfying Assumption \ref{eqn-f} are the following
		$$f(\d):=\tilde{f}(\vert \d \vert^2)\d=\frac{1}{\eps^2}(1-\lvert \d \rvert^2)\d \text{ and } F(\d):=\frac1{4\eps^2} [\tilde{f}(\vert \d \vert^2) ]^2, \; \d\in \mathbb{R}^d.$$
	\end{Rem}
	\subsection{The assumption on the coefficients of the noise}
	
	\begin{assum}\label{Assum:Usual hypotheses}
		Throughout this paper we are given a complete filtered probability space $(\Omega,
		\mathcal{F}, \mathbb{P})$
		with the filtration $\mathbb{F}=\{\mathcal{F}_t: t\geq 0\}$
		satisfying the usual hypothesis, \textit{i.e.},
		the filtration is right-continuous and all null sets of $\mathcal{F}$ are elements of $\mathcal{F}_0$.
	\end{assum}
	Throughout, let $\rK_1$ be a separable Hilbert space, and $W_1=(W_1(t))_{t\geq 0}$ and 
	$W_2=(W_2(t))_{t\geq 0}$ be  independent $\rK_1$-cylindrical
	Wiener process and  standard Brownian motion on $(\Omega,
	\mathcal{F},\mathbb{F}, \mathbb{P})$.  If $\rK=\rK_1\times \mathbb{R}$ then  we can assume that $W=(W_1(t),W_2(t))$ is $\rK$-cylindrical
	Wiener process. 
%	\begin{Rem}\label{Wiener}
%		%We assume that $W_1$ is a cylindrical   Wiener process on separable Hilbert space  $\mathbb{K}_1$.
%		If $\ku_2$ is a Hilbert space such that the embedding $\ku_1\subset \ku_2$ is Hilbert-Schmidt, then $W_1$ can be viewed as a $\ku_2$-valued Wiener process. Moreover, there exists a trace class symmetric nonnegative operator $Q\in \mathcal{L}(\ku_2)$ such that $W_1$ has  covariance $Q$. This $\ku_2$-valued $\ku_1$-cylindrical  Wiener process is characterised by, for all $t\geq 0$,
%		\[
%		\mathbb{E} e^{i \;\fourIdx{}{}{}{\ku_2^\ast,\ku_2} {\lb x^\ast, W(t)\rb}  }= e^{-\frac{t}2 \vert x^\ast\vert_{\ku_1}^2}, \;\; x^\ast \in \ku_2^\ast,
%		\]
%		where $\ku_2^\ast$ is the dual space to $\ku_2$ such that identifying $\ku_1^\ast$ with $\ku_1$ we have
%		\[
%		\ku_2^\ast \embed \ku_1^\ast=\ku_1 \embed \ku_2.\]
%	\end{Rem}
	
	%	Let $\mathrm{K}$ be a separable Banach space, $\mathscr{L}(\mathscr{H}, \mathrm{K})$ be the space of all bounded linear $\mathrm{K}$-valued  operators defined on $\mathscr{H}$,
	Let $\tilde{\mathrm{K}}$ and $\tilde{\mathrm{H}}$ be a separable Hilbert and Banach spaces. We denote by $\gamma(\tilde{\rK}, \tilde{\rH})$ the space of $\gamma$-radonifying operators which generalises the space of Hilbert-Schmidt operators $\mathcal{T}_2(\tilde{\rK}, \tilde{\rH})$ if $\tilde{\rH}$ is a separable Hilbert space, see \cite{ZB-97}.   Let $\mathscr{M}^2(\Omega\times [0,T]; \mathcal{T}_2(\tilde{\mathrm{K}}, \tilde{\mathrm{H}} ))$ the space of all equivalence classes of progressively measurable processes $\Psi: \Omega\times [0,T]\to\mathcal{T}_2(\tilde{\mathrm{K}}, \tilde{\mathrm{H}} )$ satisfying
	$$ \E\int_0^T \Vert \Psi(s)\Vert^2_{\mathcal{T}_2(\tilde{\mathrm{K}}, \tilde{\mathrm{H}} )}ds <\infty.$$
		For a $\tilde{\rK}$-cylindrical Wiener process $\tilde{W}$ and   $\Psi \in\mathscr{M}^2(\Omega\times [0,T]; \mathcal{T}_2(\tilde{\mathrm{K}}, \tilde{\mathrm{H}} ))$ the process $M$ defined by
	$ M(t) =\int_0^t \Psi(s)d\tilde{W}(s), t\in [0,T],$ is a $\tilde{\mathrm{H}}$-valued martingale. 
	For more detail on the theory of stochastic integration we refer to \cite[Section 26 ]{Metivier_1982} and \cite[Chapter 4]{DP+JZ-14}. 
%	for all $\Psi \in\mathscr{M}^2(\Omega\times [0,T]; \mathcal{T}_2(\tilde{\mathrm{K}}, \tilde{\mathrm{H}} ))$ the process $M$ defined by
%	$$ M(t) =\int_0^t \Psi(s)dW(s), t\in [0,T],$$ is a $\tilde{\mathrm{H}}$-valued martingale. Moreover,  we have the following It\^o isometry
%	\begin{equation}
%	\E \biggl(\biggl\Vert \int_0^t \Psi(s) dW(s)\biggr\Vert^2_{\tilde{\mathrm{H}} }\biggr) =\E \biggl(\int_0^t \Vert \Psi(s) \Vert^2_{\mathcal{T}_2(\tilde{\mathrm{K}}, \tilde{\mathrm{H}})} ds \biggr), \forall t\in [0,T],
%	\end{equation}	
%	and the Burkholder-Davis-Gundy inequality
%	\begin{equation}
%	\E\biggl(\sup_{0\le s\le t}\biggl \Vert \int_0^s \Psi(s) dW(s)\biggr\Vert^q_{\tilde{\mathrm{H}} }\biggr) \le C_q \E \biggl( \int_0^t \Vert \Psi(s)  \Vert^2_{\mathcal{T}_2(\tilde{\mathrm{H}},\tilde{\mathrm{H}})}ds \biggr)^\frac q2, \forall t\in [0,T], \forall q\in (1,\infty).
%	\end{equation}

Let $G$ be the map defined in \eqref{Eq:LinG} and $G^2=G\circ G$. Then, we have the following relation identity
	Stratonovich and It\^o's integrals, see \cite{Brz+Elw_2000},
	\begin{equation*}
	G(\bd)\circ dW_2= \frac 12 G^2(\bd) \,dt +
	G(\bd)\,dW_2.
	\end{equation*}

	We now introduce the standing set of hypotheses on the function $S$.
	
	\begin{assum}\label{HYPO-ST}
		We assume that $S: \h\to \mathcal{T}_2(\rK_1,\ve)$ is a globally Lipschitz map. In particular,
		there exists $\ell_5\geq 0$
		such that
		\begin{equation}\label{Eq:Hypo-ST}
		\lve  S(\bu)\rve^2_{\mathcal{T}_2(\rK_1,\ve)}\leq \ell_5 (1+\lvert \bu \rvert_{\el^2}^2),\;\; \mbox{ for all } \bu \in \h.
		\end{equation}
	\end{assum}
	\begin{Rem}\label{Rem:HYPO-ST}
		Notice that the assumption \eqref{Eq:Hypo-ST} implies that there exists a constant $\ell_6>0$ such that
		\begin{equation}\label{Eq:Hypo-ST-Rem}
		\lve  S(\bu)\rve^2_{\mathcal{T}_2(\rK_1,\ve)}\leq \ell_5 (1+\lvert \nabla \bu \rvert_{\el^2}^2),\;\; \mbox{ for all } \bu \in \h.
		\end{equation}
	\end{Rem}
	\section{Existence and uniqueness of local and global strong Solution}\label{SLC-Sect4}
Using the notations of Section \ref{sec-spaces-2}, the system \eqref{eqn-SLQE-v}-\eqref{Init-Cond}  can be  written in  the abstract
	form
	\begin{align}
	&d\bv(t)+\biggl(\rA\bv(t)+B(\bv(t),
	\bv(t))+M(\bd(t))\biggr)dt=S(\bv(t))dW_1,\label{ABS-v1}\\
	&d\bd(t)+\biggl(\rA_1\bd(t)+ \tilde{B}(\bv(t),\bd(t))-
	f(\bd(t))-\frac 12 G^2(\bd(t))\biggr)dt=G(\bd(t))dW_2,\\
	\label{ABt-d1}
	& \v(0)=v_0
	\text{ and } \d(0)=d_0.
	\end{align}
In this section we  prove the existence and uniqueness of  the strong solution to problem \eqref{ABS-v1}-\eqref{ABt-d1}. 
	%%%%%%%%%%%%%%%%%%%%
%	\red{One of the main results, the existence of a maximal local solution, of this section is a corollary of a general framework that will be introduced in Section \ref{ABST-STRONG}.
%	In this section we will check that the system \eqref{ABS-v1}-\eqref{ABt-d1} fits into this general framework and hence establish the existence of maximal local
%	solution of the stochastic nematic liquid crystal. We also prove that the maximal solution turns out to be a global one in the two dimensional case.
%	For these ends we will introduce some additional notations and prove several key inequalities related to the nonlinear terms of the stochastic system
%	\eqref{ABS-v1}-\eqref{ABt-d1}.}
	\subsection{Definition of local solutions}
	Let  $(B_i, \lve \cdot \rve_{B_i})$, $i=1, 2$, be two Banach spaces. We endow $B_1\times B_2$ with the norm $
	\rve(b_1,b_2)\lve=\sqrt{\lve b_1\rve_{B_1}^2 +\lve b_2\rve_{{B}_2}^2}.$
	Henceforth, we put
	\begin{equation}\label{eqn-spaces}
	\mathscr{H}=\h\times \bX_{\frac12}, \;\; \mathscr{V}=\ve\times \bX_{1} \mbox{ and } \mathscr{E}=\ve_1 \times \bX_{\frac32}.
	\end{equation}
	Next,  we denote by $\{\mathbb{S}_1(t)\}_{t\geq 0}$ and $\{\mathbb{T}(t)\}_{t\geq 0}$  the analytic semigroups 
	generated by
	$-\rA$ on $\h$ and by  $\rrA$ on $\elb^2$, respectively. 
It is well-known that 
	the space $\bX_{\frac12}$ is invariant wrt $\{\mathbb{T}(t)\}_{t\geq 0}$.  The restriction of $\{\mathbb{T}(t)\}_{t\geq 0}$ to $\bX_{
		\frac12}$ is also an analytic semigroup which will be denoted by $\{\mathbb{S}_2(t)\}_{t\geq 0}$.
	The minus infinitesimal
	generator $\trrA$ of $\{\mathbb{S}_2(t)\}_{t\geq 0}$ is the part
	of $\rrA$ on $\bX_{\frac12}$, that is,
	\begin{align*}
	D(\trrA)=\{\bu \in D(\rrA): \rrA\bu\in \bX_{\frac12}\},\;\;
	\trrA\bu=\rrA\bu \text{ for all } \bu \in D(\trrA).
	\end{align*}
	Note that $\bX_{\frac32}\subset D(\trrA).$
	With all the above notation, the problem \eqref{ABS-v1}-\eqref{ABt-d1} can be rewritten as the following stochastic evolution equation in the space $\mathscr{H}$,
	\begin{equation}\label{ABSTRACT-LC}
	d\y(t) +\mathbf{A}\y(t) dt+\mathbf{F}(\y(t))
	dt+\mathbf{L}(\y(t))dt=\mathbf{G}(\y(t)) d{W}(t),
	\end{equation}
	where, for  $\y=(\bv, \d)\in E$ and $k=(k_1,k_2)\in \rK$,
	
	\begin{equation}\label{eqn-def-A-F}
	\A\y=\begin{pmatrix} \rA \bv  \\
	 \rrA \bd
	\end{pmatrix},\;\;
	\f(\y)=\begin{pmatrix} B(\bv,\bv)+M(\bd)\\
	\tilde{B}(\bv,\bd)-f(\bd)
	\end{pmatrix},
	\end{equation}
	\begin{equation}\label{eqn-def-L-G}
	\mathbf{L}(\y)=\begin{pmatrix} 0\\ -\frac 1 2 G^2(\bd)
	\end{pmatrix}, \mathbf{G}(\y)k=\begin{pmatrix}S(\bu)k_1\\
	G(\bd)k_2 \end{pmatrix}.
	\end{equation}
	The operator $-\mathbf{A}$ generates an  analytic semigroup  $\{\mathbb{S}(t)\}_{t\geq 0}$ on  $\mathscr{H}=\h\times \bX_{\frac12}$
	defined by
	\begin{equation*}
	\mathbb{S}(t)\begin{pmatrix}\bv\\\d
	\end{pmatrix}=\begin{pmatrix}\mathbb{S}_1(t)\bv\\
	\mathbb{S}_2(t)\d
	\end{pmatrix}, \;\; (\bv,\d)\in \mathscr{H}.
	\end{equation*}
	%Some properties of  will be given in Lemmata \ref{SEM-1}-\ref{SEM-3}.
	%%%%%%%%%%%%%%%%%%%%%%%%%%%%%%%%%%%%%%%%%%%%%%
	Important properties of $\{\mathbb{S}(t): t\ge 0\}$ are given in the next two lemma.
	\begin{lem}\label{SEM-1} \label{SEM-2}
		Let $T\in (0,\infty)$, $\mathbf{g}=\begin{pmatrix}
		\tilde{g} \\ g
		\end{pmatrix}\in L^2(0,T;  \h \times \bX_{\frac12})$ and
		$
\begin{pmatrix}
	\mathbf{v}(t)\\
	\mathbf{n}(t)
	\end{pmatrix}= \int_0^t\mathbb{S}(t-s) \mathbf{g}(s)\, ds, \;\; t\ge 0.$  Then, there exists  $c_1>0$ such that
		\begin{equation*}
		\left \lve \begin{pmatrix}
		\mathbf{v}\\\mathbf{n}
		\end{pmatrix} \right\rve_{C([0,T];\ve \times \bX_{1})}+\left \lve \begin{pmatrix}
		\mathbf{v}\\\mathbf{n}
		\end{pmatrix} \right\rve_{L^2(0,T;D(\rA) \times \bX_{\frac32}) }\le c_1 \left \lve \begin{pmatrix}
		\tilde{g} \\ g
		\end{pmatrix}\right \rve_{L^2(0,T;
			\h\times \bX_{\frac12})}.
		\end{equation*}
	\end{lem}
	\begin{proof}
		This result is well-known and is a special case of \cite[Lemma 1.2]{Pardoux}.
	\end{proof}
%	Similarly we have the following result.
%	\begin{lem}
%		Let $T\in (0,\infty)$, $\tilde{g}\in L^2(0,T; \h)$ and
%		\begin{align*}
%		\bv(t)=\int_0^T \mathbb{S}_1(t-s) \tilde{g}(s)\, ds.
%		\end{align*}
%		We have
%		\begin{equation*}
%		\lve \bv\rve_{C([0,T]; D(\rA^\frac12)}+\lve \bv\rve_{L^2(0,T;D(\rA))
%		}\le (1+\frac 1 \mu_1) \lve \tilde{g}\rve_{L^2(0,T;\h) },
%		\end{equation*}
%		where $\mu_1$ is the smallest eigenvalues of the Stokes operator
%		$\rA$.
%	\end{lem}
	\begin{lem}\label{SEM-3}
				Let $T\in (0,\infty)$, $\zeta= \begin{pmatrix}
				\zeta_1\\ \zeta_2
				\end{pmatrix}\in \mathscr{M}^2(0,T;\h\times \bX_1) $ and 
				$ \begin{pmatrix}
				\w_1(t)\\ \w_2(t)
				\end{pmatrix} = \int_0^t \mathbb{S}(t-s) \zeta(s) dW(s), t \ge 0. $ Then, 
		there exists $C>0$ such that 
		\begin{equation*} \mathbb{E} \left \lve
	\begin{pmatrix}
	\w_1 \\ \w_2
	\end{pmatrix}  \right \rve^2_{C([0,T];\ve \times \bX_1 )}+\mathbb{E}\left \lve
		\begin{pmatrix}
		\w_1 \\ \w_2
		\end{pmatrix}  \right \rve^2_{L^2(0,T;D(\rA) \times \bX_{\frac32})}\le C \mathbb{E}\left \lve
	\begin{pmatrix}
	\zeta_1\\ \zeta_2
	\end{pmatrix}\right \rve^2_{L^2(0,T;\ve \times \bX_1)},\;\; T\ge 0.
		\end{equation*}
	\end{lem}
	\begin{proof}
		This result is also well-known, see \cite[Lemma 1.4]{Pardoux}.
	\end{proof}
	%%%%%%%%%%%%%%%%%%%%%%%%%%%%%%%%%%%%%%%%%%%%%%%%%
	Let us recall the following notations/definition which are borrowed
	from \cite{Kunita-90}, see also \cite{Brz+Elw_2000}.
	%%%%%%%%%%%%%%
	%%%%%%%%%%%%%%%%%%%%%%%%%%%%%%%%%%%%%
	%%%%%%%%%%%%%%%
	\begin{Def}\label{def-accessible stopping time}
 A random function $\tau:\Omega \to [0,\infty]$ is called a stopping time, see \cite[Definition I.2.1]{Kar-Shr-96}, \cite[Definition 4.1]{Metivier_1982} and \cite[section III.5]{Elw_1982}, iff for each $t\geq 0$, the set
$\{\omega \in \Omega: t< \tau(\omega)\} \in \mathcal{F}_t$ (or equivalently, $\{\omega \in \Omega: \tau(\omega) \leq t\} \in \mathcal{F}_t$).  A  stopping time
		$\tau:\Omega \to [0,\infty]$ is called accessible, see \cite[section 2.1, p. 45]{Kunita-90},  iff there exists an  increasing
		sequence\footnote{In the sense that for all $n\in \mathbb{N}$, $\tau_n \leq
		\tau_{n+1}$,  $\mathbb{P}$-a.s. } of stopping times  $\tau_n{:\Omega \to [0,\infty)}$ such that $\mathbb{P}$-a.s. 
	\begin{inparaenum}
		\item[(i)] for all $n\in \mathbb{N}$,  $\tau_n <
		\tau$;
		\item[(ii)] and $\lim_{n\to \infty} \tau_n =\tau$.
	\end{inparaenum}

The sequence   $(\tau_n)_{n\in \mathbb{N}}$   as above is usually called an announcing sequence for $\tau$.
	\end{Def}
%\begin{Rem}\label{rem-accessible stopping time}
%\begin{trivlist}
%\item[(i)] If $\tau$ and $\sigma$ accessible stopping times with announcing sequences $(\tau_n)_{n \in \mathbb{N}}$ and $(\sigma_n)_{n \in \mathbb{N}}$, then by \cite[Proposition 4.3]{Metivier_1982}, $\tau\vee \sigma$ is an accessible stopping time with announcing sequence $(\tau_n \vee \sigma_n)_{n \in \mathbb{N}}$.
%\end{trivlist}
%Under some additional assumptions,  this result is valid for countable suprema, see Remark \ref{rem-predictable stopping time}
%\end{Rem}

\begin{Rem}\label{rem-predictable stopping time}
Under the Assumption \ref{Assum:Usual hypotheses} we have the following facts.
\begin{trivlist}
	%	\begin{enumerate}
		\item[(i)]
%Because the filtration $\mathbb{F}=\{\mathcal{F}_t: t\geq 0\}$ is is right-continuous,
It follows from \cite[Proposition 6.6 (3)]{Metivier_1982} that a stopping time  is  accessible if and only if it is  predictable. Let us recall,  see \cite[Definition 4.9]{Metivier_1982}, that
a stopping time $\tau$ is predictable iff its  graph $[\tau]:=\{(t,\omega)\in [0,\infty)\times\Omega: t=\tau(\omega)\}$ is a predictable set.
The $\sigma$-field $\mathcal{P}$ of predictable sets is generated by the family $\mathcal{R}:=\{ (s,t]\times F: 0\leq s\leq t, F\in \mathcal{F}_s\}\cup \{\{0\} \times F: F\in \mathcal{F}_0\}$, see \cite[Theorem 3.3]{Metivier_1982}.
		\item[(ii)] If $(\tau_n)_{n \in \mathbb{N}}$ is a sequence of accessible stopping times,  then $\sup_{n\in \mathbb{N}} \tau_n$ is also an accessible stopping times, see \cite[Proposition 6.6]{Metivier_1982}. In particular,  If $\tau$ and $\sigma$ accessible stopping times with announcing sequences $(\tau_n)_{n \in \mathbb{N}}$ and $(\sigma_n)_{n \in \mathbb{N}}$, then  $\tau\vee \sigma$ is an accessible stopping time with announcing sequence $(\tau_n \vee \sigma_n)_{n \in \mathbb{N}}$. Furthermore, if $\mathcal{A}$ is an arbitrary  family of accessible stopping times then a family
		\[
		\mathcal{B}:=\{ \sup \mathcal{C}: \mathcal{C} \mbox{ is a finite subset of } \mathcal{A} \}
		\]
		is also a family of  accessible stopping times such that $\mathcal{A} \subset \mathcal{B}$ and the supremum of each finite subset of $\mathcal{B}$ belongs to $\mathcal{B}$. In particular, if $\Delta$
		is the family of all accessible stopping times, then the supremum of each finite subset of $\Delta$ belongs to $\Delta$.
%	\end{enumerate}
\end{trivlist}

\end{Rem}

	\textbf{Notation}.  For  a  stopping  time  $\tau$  we  set
	%\begin{eqnarray*}
	\[ \Omega_t(\tau) =%&=&
	\{ \omega \in \Omega : t < \tau(\omega)\},
	\]%\\
	\[
	[0,\tau)\times \Omega =%&=&
	\{ (t,\omega) \in [0,\infty)\times \Omega: 0\le t < \tau(\omega)
	\}.
	\]%\end{eqnarray*}
	
	\begin{Def}\label{def-adapted process}
 Assume that  $X$ is    a {topological space}.
		An $X$-valued process {local} process $\eta : [0,\tau) \times \Omega \to X$ (we will also write $ \eta(t)$, $t< \tau$)    is
		admissible iff
		\begin{inparaenum}
			\item[(i) ] it is adapted, i.e.  $\eta|_{\Omega_t(\tau)}: \Omega_t(\tau) \to
			X$ is $\mathcal{F}_t$ measurable, for all $t\ge 0$; \item[(ii)]
			for almost all $\omega \in \Omega$, the function $[0,
			\tau(\omega))\ni t \mapsto \eta(t, \omega) \in X$ is continuous.
			\end{inparaenum}
%		A {local} process $\eta : [0,\tau) \times \Omega \to X$ is  progressively
%		measurable  iff, for all $t> 0$,  the  map $$[0,t\wedge \tau)
%		{%\bar
%			\times}  \Omega \ni  (s,\omega) \mapsto  \eta(s,\omega) \in  X$$ is
%\hdoubtz{How do we define $\mathcal{B}_{t\wedge \tau}$?}		$\mathcal{B}_{t\wedge \tau} \times \mathcal{F}_{t\wedge \tau}$ measurable.\\
		%
		
		Two  {local} processes $\eta_i: [0,\tau_i)   \times \Omega  \to X$,
		$i=1,2$ are called equivalent (and we will write $(\eta_1,\tau_1) \sim
		(\eta_2,\tau_2)$)   iff $\tau_1=\tau_2$  $\mathbb{P}$-a.s. and,   for all  $t>0$,
		the  following condition holds
		\[
		\eta_1(\cdot,\omega)= \eta_2(\cdot,\omega) \mbox{ on } [0,t], ; \mbox{ 		for a.e. $\omega \in \Omega_t(\tau_1)\cap \Omega_t(\tau_2)$}.
		\]
		Note that if { two local} admissible processes  $\eta_i  :  [0,\tau_i)  \times  \Omega \to
		X$,  $i=1,2$ are such that  for all $t>0$
		$\eta_1(t)|_{\Omega_t(\tau_1)}= \eta_2(t)|_{\Omega_t(\tau_2)}$
		$\mathbb{P}$-a.s.,  then they are  equivalent.
	\end{Def}

\begin{Rem}\label{rem-adapted local process}
Let $\tau$ be an accessible stopping time with an announcing sequence $(\tau_n)_{n \in \mathbb{N}}$ and $\eta : [0,\tau) \times \Omega \to X$ is   a local process.
 Kunita \cite[section 2.1, p. 46]{Kunita-90} defined $\eta$ to be a  local adapted process iff
 the following  condition is satisfied for every $n$, the  stopped process $(\eta_{t\wedge \tau_n})_{t\geq 0}$ is adapted. We do not know how  condition (i) from Definition \ref{def-adapted process}  is related to Kunita's definition. 
\end{Rem}
	\begin{Def}\label{def-progrssively measurable process}
% Assume that  $X$ is    a {topological space}.
Let $\tau$ be an accessible stopping time with an announcing sequence $(\tau_n)_{n \in \mathbb{N}}$. Motivated by  \cite[Proposition 2.18]{Kar-Shr-96}, a local process $\eta : [0,\tau) \times \Omega \to X$  is  called progressively
measurable  iff   for every $n$, the  stopped process $(\eta_{t\wedge \tau_n})_{t\geq 0}$ is progressively measurable. 
\end{Def}
%\begin{Rem}\label{rem-progrssively measurable process} Definition \ref{def-progrssively measurable process} is motivated by Proposition 2.18 from \cite{Kar-Shr-96} according to which if $X=(X_t)_{t\geq 0}$ is a progressively measurable stopping time and $\tau$ is a finite stoping time, then the stopped process $(X_{t\wedge \tau})_{t\geq 0}$ is also  progressively measurable.
%\end{Rem}	
	We now define some concepts of solution to  \eqref{ABSTRACT-LC}, see \cite[Definition 4.2]{Brz+Millet_2012} or \cite[Definition
	2.1]{Mikulevicius}.
	\begin{Def}\label{def-local solution} Let $\y_0:\Omega \to V$ be $\mathcal{F}_0$-measurable random variable satisfying $\mathbb{E} \Vert \y_0\Vert^2_{\mathscr{V}}<\infty$. A local
		solution to problem  \eqref{ABSTRACT-LC}(with the initial time
		$0$) is a pair $(\y,\tau)$ such that
		\begin{enumerate}[(1)]
			\item $\tau$ is an accessible stopping time with an announcing sequence $(\tau_n)_{n\in \mathbb{N}}$,
			\item $\y: [0,\tau)\times \Omega \to V$ is an admissible process,
			\item for every $n\in \mathbb{N}$ and $t\in [0,\infty)$,
			we have
			\begin{align}
			\mathbb{E}\Big(  \sup_{s\in [0,t\wedge \tau_n]} \Vert
				\y(s)\Vert^2_{\mathscr{V}} +\int_0^{t\wedge \tau_n} \Vert \y(s)\Vert_\mathscr{E}^2 \,
				ds\Big)<\infty,\label{eq-locsol_01-a}
				\end{align}
				and $\mathbb{P}$-a.s.
				\begin{align}
				\y(t\wedge \tau_n)= \mathbb{S}(t\wedge
				\tau_n)\y_0-\int_0^{t\wedge \tau_n}
				\mathbb{S}(t\wedge\tau_n-s)[ \mathbf{F}(\y(s\cnew{\wedge \tau_n}))+\mathbf{L}(\y(s\cnew{\wedge \tau_n})] \,ds +I_{\tau_n}(t\wedge \tau_n), \label{eq-locsol_01-b}
			\end{align}
where $I_{\tau_n}$ is a continuous $V$-valued process process defined by
			\begin{equation}\label{eq-locsol_01-c}
			\begin{split}
I_{\tau_n}(t):= \int_0^t\mathds{1}_{[0,\tau_n)}(s)\mathbb{S}(t-s)\mathbf{G}(\y(s \wedge \tau_n))\,
			d{W}(s),\;\; t\in[0,\infty).
			\end{split}
			\end{equation}
		\end{enumerate}
		Along the lines of the paper \cite{Brz+Elw_2000}, we say that a
		local solution $\y(t)$, $t < \tau$ is  global iff
		$\tau=\infty$ $\mathbb{P}$-a.s.
	\end{Def}
Hereafter, we simply write local solution in place of local mild solution. 
\begin{Rem}\label{rem-stochastic interval}
\begin{trivlist}
\item[(i)]  Since $\tau_n$ is a stopping, the process $\mathds{1}_{[0,\tau_n)}(s)$, $s\in [0,\infty)$ is well-measurable, see \cite[Proposition 4.2]{Metivier_1982}. Therefore, since by \cite[Theorem 1.6]{Metivier_1982}, the $\sigma$-field of well measurable sets is smaller than the $\sigma$-field of progressively measurable sets, it follows
that the process $\mathds{1}_{[0,\tau_n)}(s)$, $s\in [0,\infty)$ is progressively measurable. In particular, the integrand in \eqref{eq-locsol_01-c} is  progressively measurable.
\item[(ii)] On the other hand, one could use in \eqref{eq-locsol_01-c} a process $\mathds{1}_{[0,\tau_n]}(s)$, $s\in [0,\infty)$, which according \cite[Proposition 4.4]{Metivier_1982}, is predictable. However, here we still stick to the processes as in (i).
\end{trivlist}
\end{Rem}
\begin{Rem}\label{rem-local solution} Suppose that $\tau:\Omega \to [0,\infty)$ is an accessible stopping time and $\y: [0,\tau]\times \Omega \to V$ is an admissible process such that
for every $t \in [0,\infty)$
			\begin{align}
		&	\mathbb{E}\Big(  \sup_{s\in [0,t\wedge \tau]} \Vert
				\y(s)\Vert^2_{\mathscr{V}} +\int_0^{t\wedge \tau} \Vert \y(s)\Vert_\mathscr{E}^2 \,
				ds\Big)<\infty,\label{eq-locsol_01-a2}\\ 
	&		\y(t\wedge \tau)= \mathbb{S}(t\wedge
			\tau)\y_0-\int_0^{t\wedge \tau}
			\mathbb{S}(t\wedge\tau-s)[ \mathbf{F}(\y(s\wedge \tau))+\mathbf{L}(\y(s \wedge \tau)]\, ds
 +I_{\tau}(t\wedge \tau), \;\;\;\mbox{ $\mathbb{P}$-a.s.,} \label{eq-locsol_01-d2}
			\end{align}
where $I_{\tau}$ is a continuous $V$-valued process process defined by
			\begin{equation}\label{eq-locsol_01-d}
			\begin{split}
I_{\tau}(t):= \int_0^t\mathds{1}_{[0,\tau)}(s)\mathbb{S}(t-s)\mathbf{G}(\y(s \wedge \tau))\,
			d{W}(s),\;\; t\in [0,\infty).
			\end{split}
			\end{equation}
Let us choose an announcing sequence $(\tau_n)_{n\in \mathbb{N}}$  for  $\tau$.  Then, by using \cite[Lemma A.1]{Brz+Masl+Seidler_2005}, we can show that for every $n$,  the conditions
\eqref{eq-locsol_01-a} and \eqref{eq-locsol_01-b} are satisfied with $I_{\tau_n}$ defined by \eqref{eq-locsol_01-c}. Therefore we infer that the restriction of the process
$\y$ to the open stochastic interval $[0,\tau)\times \Omega$ is a local  solution to problem \eqref{ABSTRACT-LC}.
\end{Rem}
	
	We now introduce the definition of a maximal local solution.
	
	\begin{Def}\label{Def-maxsol-0}
		Consider  a family  $\mathcal{ LS}$ of all local solution
		$(u,\tau)$ to  the problem \eqref{ABSTRACT-LC}. For two
		elements $(u,\tau), (v,\sigma) \in \mathcal{ LS} $ we write that
		$(u,\tau)\preceq (v,\sigma)$ iff $\tau \leq \sigma$ $\mathbb{P}$-a.s. and
		$v_{\vert [0,\tau)\times \Omega} \sim u$. Note that if
		$(u,\tau)\preceq (v,\sigma)$ and $(v,\sigma)\preceq (u,\tau)$,
		then  $(u,\tau)\sim (v,\sigma)$. We write $(u,\tau)\prec
		(v,\sigma)$ iff $(u,\tau)\preceq (v,\sigma)$ and $(u,\tau)\not\sim
		(v,\sigma)$. Then, the pair $(\mathcal{ LS},\preceq)$ is 
		partially ordered. 
		Each maximal element $(u,\tau)$ in the set $(\mathcal{
			LS},\preceq)$
		is called a maximal local solution to  the problem  \eqref{ABSTRACT-LC}. The existence of an upper bound of  every non-empty chain of  $(\mathcal{
			LS},\preceq)$ is justified by  Amalgamation
		Lemma \ref{lem-amalgamation}. \\
		If $(u, \tau)$ is a  maximal local solution to equation
	\eqref{ABSTRACT-LC}, the stopping time $\tau$ is called its
		lifetime.
	\end{Def}
%
%	A priori, there may be many maximal elements in $(\mathcal{
%		LS},\preceq)$ and hence many maximal local solutions    to
%	the problem \eqref{ABSTRACT-LC} . However, as we will see in Section \ref{ABST-STRONG}, if  the uniqueness of local solutions holds,  then the
%	uniqueness of the maximal local solution will follow.
	\subsection{Existence and uniqueness of a maximal local solution: 2D and 3D cases}
	By using Theorem \ref{thm_local} we will prove in this subsection that the problem \eqref{ABSTRACT-LC} has a unique maximal local solution. In order to do this, we need to  establish several auxiliary results. Throughout this subsection $d=2,3$ and $a=\frac{d}4$.
	\begin{lem}\label{Local-LIP-Lem}
		There exists $c_2>0$ such that for
		all $\bd_i\in \bH^3$, i$=1,2$,
		\begin{equation}\label{local-Lip-F-2}
		\begin{split}
		\lvert M(\bd_1)-M(\bd_2)\rvert_{\el^2} \le c_2\biggl(\lve \d_1-\d_2\rve_2
		\lve \d_1\rve_2^{1-a} \lve \d_1\rve_3^a+ \lve
		\d_1-\d_2\rve^{1-a}_2 \lve \d_1-\d_2\rve_3^a \lve \d_2\rve_2
		\biggr).
		\end{split}
		\end{equation}
	\end{lem}
	\begin{proof}
		The proof of this result can be found in \cite[Lemma 6.4]{BHP18}
	\end{proof}
	%%%%%%%%%%%%%%%%%%%%%%%%%%%%%%%%%%%%%%%%%%%%%%%%%%%%%%%%%%%%%%%%%%%%%%%%%%%%%%%%%%%%%%%%%%%%%%%%%%%%%%%%%%%%%%%%%%%%%%%%%%
	%%%%%%%%%%%%%%%%%%%%%%%%%%%%%%%%%%%%%%%%%%%%%%%%%%%%%%%%%%%%%%%%%%%%%%%%%%%%%%%%%%%%%%%%%%%%%%%%%%%%%%%%%%%%%%%%%%%%%%%%%%%%
	%%%%%%%%%%%%%%%%%%%%%%%%%%%%%%%%%%%%%%%%%%%%%%%%%%%%%%%%%%%%%%%%%%%%%%%%%%%%%%%%%%%%%%%%%%%%%%%%%%%%%%%%%%%%%%%%%%%%%%%%%%%%%%
	\begin{lem}\label{Local-LIP-Lem-2}
	There exist $c_3>0$ such that for all $(\bv_i,
		\bd_i)\in \mathscr{E}$, i$=1,2$,
		\begin{equation}\label{local-Lip-F-3}
		\begin{split}
		\lve \tilde{B}(\bv_1,\bd_1)-\tilde{B}(\bv_2,\bd_2)\rve_1 \le c_3
		\biggl(\lve \nabla(\bv_1-\bv_2)\rve \lve
		\bd_1\rve_2^{1-a}\lve \d_1\rve_3^a
		+\lve(\bd_1-\bd_2)\lve^{1-a}_ 2 \lve(\bd_1-\bd_2)\lve^{a}_3\lve
		\nabla\bv_2\lve\biggr).
		\end{split}
		\end{equation}
	\end{lem}
	\begin{proof}
		Throughout this proof $C>0$ is an universal constant which may change from one term to the other. 	
	Let $(\bv_i,
	\bd_i)\in \mathscr{E}$, i$=1,2$, and
	$(\w,
	\db)=(\bv_1-\bv_2, \d_1-\d_2). $
Since $\tilde{B}$ is bilinear,  we have
	\begin{equation}
	\tilde{B}(\bv_1,\bd_1)-\tilde{B}(\bv_2,\bd_2)=\tilde{B}(\w,\d_1)+\tilde{B}(\v_2, \db)=J_1 +J_2.
	\end{equation}
	In order to estimate $\lve J_i\rve_1$,$i=1,2,$, we only  focus on estimating $\lve \nabla J_i \rve$, because by \eqref{EST-G1}  estimating $\lve J_i\rve$
	is easy.  By using the Leibniz rule, the H\"older inequality, \eqref{GAG-l4} and\eqref{GAG-LInf-2} we infer that
		\begin{align*}
		\lve \nabla J_1\rve\le & C \biggl(\lve \nabla \w\rve \lve \nabla
		\d_1\rve_{\el^\infty}+\lve \w\rve_{\el^4} \lve \nabla^2
		\d_1\rve_{\el^4}\biggr)\\
		\le& C \lve \w \rve_1 \biggl(\lve \nabla \d_1\rve^{1-a}_1 \lve \nabla \d_1\rve^a_2 +
		\lve \nabla^2 \d_1\rve^{1-a}_{L^2} \lve \nabla^2 \d_1\rve_{1}^a \biggr)
%		\le & C \lve \w \rve_1 \biggl(\lve  \d_1\rve^{1-a}_2 \lve \d_1\rve^a_3 +
%		\lve  \d_1\rve^{1-a}_{2} \lve  \d_1\rve_{3}^a \biggr)\
\le  C \lve \bv_1-\bv_2 \rve_1 \lve \d_1\rve^{1-a}_2 \lve \d_1\rve^a_3.
		\end{align*}
		In a similar way, we can prove that
		\begin{equation*}
		\lve \nabla J_2\rve\le C \lvert \nabla
		\bv_2 \rvert_{\el^2} \lve \d_1-\d_2\rve^{1-a}_2 \lve \d_1-\d_2\rve^a_3.
		\end{equation*}
		The inequality \eqref{local-Lip-F-3} easily follows from \eqref{EST-G1} and the estimates for 	$\lve \nabla J_i\rve$, $i=1,2$, above.
	\end{proof}
	
	\begin{lem}\label{Local-LIP-Lem-3}
		Let Assumption \ref{eqn-f} be satisfied. Then, there exists $c_4>0$
		such that 
		\begin{equation}\label{local-Lip-F-4}
		\begin{split}
		\lve f(\bd_1)-f(\bd_2)\rve_1 \le c_4\Big[1+\lve \bd_1\rve_2^{2N}+\lve \bd_2\rve_2^{2N}\Big]\lve \bd_1-\bd_2\rve_2, \text{ for all $\bd_1, \bd_2\in \bX_{1}\cap \bX_{\frac32}$}.
		\end{split}
		\end{equation}
	\end{lem}
	\begin{proof}
Let	$N \in I_d$, $k\in \{0,\ldots, N\}$ and $f$ be as in Assumption \ref{eqn-f}.  By the Young inequality and the fact that $\bH^2 $ is an algebra, we infer that there exists $C>0$ such that for all $\bd_1,\bd_2\in \bH^2$
		\begin{equation}\label{POL-EST-0}
		\lve\lvert \bd_1 \rvert^{2k} \bd_2\rve_2\le C  \lve \bd_1 \rve_2^{2k} \lve \bd_2\rve_2\le C \lve \bd_2\rve_2(1+\lve \bd_1 \lve_2^{2N}.
		\end{equation}
		Thus, it is enough  to establish \eqref{local-Lip-F-4}  for the leading term $a_N \lvert \bd \rvert^{2N}\bd$. For doing so,  we have $$\lvert \bd_1\rvert^{2N}\bd_1-\lvert \bd_2\rvert^{2N}\bd_2= \vert \bd_1\vert^{2N}(\bd_1-\bd_2)+
		\bd_2(\vert\bd_1\vert-\vert\bd_2\vert)(\sum_{k=0}^{2N-1}\vert \bd_1\vert^{2N-1-k}\vert \bd_2\rvert^{k} ),$$
		from which along with \eqref{POL-EST-0} we infer that \eqref{local-Lip-F-4}  is true for the leading term  $a_N \lvert \bd \rvert^{2N}\bd$.
	\end{proof}
%	Gathering the results of Lemma \ref{Local-LIP-Lem}-\ref{Local-LIP-Lem-3} yields the following important result.
	\begin{prop}\label{SLC-ST}
	Let $\alpha=\frac d4$, $d=2,3$. If Assumption \ref{eqn-f} is satisfied, then there exists $C_0>0$ such that for all $\y_i=(\bv_i,\d_i)$, \,\,
		$i=1,2$
		\begin{equation*}
		\begin{split}
		\lve \f(\y_1) -\f(\y_2)\rve_{\mathscr{H}} \le C_0 \Vert \y_1-\y_2\Vert_{\mathscr{V}}^{1-\alpha}
		\Big[\Vert \y_1-\y_2\Vert^{\alpha}_{\mathscr{V} }  \Vert \y_1\Vert^{1-\alpha}_{\mathscr{V}} \Vert \y_1\Vert_{\mathscr{E}}^\alpha + \Vert
		\y_1-\y_2\Vert_{\mathscr{E}}^\alpha  \Vert
		\y_2\Vert_{\mathscr{V}}\Big]\\
		+c_4\lve \y_1-\y_2\rve_{\mathscr{V} }\Big[1+\lve \y_1\rve_{\mathscr{V}}^{2N}+\lve \y_2\rve_{\mathscr{V}}^{2N}\Big].
		\end{split}
		\end{equation*}
	\end{prop}
	\begin{proof}
		The proposition is an easy consequence of Lemmata \ref{Local-LIP-Lem}-\ref{Local-LIP-Lem-3}. Thus, we omit  its proof.
	\end{proof}

	Using Theorems \ref{Thm:LocalUniqueness}, 
	\ref{thm_local} and \ref{thm_maximal-abstract} we obtain  our first main result.
	\begin{thm}\label{LC-Local-Sol}
		Let  $d=2,3$, $(\v_0,\d_0)\in \el^2(\Omega;\ve \times \bX_{1})$ be $\mathcal{F}_0$-measurable, $\bh\in \bH^2$. If Assumptions  \ref{eqn-f}, \ref{HYPO-ST} and \ref{Assum:Usual hypotheses} are satisfied,  then the problem  \eqref{ABSTRACT-LC}  has a unique maximal local solution $((\bv;\bd), \tilde{\tau}_\infty)$ satisfying the following properties.
		\begin{enumerate}[(1)]
%			\item\label{item-0} if, for all For all  $k\in \mathbb{N}$ and $(\bu, \mathbf{d})\in C([0,T]; \mathscr{V}) \bigcap L^2(0,T; \mathscr{E})$ we define
%			\begin{equation}\label{STOP}
%			\tau_k=\inf\{t\ge 0: \lvert \nabla \v(t)\rvert_{\el^2}^2 +\lve \d(t)\rve^2_2+\int_0^t \left(\lvert \rA\v(s)\rvert^2_{\el^2}+ \lve \d(s) \rve^2_3  \right)ds >
%			k^2\},
%			\end{equation}
%			then $\mathbb{P}$-a.s.
%			\begin{equation}
%			\tau_k \toup \tilde{\tau}_\infty.
%			\end{equation}
			\item\label{item-1} Given $R>0$ and  $\varepsilon >0$ there exists
			$\tau(\varepsilon,R)>0$ such that if $\mathbb{E}\Vert (\v_0,\d_0) \Vert^{2}_{\ve \times \bX_{1}} \leq R^{2}$, then 
			\[{\mathbb P}\big(\tilde{\tau}_\infty \geq \tau(\varepsilon,R)\big) \geq
			1-\varepsilon.\]
			\item \label{THM-ii} We also have
			\begin{align}
			&\mathbb{P}\left(\{ \tilde{\tau}_\infty <\infty \}\cap \{ \lvert \nabla \bv(t) \rvert_{\el^2} + \lve  \bd(t) \rve_2<\infty  \} \right)=0,\label{MAX-P1}\\
		&	\limsup_{t\toup\tilde{\tau}_\infty } \lvert \nabla \v(t)\rvert_{\el^2}^2 +\lve \d(t)\rve^2_2+\int_0^t \left(\lvert \rA\v(s)\rvert^2_{\el^2}+ \lve \d(s) \rve^2_3  \right)ds=\infty \text{ $\mathbb{P}$-a.s. on } \{\tilde{\tau}_\infty<\infty \}\label{MAX-P2}.
			\end{align}
		\end{enumerate}
		
	\end{thm}
	\begin{proof}
		Lemma \ref{SEM-1}-\ref{SEM-3} show that  $\{\mathbb{S}(t)\}_{t\geq 0}$ on
		$\mathscr{H}=\h\times \bX_{0}$ satisfies Assumption \ref{assum-01}.
		Thanks to Proposition \ref{SLC-ST} we can infer by applying Theorems \ref{Thm:LocalUniqueness},
		\ref{thm_local} and \ref{thm_maximal-abstract} that problem
		\eqref{ABSTRACT-LC} has a unique maximal local solution satisfying items \eqref{item-1} and \eqref{THM-ii} of Theorem \ref{LC-Local-Sol}.
	\end{proof}
	\subsection{Existence and uniquness of global strong solution: 2D case}
By using the Khashminskii test for non-explosions, see \cite[Theorem
III.4.1]{Kh_1980},  and some arguments from \cite{Brz+Masl+Seidler_2005}, we prove in this section that if $d=2$ then the problem \eqref{ABSTRACT-LC} has a unique global solution. 
	
	For all $(\bu, \mathbf{d}) \in C([0,T]; \h\times \bH^1) \bigcap L^2(0,T; \ve\times \bH^2)$ and $t\in [0,T]$  we put 
	\begin{align}
	& \mathcal{E}[\bu,\mathbf{d}](t)= \frac{1}{2}\left(\lvert \bu (t) \rvert^2_{\el^2}+ \lvert \mathbf{d}(t)\rvert_{\elb^2} + \lvert \nabla \mathbf{d}(t) \rvert^2_{\elb^2} + \int_\MO F( \mathbf{d}(t,x)) dx \right)\\
	& \mathscr{D}[\bu, \mathbf{d} ](t) = \lvert \rA^\frac12 \bu (t) \rvert^2_{\el^2}+ \lvert \mathrm{A}_1\mathbf{d}(t) - f(\mathbf{d}(t)  ) \rvert^2_{\elb^2}.
	\end{align}

	\begin{thm}\label{GLOBAL-ST}
		Let $d=2$, $N\in \mathbb{N}$, $\bh \in \bH^2$ and $(\v_0, \d_0) \in \el^2(\Omega; \ve \times  \bX_{1})$ such that 
		\begin{equation}\label{E_0}
		\mathbb{E} \lvert \mathcal{E}[\v_0,d_0]\rvert^{2(4N+2)})= \mathbb{E} \biggl(\vert \bv_0\vert^2_{\el^2} + \vert \bd_0\vert^2_{\elb^2}+\vert \nabla \bd_0 \vert^2_{\elb^2} + \int_{\mo}F(\bd_0(x)) dx \biggr)^{2(4N+2)}<\infty.
		\end{equation}
	If Assumptions  \ref{eqn-f}, \ref{Assum:Usual hypotheses} and \ref{HYPO-ST}, 
		then  the problem  \eqref{ABSTRACT-LC}  has a unique global strong solution.
	\end{thm}
	The proof of this theorem is given at the end of this subsection.
	%One of the crucial proposition we mentioned above is about the following uniform estimates.
	\begin{prop}\label{EST1}
		Let  all the assumptions of Theorem \ref{GLOBAL-ST} be satisfied and $p \in [2, 2(4N+1)]$.  Also, let $(\tau_k)_{k \in \mathbb{N}}$ be the sequence of stopping times defined by
			\begin{equation}\label{STOP}
					\tau_k=\inf\{t \in [0,\infty) : \lvert \nabla \v(t)\rvert_{\el^2}^2 +\lve \d(t)\rve^2_2+\int_0^t \left(\lvert \rA\v(s)\rvert^2_{\el^2}+ \lve \d(s) \rve^2_3  \right)ds >
					k^2\}, \; k\in \mathbb{N}.
					\end{equation}
		Then, there exist an increasing function $\varphi: [0,\infty) \to (0,\infty)$  and $\kappa_0=\kappa_0(p,\lvert \bh \rvert_{\bW^{1,4}})>0$ such that for all $k \in \mathbb{N}$
		\begin{equation}\label{Eq:ESTofVDinWeakerNorm}
		\me \sup_{t\in [0,T]}\left( \mathcal{E}[\v,\d] (t\wedge \tau_k)   \right)^p + \me \left[\int_0^{T\wedge \tau_k} \left(\mathscr{D}[\v,\d](s) -\frac{a_{N+1}}{2}  \lvert \d(s) \rvert^{2N+2}_{\elb^{2N+2}} \right)\, ds  \right] \le \kappa_0 \varphi(T) \Big( 1 + \me\lvert \mathcal{E}[\v,\d](0)\rvert^p    \Big).
		\end{equation}
	\end{prop}
	\begin{proof}
		The proof of this proposition will be given in Section \ref{AppB}.
	\end{proof}
	Hereafter, we set
	\begin{equation}\label{Eq:ConstantFrakC0}
	\mathfrak{C}_0= \kappa_0 \varphi(T)(1 + \me \lvert \mathcal{E}(\v,\d)(0)\rvert^{2(4N+2)} )
	\end{equation}
	\begin{cor}
		Let all the assumptions of Proposition \ref{EST1} be satisfied. Then, there exists $C>0$  such that for all $k \in \mathbb{N}$
		\begin{equation}\label{Eq:EstIntegralOfH2Norm}
		\me \left[\int_0^{T\wedge \tau_k}\lve \d(s) \rve^2_2   \right]^2 \le C (\mathfrak{C}_0 +1).
		\end{equation}
	\end{cor}
	\begin{proof}
		By \eqref{bigdanh2} and $\bH^1 \hookrightarrow \elb^{4N+2}$, which is valid for $d=2$, there exists a constant $C>0$ such that
		\begin{equation}\label{Eq:Bigdanh2withH1norm}
		\lve \d \rve^2_2 \le C( \lvert \rrA \d -f(\d)\rvert^2_{\elb^2}+ \lve \d \rve^{4N+2}_1+1  ),
		\end{equation}
		from which along with \eqref{Eq:ESTofVDinWeakerNorm} we conclude the proof of the corollary.
	\end{proof}
	Let  $\Psi_1: \bH^2 \to [0, \infty)$, $\Psi_2: D(\rA) \to [0,\infty)$ and $\Psi: D(\rA)\times \bH^2 \to 0,\infty) $ be defined by
	\begin{align}
	& \Psi_1(\mathbf{d})= \frac12 \lvert \rA_1 \mathbf{d} -f(\mathbf{d}) \rvert^2_{\elb^2}, \;\; \mathbf{d}\in \bH^2,\label{Eq:DefPsi1forD} \\
	& \Psi_2(\mathbf{u})= \frac12 \lvert \nabla \mathbf{u} \rvert^2_{\el^2},\;\; \mathbf{u}\in D(\rA),\label{Eq:DefPsi2forv}\\
	& \Psi(\mathbf{u},\mathbf{d})=\Psi_1(\mathbf{d}) + \Psi_2(\mathbf{u}), \text{ } (\mathbf{u},\mathbf{d})\in D(\rA) \times \bH^2.\label{Eq:DefFunctionForIto}
	\end{align}
	Hereafter, $\Psi_i^\prime$ and $\Psi_i^{\prime\prime}$, $i=1,2$, are the first and second Fr\'echet derivatives of $\Psi_i$, $i=1,2$.
	%%%%%%%%%%%%%%%%%%%%%%%%%%%%%%%%%%%%%%%%%%%%%%%%%%%%%%%%%%%%%%%%%%%%%%%%%%%%%%%%%
	\begin{lem}\label{Lem:FirstDerAppliedtoVdotNabalaD}
		There exists $\kappa_1>0$ such that for all $\mathbf{d}\in \bH^3$ and $\mathbf{u}\in D(\rA)$ we have
		\begin{equation}\label{Eq:FirstDerAppliedtoVdotNabalaD}
		-	\Psi_1^\prime(\mathbf{d})[\mathbf{u}\cdot \nabla \mathbf{d}] \le \kappa_1\Psi(\mathbf{u},\mathbf{d}) \left[\lve \mathbf{d} \rve^2_{1} + 1\right] \lVert \mathbf{d} \rVert^2_{2}   + \frac14 \lvert \rA \mathbf{u} \rvert^2_{\el^2} + \frac16 \lvert \nabla (\rA_1 \mathbf{d} -f(\mathbf{d})) \rvert^2_{\elb^2}.
		\end{equation}
	\end{lem}
	\begin{proof}
		In this proof $C>0$ is an universal constant which may change from one term to the other.
		Let $\mathbf{d}\in \bH^3\cap D(\rrA)$ and $\mathbf{u}\in D(\rA)$. 
		Observe that 
		\begin{align}
		\Psi_1^\prime(\mathbf{d})[\mathbf{g}]=\langle \Delta \mathbf{d} -f(\mathbf{d}), \rA_1 \mathbf{g} + f^\prime(\mathbf{d})[\mathbf{g}] \rangle \text{ for all } \mathbf{g}\in \bH^2, \label{Eq:FirstDeriveOfPsi}
		\end{align}
		and 
		\begin{equation*}
		\begin{split}
		-\Delta(\mathbf{u}\cdot \nabla \mathbf{d})=-(\Delta \mathbf{u} \cdot \nabla) \mathbf{d} - (\mathbf{u}\cdot \nabla) \Delta \mathbf{d} -2 \mathrm{tr}(\nabla \mathbf{u} \nabla^2) \mathbf{d}\\
		=-(\Delta \mathbf{u} \cdot \nabla) \mathbf{d} - 2 \mathrm{tr}(\nabla \mathbf{u} \nabla^2) \mathbf{d}+ (\mathbf{u}\cdot \nabla)[\Delta \mathbf{d} -f(\mathbf{d})] - f^\prime(\mathbf{d})[\mathbf{u}\cdot \nabla\mathbf{d}]
		\end{split}
		\end{equation*}
		Hence, by using the identities \eqref{tild-b-0} and \eqref{Eq:FirstDeriveOfPsi} we obtain
		\begin{equation}
		\begin{split}
		-	\Psi_1^\prime(\mathbf{d})[ \mathbf{u}\cdot \nabla \mathbf{d}]=\langle \rA_1\mathbf{d}-f(\mathbf{d}),  (\Delta \mathbf{u}\cdot \nabla) \mathbf{d} \rangle +2\langle \rA_1\mathbf{d}-f(\mathbf{d}), \mathrm{tr}(\nabla \mathbf{u} \nabla^2)\mathbf{d}\rangle,
		\end{split}
		\end{equation}
		which along with \eqref{GAG-l4} and the H\"older and Young inequalities imply that
		\begin{equation}
		\begin{split}
		-	\Psi_1^\prime(\mathbf{d})[ \mathbf{u}\cdot \nabla \mathbf{d}]& \le C \lvert \rA_1 \mathbf{d}-f(\mathbf{d}) \rvert_{\elb^4} [\lvert \rA \mathbf{u} \rvert_{\el^2} \lvert \nabla \mathbf{d} \rvert_{\elb^4} +  \lvert \nabla \mathbf{u} \rvert_{\el^4} \lvert \nabla^2 \mathbf{d} \rvert_{\elb^2} ]\\
%		& 	\le C \biggl(\lvert \rA_1 \mathbf{d}-f(\mathbf{d})\rvert_{\elb^2}^\frac12 \lvert \nabla (\rA_1 \mathbf{d} -f(\mathbf{d}) ) \rvert^\frac12_{\elb^2} + \lvert \rA_1 \mathbf{d}-f(\mathbf{d})\rvert_{\elb^2} \Biggr)\lvert \rA \mathbf{u} \rvert_{\el^2} [\lvert \nabla \mathbf{d} \rvert_{\elb^4} +   \lvert \nabla^2 \mathbf{d} \rvert_{\elb^2} ]
%		\\
		& \le \frac16 \lvert \nabla (\rA_1 \mathbf{d} -f(\mathbf{d}) ) \rvert^2_{\elb^2}+ \frac14 \lvert \rA \mathbf{u} \rvert^2_{\el^2}+ C \lvert \rA_1 \mathbf{d}-f(\mathbf{d})\rvert_{\elb^2}^2 [ \lvert \nabla \mathbf{d} \rvert_{\elb^4}^4 +  \lvert \nabla^2 \mathbf{d} \rvert_{\elb^2} ].
		\end{split}
		\end{equation}	
		Now, \eqref{Eq:FirstDerAppliedtoVdotNabalaD} easily follows from the last line of the the above chain of inequalities.
	\end{proof}
	\begin{lem}\label{Lem:FirstDerAppliedtoDeltaD+f(D)}
		There exists  $\kappa_2>0$ such that for all $\mathbf{d}\in \bH^3\cap D(\rrA)$ and $\mathbf{u}\in D(\rA)$  we have
		\begin{equation}\label{Eq:FirstDerAppliedtoDeltaD+f(D)}
	\langle	f ^\prime(\mathbf{d})[\rA_1\mathbf{d}-f(\mathbf{d})], \rA_1\mathbf{d}-f(\mathbf{d}) \rangle\le \frac16 \lvert \nabla(\rA_1  \mathbf{d} -f(\mathbf{d})) \rvert^2_{\elb^2} + \kappa_2 \Psi(\mathbf{u},\mathbf{d}) (1+ \lve \mathbf{d} \rve^{4N}_1 ).
		\end{equation}
	\end{lem}
	\begin{proof}
		Using part \eqref{Itemi:REM-H2} of  Remark \ref{REM-H2},  \eqref{GAG-l4}, the H\"older and  Young  inequalities, and by $\bH^1\hookrightarrow \el^{4N}$ (that is valid for $d=2$) we infer that  there exist  $C>0$ and $\kappa_2>0$ such that
		\begin{equation}
		\begin{split}
			\langle	f ^\prime(\mathbf{d})[\rA_1\mathbf{d}-f(\mathbf{d})], \rA_1\mathbf{d} -f(\mathbf{d}) \rangle
			&\le c_1 \int_\mo (1+ \lvert \mathbf{d} \rvert^{2N}) \lvert \rA_1\mathbf{d} -f(\mathbf{d})\rvert^2 dx\\
			&\le C \lvert \rA_1\mathbf{d}-f(\mathbf{d}) \rvert_{\el^4}^2 \lvert (1+ \lvert \mathbf{d} \rvert^{2N}) \rvert^2_{\elb^2}\\
%		&\le C \biggl( \lvert \rA_1\mathbf{d}-f(\mathbf{d}) \rvert_{\el^2} \lvert \nabla (\rA_1 \mathbf{d}+f(\mathbf{d})) \rvert_{\elb^2}+ \lvert \rA_1\mathbf{d}-f(\mathbf{d}) \rvert_{\el^2}^2 \biggr) (1 + \lvert\mathbf{d}\rvert^{2N}_{\elb^{4N}}) \\
		&\le \frac16 \lvert \nabla (\rA_1\mathbf{d}-f(\mathbf{d})) \rvert^2_{\elb^2}+ \kappa_2 \lvert \rA_1\mathbf{d}-f(\mathbf{d}) \rvert_{\el^2}^2 (1 + \lVert\mathbf{d}\rVert^{4N}_{1}),
		\end{split}
		\end{equation}	
		for all $\mathbf{d}\in \bH^3\cap D(\rrA)$ and $\mathbf{u}\in D(\rA)$. This completes the proof of the lemma.
	\end{proof}
	\begin{lem}\label{Lem:SecondDerofPsiAppliedtoGG}
		Let $\bh\in \bH^2$. Then, there exists $\kappa_5=\kappa_5(\lVert \bh \rVert_2)>0$ such that 
		\begin{equation}
		\lvert	\Psi_1^{\prime \prime}(\mathbf{d})[\mathbf{d}\times \bh, \mathbf{d}\times \bh] \rvert+\lvert	\Psi_1^\prime(\mathbf{d})[(\mathbf{d}\times\bh)\times \bh]  \rvert \le \Psi_1(\mathbf{d})+ \kappa_5 (1+ \lve \mathbf{d} \rve^{4N}_{1})   \lve \mathbf{d} \rve^2_{2},\text{ for all $ \mathbf{d}  D(\rrA)$.}
		\label{Eq:SecondDerofPsiAppliedtoGG}
		\end{equation}
	\end{lem}
	\begin{proof}
		Let \text{ $(\bu, \mathbf{d})\in D(\rA) \times D(\rrA)$.}
		We firstly recall that 
		\begin{equation}
		\Psi_1^{\prime \prime}(\mathbf{d})[\mathbf{g}, \mathbf{p}]=\langle \rA_1 \mathbf{d}-f(\mathbf{d}),  -f^{\prime \prime}(\mathbf{d})[\mathbf{g}, \mathbf{p}] \rangle + \langle \rA_1 \mathbf{p} - f^{\prime}(\mathbf{d})[\mathbf{p}], \rA_1 \mathbf{g}- f^\prime(\mathbf{d})[\mathbf{g}]\rangle,\; \mathbf{p}, \mathbf{g} \in D(\rrA).
		\end{equation}
		Hence, by recalling that $G(\mathbf{d})=\mathbf{d}\times \bh$ we have 
		\begin{align}
		\Psi_1^{\prime \prime}(\mathbf{d})[G(\mathbf{d}),G(\mathbf{d})]&= \lvert \rA_1(G(\mathbf{d}))- f^\prime(\mathbf{d})[G(\mathbf{d})]\rvert^2_{\elb^2}+ \langle \rA_1\mathbf{d} -f(\mathbf{d}), \rA_1(G(\mathbf{d}))- f^{\prime\prime}(\mathbf{d})[G(\mathbf{d}), G(\mathbf{d})]\rangle
	\nonumber 	\\
		&\le \frac14 \left\lvert\rA_1 \mathbf{d}-f(\mathbf{d})\right\rvert^2_{\elb^2}+\left\lvert \rA_1(G(\mathbf{d}))-f^{\prime\prime}(\mathbf{d})[G(\mathbf{d}), G(\mathbf{d})]\right\rvert^2_{\elb^2} + \lvert \rA_1(G(\mathbf{d}))-f^\prime(\mathbf{d})[G(\mathbf{d})]\rvert^2_{\elb^2}\nonumber \\
		& =\mathrm{I}_1+ \mathrm{I}_2+\mathrm{I}_3\label{Eq:EstDerA}.
		\end{align}
		Secondly, by the part \eqref{Itemi:REM-H2} of Remark \ref{REM-H2}, $\bH^2\hookrightarrow\elb^\infty$, $\bH^1 \hookrightarrow \el^4, \elb^{4N}$ and the H\"older and Young inequalities we infer that there exists a constant $C=C(\lVert \bh \rVert_2 )>0$ such that 
		\begin{equation}\label{Eq:EstDerB}
		\begin{split}
		\mathrm{I}_2 \le & \lvert \Delta(\mathbf{d}\times \bh) \rvert^2_{\elb^2} + \lvert \lvert \mathbf{d} \times \bh \rvert^2 \lvert f^{\prime\prime}(\mathbf{d})\rvert \rvert^2_{\elb^2}\\
		&\le 4 \left(\lvert \Delta \mathbf{d}\times \bh \rvert^2_{\elb^2}+ 2 \lvert \nabla \mathbf{d} \times \nabla \bh \rvert^2_{\elb^2} + \lvert \mathbf{d} \times \Delta \bh \rvert^2_{\elb^2}\right) + C (1+ \lvert \mathbf{d} \rvert^{4N}_{\elb^{4N}} )\lvert \mathbf{d}\times \bh\rvert_{\elb^\infty}^2 \\
		&\le C \lve \mathbf{d} \rve^2_2(1+ \lve \mathbf{d}\rve^{4N}_1 ).
		\end{split}
		\end{equation}
		In a similar way, we prove that  there exists a constant $C_7=C(\lVert \bh \rVert_2 )>0$ such that 
		\begin{equation}
		\begin{split}
		\mathrm{I}_3  \le C \lve \mathbf{d} \rve^2_2(1+ \lve \mathbf{d}\rve^{4N}_1 ).
		\end{split}
		\end{equation}
		Combining this last inequality with \eqref{Eq:EstDerA} and \eqref{Eq:EstDerB}  proves that $$\lvert \Psi_1^{\prime \prime}(\mathbf{d})[G(\mathbf{d}),G(\mathbf{d})]\rvert \le \frac12 \Psi_1(\mathbf{d}) +  C_7 \lve \mathbf{d} \rve^2_2(1+ \lve \mathbf{d}\rve^{4N}_1 ).$$ In a similar way, we can also show that 
		$\lvert \Psi_1^{\prime}(\mathbf{d})[(\mathbf{d}\times \bh)\times \bh]\rvert \le \frac12 \Psi_1(\mathbf{d}) +  C_7 \lve \mathbf{d} \rve^2_2(1+ \lve \mathbf{d}\rve^{4N}_1 ).$  One easily conclude the proof of the lemma from the last two estimates.
	\end{proof}
%	The last result related to the first derivative $\Psi^\prime_1$ of $\Psi_1$ is given in the following lemma.
	\begin{lem}\label{Lem:Psi1BDGNeeded}
		Let $\bh\in \bH^2$. Then, there exists $\kappa_5=\kappa_5(\lVert \bh \rVert_2)>0$ such that  for all $\mathbf{d} \in D(\rrA)$ 
		\begin{equation}\label{Eq:Psi1BDGNeeded}
		\lvert \Psi_1^{\prime}(\mathbf{d})[\mathbf{d}\times \bh] \rvert\le \kappa_6 \left[1+ \Psi_1(\mathbf{d})    + \lve \mathbf{d} \rve^{4N+1}_{1} +\lve \mathbf{d} \rve_1 \lve \mathbf{d}\rve_2 \right].
		\end{equation}
	\end{lem}
	\begin{proof}
		By part \eqref{Itemi:REM-H2} of Remark \ref{REM-H2},  \eqref{GAG-l4}, \eqref{GAG-LInf}, \eqref{GAG-LInf-2}, the H\"older and Young inequalities and  $\bH^1\hookrightarrow \elb^{4N+2}, \elb^4$, we infer that there exists $\kappa_6=\kappa_6(\lVert \bh \rVert_2)>0$ such that  for all $\mathbf{d} \in D(\rrA)$ 
		\begin{equation*}
		\begin{split}
		\lvert \Psi_1^{\prime}(\mathbf{d})[\mathbf{d}\times \bh] \rvert&\le
		\lvert  \rA_1 \mathbf{d} -f(\mathbf{d})\rvert_{\elb^2} \left(\lvert [\rA_1\mathbf{d} -f(\mathbf{d})]\times \bh + 2  \nabla \mathbf{d} \times \nabla \bh  + \mathbf{d} \times \Delta \bh  -f(\mathbf{d}) \times \bh \rvert_{\elb^2}\right)\\
		&	\le \kappa_6 \left(  \Psi_1(\mathbf{d})+ \lvert \nabla \mathbf{d} \rvert^2_{\elb^4} + \lvert \mathbf{d} \rvert^2_{\elb^\infty} + \lvert \mathbf{d} \rvert^{4N+2}_{\elb^{4N+2}} +1\right),\\
%			&	\le \kappa_6 \left(  \Psi_1(\mathbf{d})+ \lvert \nabla \mathbf{d} \rvert_{\elb^2} \lvert \nabla \mathbf{d} \rvert_{\mathbf{H}^1} + \lVert \mathbf{d} \rVert_{1} \lvert\nabla  \mathbf{d} \rvert_{2}   + \lvert \mathbf{d} \rvert^{4N+2}_{1} +1\right)\\
		&	\le \kappa_6 \left(  \Psi_1(\mathbf{d})+ \lve \mathbf{d} \rve_{1} \lve \mathbf{d} \rve_{2}   + \lvert \mathbf{d} \rvert^{4N+2}_{1} +1 \right).
		\end{split}
		\end{equation*}
%		for a constant $C>0$ depending only on the $\bH^2$-norm of $\bh$.
%		Now, using Gagliardo-Nirenberg inequalities \eqref{GAG-l4}, \eqref{GAG-LInf} and \eqref{GAG-LInf-2},  and the Sobolev embeddings $\bH^1\subset \elb^{4N+2}, \elb^4$ yields that there exists a constant $\kappa_6$ depending only on $\lve \bh \rve_2$ such that
%		\begin{equation}
%		\begin{split}
%		\lvert \Psi_1^{\prime}(\mathbf{d})[\mathbf{d}\times \bh] \rvert
%		&	\le \kappa_6 \left(  \Psi_1(\mathbf{d})+ \lvert \nabla \mathbf{d} \rvert_{\elb^2} \lvert \nabla \mathbf{d} \rvert_{\mathbf{H}^1} + \lvert \mathbf{d} \rvert_{1} \lvert\nabla  \mathbf{d} \rvert_{2}   + \lvert \mathbf{d} \rvert^{4N+2}_{1} +1\right)\\
%		&	\le \kappa_6 \left(  \Psi_1(\mathbf{d})+ \lve \mathbf{d} \rve_{1} \lve \mathbf{d} \rve_{2}   + \lvert \mathbf{d} \rvert^{4N+2}_{1} +1 \right).
%		\end{split}
%		\end{equation}
		This completes the proof of the Lemma \ref{Lem:Psi1BDGNeeded}.
	\end{proof}
	We will also need to the following results.
	\begin{lem}\label{Lem:NonlienarNSeItoV}
		There exists $\kappa_3>0$ such that for all $\mathbf{u}\in D(\rA)$ and $\mathbf{d}\in D(\rrA)$ 
		\begin{equation}
		-\Psi^\prime_2(\bu)(B(\bu,\bu))=	-\langle B(\bu,\bu), \rA \mathbf{u}\rangle \le \frac14 \lvert \rA \mathbf{u} \rvert^2_{\el^2} + \kappa_3 ( \lvert \mathbf{u}\rvert^2_{\el^2} \lvert \nabla \mathbf{u} \rvert^2_{\el^2}) \Psi(\mathbf{u},\mathbf{d}).
		\end{equation}
	\end{lem}
	\begin{proof}
		Using \eqref{GAG-l4}, the H\"older and Young inequalities we infer that $C>0$ such that
		\begin{equation}
		\begin{split}
		\langle B(\mathbf{u},\mathbf{u}) , \rA\mathbf{u} \rangle &=\langle \mathbf{u}\cdot \nabla \mathbf{u}, \rA\mathbf{u}\rangle
		\le \lvert \rA \mathbf{u} \rvert \lvert \mathbf{u} \rvert_{\el^4} \lvert \nabla \mathbf{u} \rvert_{\el^4}\\
		&	\le \lvert \rA \mathbf{u} \rvert^\frac32 \lvert \mathbf{u} \rvert^\frac12_{\el^2} \lvert \nabla \mathbf{u} \rvert^\frac12_{\el^2}
			\le \frac14 \lvert \rA \mathbf{u} \rvert^2_{\el^2} + C \lvert \mathbf{u} \rvert^2_{\el^2} \lvert \nabla \mathbf{u}\rvert^2_{\el^2} \Psi_1(\mathbf{u}),
		\end{split}
		\end{equation}
	 for all $(\mathbf{u}, \mathbf{d})\in D(\rA)\times  D(\rrA)$. We easily conclude the proof of Lemma  \ref{Lem:NonlienarNSeItoV} from last line.
	\end{proof}
	\begin{lem}\label{Lem:CouplingTermIto}
		There exists $\kappa_4>0$ such that 
		\begin{equation}
		-\Psi^\prime_2(\bu)(M(\mathbf{d},\mathbf{d}))=	-\langle M(\mathbf{d},\mathbf{d}), \rA\mathbf{u}\rangle \le \frac14 \lvert \rA \mathbf{u}\rvert^2_{\el^2} + \frac16 \lvert \nabla (\rA_1 \mathbf{d} -f(\mathbf{d})) \rvert^2_{\elb^2} + \kappa_4 \Psi(\mathbf{u},\mathbf{d}) \lve \mathbf{d} \rve^2_1 \lve \mathbf{d} \rve^2_2,
		\end{equation}
		for all $\mathbf{d}\in D(\rrA)$ satisfying $\nabla (\rA_1 \mathbf{d} +f(\mathbf{d})) \in \el^2$, and  $\mathbf{u}\in D(\rA)$.
	\end{lem}
	\begin{proof}
		In this proof $C>0$ is an universal constant.
		Let $\mathbf{d}\in D(\rrA)$ be such that $\nabla (\rA_1 \mathbf{d} +f(\mathbf{d})) \in \el^2$, and  $\mathbf{u}\in D(\rA)$.
		Firstly, since $\Pi: \el^2 \to \h$  is self-adjoint, $\nabla F(\mathbf{d})=\nabla \mathbf{d} f(\mathbf{d})$ and $\Div \rA \mathbf{u}=0$, we infer that
		\begin{equation}
		\begin{split}
		\langle M(\mathbf{d},\mathbf{d}) , \rA\mathbf{u} \rangle=& \frac12\langle \rA\bu, \nabla \lvert \nabla \mathbf{d}\rvert^2\rangle -\langle \rA \bu, \nabla \mathbf{d} \Delta \mathbf{d}\rangle \\
		=&- \langle \rA \mathbf{u} \cdot\nabla  \mathbf{d}, \rA_1 \mathbf{d}+f(\mathbf{d})\rangle + \langle \rA \mathbf{u} , \nabla F(\mathbf{d})\rangle.
		%=&-\langle \rA \mathbf{u} \cdot\nabla  \mathbf{d}, \rA_1 \mathbf{d}+f(\mathbf{d})\rangle.
		\end{split}
		\end{equation}
		Secondly, applying the H\"older, the Gagliardo-Nirenberg and the Young inequality yields
		\begin{equation}
		\begin{split}
		\langle M(\mathbf{d},\mathbf{d}) , \rA\mathbf{u} \rangle=& -\langle \rA \mathbf{u} \cdot\nabla  \mathbf{d}, \rA_1 \mathbf{d}+f(\mathbf{d})\rangle
		\le \lvert \rA \mathbf{u} \rvert_{\el^2} \lvert \rA_1 \mathbf{d} -f(\mathbf{d}) \rvert_{\elb^4} \lvert \nabla \mathbf{d} \rvert_{\el^4}\\
		\le &\frac14 \lvert \rA \mathbf{u} \rvert^2_{\el^2}+ \frac16 \lvert \nabla (\rA_1 \mathbf{d} -f(\mathbf{d})) \rvert^2_{\elb^2} + C \lvert \rA_1 \mathbf{d} -f(\mathbf{d}) \rvert^2_{\elb^2} \lve \mathbf{d} \rve^2_1 \lve \mathbf{d} \rve^2_2.
		\end{split}
		\end{equation}
		This completes the proof of Lemma \ref{Lem:CouplingTermIto}
	\end{proof}
Now, let $\kappa_i$, $i=1,\ldots,4$ be the constants from Lemmata \ref{Lem:FirstDerAppliedtoVdotNabalaD}, \ref{Lem:FirstDerAppliedtoDeltaD+f(D)}, \ref{Lem:NonlienarNSeItoV} and \ref{Lem:CouplingTermIto}. For all $t\ge 0 $ we set
	\begin{equation}\label{Eq:WeightToCancelbadterms}
	\begin{split}
	\Phi(t)=\exp\left({-\int_0^{t}\left[ (\kappa_1 +\kappa_4)(1+ \lve \d(s) \rve^2_{1}) \lve \d(s) \rve^2_2 + \kappa_2 (1+ \lve \d(s) \rve^{4N}_1) + \kappa_3 \lvert \v(s) \rvert^2_{\el^2}   \lvert \nabla \v(s) \rvert^2_{\el^2}     \right] ds            }\right).
	\end{split}
	\end{equation}
	% and
	% \begin{equation}
	% \Upsilon(t)= \Phi(t) \Psi(\v(t), \d(t)) = \Phi(t)\left[\Psi_1(\d(t))  + \Psi_2(\v(t)) \right].
	% \end{equation}
	Let $\mathfrak{C}_0>0$ be the constant defined in \eqref{Eq:ConstantFrakC0} and $\mathfrak{C}_1>0$ the constant defined by
	\begin{equation}\label{Eq:ConstantFrakC1}
	\mathfrak{C}_1= \Psi(\v_0,\d_0) + \mathfrak{C}_0+1
	\end{equation}
	\begin{prop}\label{STRONGER-NORM}
		Let $\Psi_1$, $\Phi$, $\mathfrak{C}_0$ and $\mathfrak{C}_1$ be defined  in \eqref{Eq:DefPsi1forD}, \eqref{Eq:WeightToCancelbadterms}, \eqref{Eq:ConstantFrakC0} and \eqref{Eq:ConstantFrakC1}, respectively. Let 	$(\tau_k)_{k \in \mathbb{N}}$ be the sequence of stopping times defined in \eqref{STOP}. 	Let $d=2$, $N\in \mathbb{N}$ and $\bh \in \bH^2$. 
		
		If  all the other assumptions of Theorem \ref{GLOBAL-ST} are satisfied, then there exists an increasing function $\psi:[0,\infty) \to (0,\infty)$ and $\kappa_9=\kappa_9(N, \lVert \bh \rVert_2)>0$ such that for all $k \in \mathbb{N}$
		\begin{align}
		&	\me \sup_{s\in[0,T]} \Phi(s\wedge \tau_k)\left(\lvert \nabla \v(s\wedge \tau_k) \rvert^2_{\el^2}+\Psi_1(\d(s\wedge \tau_k))  \right)\le \kappa_9 \psi(T) \mathfrak{C}_1,\label{Eq:EstSupofDel+F}\\
		& \me \int_0^{T\wedge \tau_k} \Phi(s) \left(  \lvert \rA\v(s)\rvert^2_{\el^2} +\lvert \nabla (\rA_1\d(s)+f(\d(s)   )\rvert^2_{\elb^2}  \right)\, ds \le \kappa_9 \psi(T) \mathfrak{C}_1. \label{Eq:EstInteofNablaDelta+F}
		\end{align}
	\end{prop}
	\begin{proof}
		The proof of this proposition will be given in Section \ref{AppB}.
	\end{proof}
	\begin{cor}
		Under all the assumptions of Proposition \ref{STRONGER-NORM}, there exists a $C>0$ such that for all $k \in \mathbb{N}$
		\begin{equation}\label{Eq:EstIntegrofH3NormofD}
		\me \int_0^{T\wedge \tau_k} \Phi(s) \lve \d(s) \rve^2_{3} ds\le C(\mathfrak{C}_0 + \mathfrak{C}_1+1).
		\end{equation}
	\end{cor}
	\begin{proof}
		By part \eqref{Itemi:REM-H2} of Remark \ref{REM-H2}, the H\"older inequality and $\bH^1\hookrightarrow \elb^{8N}\hookrightarrow \elb^4$ we infer that
		\begin{align*}
		\lvert \nabla f(\d) \rvert^2_{\elb^2} = & \lvert f^\prime(\d)[\nabla \d] \rvert^2_{\elb^2}\le C(( \lvert 1+ \lvert \d \rvert^{2N}) \lvert \nabla \d \rvert \rvert^2_{\elb^2}  )\\
		\le & \lvert \nabla \d\rvert^2_{\elb^4}+ \lvert \d \rvert^{4N}_{\elb^{8N}} \lvert \nabla \d \rvert^2_{\elb^4}
		\le C (\lve \d \rve^2_{2} + \lve \d \rve_1^{8N+2}).
		\end{align*}
		With this at hand we complete the proof by using  \eqref{Eq:EstIntegralOfH2Norm}, \eqref{Eq:EstInteofNablaDelta+F} and  the fact $$\lve \d \rve^2_3\le \lve \d\rve^2_2 + 2 \lvert \nabla (\Delta \d +f(\d)) \rvert^2_{\elb^2} +2 \lvert  \nabla f(\d) \rvert^2_{\elb^2}.$$
	\end{proof}
	After all these preparations we now proceed to the promised proof of Theorem \ref{GLOBAL-ST}.
	\begin{proof}[Proof of Theorem \ref{GLOBAL-ST}]
		By Theorem \ref{LC-Local-Sol} the problem \eqref{ABSTRACT-LC} has  a unique maximal local solution $((\bv,\bd); \tilde{\tau}_\infty)$. We shall prove that $\mathbb{P}\Big(\tilde{\tau}_\infty<\infty \Big)=0 $. For this aim, let $\{ \tau_k; k \in \mathbb{N}\}$ be the sequence of stopping times defined in \eqref{STOP}. 
		% % % % % % % % % % % % % % % % % % % % % % % % % % % % % % % % % % % % % % % % % % % % % % % % % % % % % % % % % % % % % % % % % % % % % % % % % %
	We first establish the following chain of inequalities
		\begin{equation*}
		\begin{split}
		\mathbb{P}\left(\tau_k<t\right)&\le  \mathbb{E}\left(1_{\{\tau_k<t\}} 1_{\{e^{\int_0^{t\wedge \tau_k}\phi(r)dr}\le k^2 \} } \right)+  \mathbb{E}\left(1_{\{\tau_k<t\}} 1_{\{e^{\int_0^{t\wedge \tau_k}\phi(r)\, dr}> k^2\} } \right),\\
		&\le \frac{1}{k^2}\mathbb{E}\biggl(1_{\{\tau_k<t\}} k^2 e^{-
			\int_0^{t\wedge \tau_k}\phi(s)ds }\biggr)  +\mathbb{E}\biggl(1_{\{\tau_k<t\}} 1_{\{\int_0^{t\wedge \tau_k} \phi(s)\, ds
			> 2 \log{k}\} } \biggr)=:\mathrm{I}+\mathrm{II}.
		\end{split}
		\end{equation*}
		Now, we estimate $\mathrm{I}$ and $\mathrm{II}$ separately.
		Firstly, from the definition of $\tau_k$ and $\Phi$ we have
		\begin{align*}
		\mathrm{I}\le &\frac{1}{k^2}\mathbb{E} \left[1_{\{\tau_k < t\}} \Phi(t\wedge \tau_k)\left( \lvert \nabla \bv(t\wedge \tau_k) \rvert^2_{\el^2} + \lve \bd(t\wedge \tau_k) \rve^2_2 \right) \right]\\
		&\qquad + \frac1{k^2}\mathbb{E} \left[1_{\{\tau_k < t\}} \Phi(t\wedge \tau_k)\int_0^{t\wedge \tau_k}\left( \lvert \rA \bv(s) \rvert^2_{\el^2} + \lve \bd(s) \rve^2_3 \right)\, ds \right]\\
		\le & \frac{1}{k^2}\mathbb{E} \left[\Phi(t\wedge \tau_k)\left( \lvert \nabla \bv(t\wedge \tau_k) \rvert^2_{\el^2}  + \lve \bd(t\wedge \tau_k) \rve^2_2\right) \right]+\frac{1}{k^2}\mathbb{E} \left[\int_0^{t\wedge \tau_k}\Phi(s)\left( \lvert \rA \bv(s) \rvert^2_{\el^2} + \lve \bd(s) \rve^2_3 \right)\, ds \right]
		\end{align*}
		From \eqref{Eq:Bigdanh2withH1norm},  \eqref{Eq:ESTofVDinWeakerNorm} , \eqref{Eq:EstInteofNablaDelta+F} and \eqref{Eq:EstIntegrofH3NormofD} we infer that there exists a constant $C>0$ such that for all $k \in \mathbb{N}$
		\begin{equation*}
		\mathrm{I}\le \frac{1}{k^2}C (\mathfrak{C}_0+\mathfrak{C}_1+1).
		\end{equation*}
		Secondly, we estimate $\mathrm{II}$ as follows
		\begin{align*}
		\mathrm{II}=& \mathbb{E}\left(1_{\{\tau_k < t\}} 1_{\{\int_0^{t\wedge \tau_k} \phi(r)\, dr > 2\log{k} \} }\right)
		\le  \int_{\{\tau_k<t\} \cap \{\int_0^{t\wedge \tau_k} \phi(r)\, dr>2\log{k} \}}\frac{ \int_0^{t\wedge \tau_k} \phi(r)\, dr}{2\log{k}} d\mathbb{P}\\
		\le & \frac1{2\log{k}} \int_{\{\tau_k < t\}} \int_0^{t\wedge \tau_k} \phi(r)\, dr\;\;\; d\mathbb{P}
		\le \frac{1}{2\log{k}} \mathbb{E} \int_0^{t\wedge \tau_k} \phi(r)\, dr.
		\end{align*}
		Now, from \eqref{Eq:ESTofVDinWeakerNorm} and \eqref{Eq:EstIntegralOfH2Norm} we infer that there exists a constant $C>0$ such that for all $k \in \mathbb{N}$
		\begin{align*}
		&\E\int_0^{t\wedge \tau_k} \lvert \v(s)\rvert^2_{\el^2}\lvert \nabla \v(s)\rvert_{\el^2}^2 ds
		\le \frac12 \E\sup_{0\le s\le t\wedge \tau_k}\lvert \v(s)\rvert^4_{\el^2} + \frac12 \E \biggl[\int_0^{t\wedge \tau_k}
		\lvert \nabla \v(s)\rvert_{\el^2}^2  ds
		\biggr]^2\le C (\mathfrak{C}_0+1),\\
		& \E\int_0^{t\wedge \tau_k}(1+1 \lve \d(s)\rve^{2N}_{\h^1})^2 ds\le C (\mathfrak{C}_0+1),\\
		& \E \int_0^{t\wedge \tau_k} \Vert \bd(s) \Vert^2_1 \Vert \bd (s) \Vert^2_2 ds \le \frac12 \E \sup_{0\le s\le t\wedge \tau_k} \Vert  \bd(s)\Vert^4_1 + \frac12 \E \biggl[\int_0^{t\wedge \tau_k} \Vert \bd(s) \Vert^2_2 ds \biggr]^2\le C (\mathfrak{C}_0+1).
		\end{align*}
		Thus, there exists a constant $C>0$ such that for all $k\in \mathbb{N}$
		\begin{equation*}
		\mathrm{II}  \le \frac1{2\log{k}} \E \int_0^{t\wedge \tau_k} \phi(s)ds\le \frac{C (\mathfrak{C}_0+1)}{\log{k}}.
		\end{equation*}
		Collecting the information about $\mathrm{I}$ and $\mathrm{II}$ together, we infer that
		\begin{equation*} %\label{Prob-tau}
		\begin{split}
		\lim_{k\rightarrow \infty} \mathbb{P}\left(\tau_k<t\right)&\le \lim_{k\rightarrow \infty} C (\mathfrak{C}_0+\mathfrak{C}_1+1) \left[\frac 1{k^2} +\frac{1}{\log{k}} \right]=0.
		\end{split}
		\end{equation*}
		% % % % % % % % % % % % % % % % % % % % % % % % % % % % % % % % % % % % % % % % % % % % % % %
		Now, we easily infer from the last estimate and part \eqref{THM-ii}  of Theorem \ref{LC-Local-Sol} that $\mathbb{P}\Big(\tilde{\tau}_\infty<\infty \Big)=0 $. This completes the proof of Theorem \ref{GLOBAL-ST}. 
	\end{proof}
	%%%%%%%%%%%%%%%%%%%%%%
	%%%%%%%%%%%%%%%%%%%%%%%%
	%%%%%%%%%%%%%%%%%%%%%%%%%%%
	\section{Basic estimates for the solution  $(\bv,\bd)$}\label{AppB}
	Throughout this section,   $((\bv,\bd);\tilde{\tau}_\infty)$ is the maximal local solution to problem \eqref{ABSTRACT-LC} from Theorem \ref{LC-Local-Sol} and $\{\tau_k: k \in \mathbb{N} \}$ is the sequence defined in \eqref{STOP}.  The first and second subsections are devoted to the proofs of Proposition \ref{EST1} and Proposition \ref{STRONGER-NORM}, respectively. 
	\subsection{Proof of Proposition \ref{EST1}}
	Before proceeding to the actual proof of  Proposition \ref{EST1} we state and prove the following result.
	\begin{lem}
		Let $\bh \in \bW^{1,4}$.Then,  there exists $C=C(\lVert \bh\rVert_{\bW^{1,4}})>0$ such that for all  $\mathbf{d}\in \rH^1$ 
		\begin{equation}\label{Eq:EstimateItoStratoCorrect}
		\lvert \nabla G(\mathbf{d}) \rvert^2_{\elb^2} + \langle \nabla \mathbf{d}, \nabla G^2(\mathbf{d}) \rvert \le C \lVert \mathbf{d} \rVert^2_{1}.
		\end{equation}
	\end{lem}
	\begin{proof}
	Let $\bh \in \bW^{1,4}$. Using the H\"older inequality, the embeddings $\bH^1 \hookrightarrow \elb^4$ and $\bW^{1,4}\hookrightarrow \el^\infty$, and
		\begin{equation}\label{Eq:MixedScalCrossProducts}
		\ba \cdot [(\mathbf{b} \times \mathbf{c}) \times \mathbf{d}   ]= - (\ba \times \mathbf{d}) \cdot (\mathbf{b} \times \mathbf{c}), \;\; \forall\;\; \ba, \mathbf{b}, \mathbf{c}, \mathbf{d}\in \err^3,
		\end{equation}
		we infer that for all $\mathbf{d}\in \rH^1$ and $i \in \{1,2\}$
%		Secondly, fixing $i \in \{1,2\}$, plugging  \eqref{Eq:MixedScalCrossProducts} in the expansion of $\lvert \partial_i (\mathbf{d}\times \bh ) \rvert^2_{\elb^2}+ \langle \partial_i \mathbf{d} , \partial_i [(\mathbf{d}\times \bh) \times \bh] \rangle$,  using the H\"older inequality and the Sobolev embeddings $\bH^1\subset \elb^4$ and $\bW^{1,4}\subset \el^\infty$ yield
		\begin{equation*}
		\begin{split}
		\lvert \partial_i (\mathbf{d}\times \bh) \rvert^2_{\elb^2} + \langle \partial_i \mathbf{d}, \partial_i  [(\mathbf{d} \times \bh) \times \bh ] \rangle =&  \lvert \mathbf{d} \times \partial_i \bh \rvert^2_{\elb^2}+ \langle \partial_i \mathbf{d} \times \bh, \mathbf{d} \times \partial_i \bh \rangle + \langle \partial_i \mathbf{d}, (\mathbf{d} \times \bh)\times \partial_i \bh \rangle \\
		&\le \lvert \partial_i \mathbf{d} \rvert^2_{\elb^4} \lvert \partial_i \bh \rvert^2_{\elb^4} + \lvert \partial_i \mathbf{d} \rvert^2_{\elb^2} \lvert \mathbf{d} \rvert^2_{\elb^4} \lvert \bh \rvert_{\elb^\infty} \lvert \partial_i \bh \rvert_{\elb^4} \le C \lVert \mathbf{d} \rVert^2_{1} \lvert \bh \rvert^2_{\bW^{1,4}}.
		\end{split}
		\end{equation*}
		Hence, summing over $i$ from $1$ to $2$  imply the desired inequality \eqref{Eq:EstimateItoStratoCorrect}.
	\end{proof}
	We now give the promised of the proposition.
	\begin{proof}[Proof of Proposition \ref{EST1} ]
		Without loss of generality (Wlog) we only give a proof for $p=2(4N+1)$. Let $\bh \in \bW^{1,4}$. Throughout this proof $C=C(\lVert \bh \rVert_{p, \bW^{1,4}})>0$ is a  constant which may change from one term to the next one.   Let $k \in \mathbb{N}$ be fixed and $\tau_k$ be defined by \eqref{STOP} . 
		
		Firstly, let $\Lambda:\bH^1 \to [0,\infty)$ be the map defined by
		\begin{equation}
		\Lambda(\mathbf{d})= \frac12 \lvert \mathbf{d} \rvert^2_{\elb^2} + \lvert \nabla \mathbf{d} \rvert^2_{\elb^2} +\frac12 \int_\MO F(\lvert \mathbf{d}(x) \rvert^2) dx, \;\; \mathbf{d}\in \bH^1.
		\end{equation}
		By Assumption \ref{eqn-f} and \cite[Lemma 8.10]{Brz+Millet_2012} the map $\Lambda(\cdot)$ is twice Fr\'echet differentiable. Moreover, elementary calculations and \eqref{Eq:MixedScalCrossProducts} imply 
%		 using we have  and its first and second derivatives satisfy
%		\begin{align}
%		& \Lambda^\prime(\mathbf{d})[\mathbf{g}]= \langle \mathbf{d}, \mathbf{g} \rangle +\langle \nabla \mathbf{d}, \nabla \mathbf{g} \rangle + \langle \mathbf{g}, f(\mathbf{d})\rangle,\\
%		& \Lambda^{\prime \prime}[\mathbf{g}, \mathbf{f}]= \langle \mathbf{g}, \mathbf{f} \rangle + \langle \nabla \mathbf{g} , \nabla \mathbf{f} \rangle + \langle \mathbf{g}, 2 \tilde{f}^\prime(\lvert \mathbf{d} \rvert^2) (\mathbf{d}\cdot \mathbf{f}) \mathbf{d}+ \tilde{f}(\lvert \mathbf{d} \rvert^2) \mathbf{f}\rangle.
%		\end{align}
%		These properties, the fact $f(\mathbf{d})=\tilde{f}(\lvert \mathbf{d} \rvert^2)\mathbf{d}$ and the identity \eqref{Eq:MixedScalCrossProducts}
%		imply that
		\begin{align}
		& \Lambda^\prime(\mathbf{d})[G(\mathbf{d})]= \langle \nabla \mathbf{d}, \mathbf{d}\times \nabla \mathbf{h} \rangle,\label{Eq:PsiDerGd}\\
		& \frac12 \Lambda^\prime (G^2(\mathbf{d})) + \frac12\Lambda^{\prime \prime}[G(\mathbf{d}), G(\mathbf{d}) ]=\frac12  \lvert \nabla G(\mathbf{d}) \rvert^2_{\elb^2} + \frac12 \langle \nabla \mathbf{d}, \nabla G^2(\mathbf{d}) \rangle. \label{Eq:ItoStrato-Correct}
		\end{align}
		We also observe that if $\bu \in \ve$ such that $\Div \bu=0$, then
		\begin{equation}
		\langle \bu \cdot \nabla \mathbf{d}, f(\mathbf{d}) \rangle = \frac12 \int_\MO \bu(x)\cdot \nabla F(\mathbf{d}(x)) dx=0.\label{Eq:Divvimply}
		\end{equation}
		
		Secondly, applying the It\^o formula to $\frac12 \lvert \v(t\wedge \tau_k)\rvert^2_{\el^2} + \Lambda(\d(t\wedge \tau_k)) $ and using \eqref{tild-b-0}, \eqref{B3}, \eqref{G1-eq-Md},  \eqref{Eq:PsiDerGd},  \eqref{Eq:ItoStrato-Correct}   and  \eqref{Eq:Divvimply} yield
		\begin{equation}\label{Eq:ResultIto}
		\begin{split}
		&\mathcal{E}[\v,\d](t\wedge \tau_k)-\mathcal{E}[\v,\d](0) + \int_0^{t\wedge \tau_k} \mathscr{D}[\v,\d] (s)\, ds - \int_0^{t\wedge \tau_k} \langle \d(s) , f(\d(s)) \rangle ds-\frac12 \int_0^{t\wedge \tau_k}   \lvert \nabla G(\d(s)) \rvert^2_{\elb^2}  ds \\
		& = \frac12 \int_0^{t\wedge \tau_k}  \langle \nabla \bd(s), \nabla G^2(\bd(s)) \rangle ds +
		\int_0^{t\wedge \tau_k} \langle \v(s), S(\v(s)) dW_1(s) \rangle + \int_0^{t\wedge \tau_k} \langle \nabla \d(s), \bd(s) \times \nabla \bh\rangle dW_2.
		\end{split}
		\end{equation}
		Before proceeding further, we should observe that thanks to Assumption \ref{eqn-f}  and \cite[Lemma 8.7]{Brz+Millet_2012}  we infer that there exists  exists $c>0$ independent of $k$  such that 
%		and one has
%		\begin{equation*}
%		-\langle \tilde{f}(\vert \d\vert^2)\d, \d\rangle= -\int_{\mathcal{O}} \tilde{f}(\vert \d(x)\vert^2)\vert \d(x)\vert^2 dx
%		=-\int_{\mathcal{O}} \sum_{l=1}^{N+1} a_{l-1} (\vert \d(x)\vert^2)^{l} dx,
%		\end{equation*}
%		which along with \cite[Lemma 8.7]{Brz+Millet_2012}  imply that there 
		\begin{equation*}
		\frac{-a_{N+1}}{2} \int_{\mathcal{O}} \lvert \d(x)\rvert^{2N+2} dx-c \int_{\mathcal{O}} \lvert \d(x)\rvert^2 dx \le \langle -f(\d), \d\rangle.
		\end{equation*}
		Hence,  plugging this inequality and \eqref{Eq:EstimateItoStratoCorrect} into \eqref{Eq:ResultIto} implies
		\begin{equation*}
		\begin{split}
		&\mathcal{E}[\v,\d](t\wedge \tau_k)-\mathcal{E}[\v,\d](0) + \int_0^{t\wedge \tau_k} \mathscr{D}[\v,\d] (s)\, ds -\frac{a_{N+1}}{2} \int_0^{t\wedge \tau_k} \lvert \d(s) \rvert^{2N+2}_{\el^{2N+2}} ds\\
		&  \le  \int_0^{t\wedge \tau_k} \langle \v(s), S(\v(s)) dW_1(s) \rangle + \int_0^{t\wedge \tau_k} \langle \nabla \d(s), \bd(s) \times \nabla \bh\rangle dW_2+ \frac12 \int_0^{t\wedge \tau_k}  \mathcal{E}[\v,\d](s)ds .
		\end{split}
		\end{equation*}
		Thirdly, by taking the supremum over $s\in [0,t]$, raising to the power $p$, taking the mathematical expectation to both sides of the above inequality and applying the H\"older inequality  we obtain
		\begin{equation}\label{Eq:PowerPofNRJ}
		\begin{split}
		& \me \sup_{s\in [0,t]}\lvert \mathcal{E}[\v,\d](s\wedge \tau_k)\rvert^p + \me \left[ \int_0^{t\wedge \tau_k}\left(\mathscr{D}[\v,\d](s) - \frac{a_{N+1}}{2} \lvert \bd(s) \rvert^{2N+2}_{\elb^{2N+2}} \right)ds     \right]^p\\
		& \le C\mathbb{E}\lvert \mathscr{E}[\v,\d] (0)\rvert^p+Ct^{p-1}\me \int_0^{t\wedge \tau_k} \rvert \mathcal{E}[\v,\bd](s)\lvert^p ds + C \me\sup_{s\in [0,t]} \biggl\lvert \int_0^{s\wedge \tau_k}\langle \v(s), S(\v(s)) dW_1(s)   \biggr\rvert^p\\
		& \qquad +
		C \me\sup_{s\in [0,t]} \biggl\lvert \int_0^{s\wedge \tau_k}\langle \nabla \d(s), \bd(s) \times \bh \rangle dW_2(s)   \biggr\rvert^p.
		\end{split}
		\end{equation}
		For the time being let us assume that there exists $C=C(p,\lVert \bh \rVert_{\bW^{1,4}})>0$ such that
		\begin{equation}
		\begin{split}
		& \mathscr{M}(t\wedge \tau_k):=C \me\sup_{s\in [0,t]} \biggl\lvert \int_0^{s\wedge \tau_k}\langle \v(s), S(\v(s)) dW_1(s)   \biggr\rvert^p+
		C \me\sup_{s\in [0,t]} \biggl\lvert \int_0^{s\wedge \tau_k}\langle \nabla \d(s), \bd(s) \times \bh \rangle dW_2(s)   \biggr\rvert^p\\
		&\qquad \qquad \le \frac12 \me \sup_{s\in [0,t]} \lvert \mathcal{E}[\v,\d](s\wedge \tau_k) \rvert^p + C_1t^\frac{p}2 + C_1 t^{\frac{p-2}{2}}\me \int_0^{t\wedge \tau_k} \lvert \mathscr{E}[\v,\d](s)\rvert^p ds,\label{Eq:EstPowerPofMart}
		\end{split}
		\end{equation}
		from which along with \eqref{Eq:PowerPofNRJ} we infer that there exists $C=C(p,\lVert \bh \rVert_{\bW^{1,4}})>0$  such that
		\begin{equation}%\label{Eq:PowerPofNRJ}
		\begin{split}
		& \me \sup_{s\in [0,t]}\lvert \mathcal{E}[\v,\d](s\wedge \tau_k)\rvert^p + 2 \me \left[ \int_0^{T\wedge \tau_k}\left(\mathscr{D}[\v,\d](s) - \frac{a_{N+1}}{2} \lvert \bd(s) \rvert^{2N+2}_{\elb^{2N+2}} \right)ds     \right]^p\\
		& \le C \mathbb{E}\lvert \mathscr{E}[\v,\d] (0)\rvert^p+ Ct^\frac{p}2  + C(t^{p-1}+ t^{\frac{p-2}{2}}) \me \int_0^{t\wedge \tau_k} \rvert \mathcal{E}[\v,\bd](s)\lvert^p ds.
		\end{split}
		\end{equation}
		We then apply the Gronwall lemma and obtain the desired \eqref{Eq:ESTofVDinWeakerNorm}.\\ 
		Thus, it remains to prove  \eqref{Eq:EstPowerPofMart}. For this purpose, by applying the Burkholder-Davis-Gundy (BDG), the H\"older and Young  inequalities, and Assumption \ref{HYPO-ST} (mainly \eqref{Eq:Hypo-ST})  we infer that 
		\begin{equation*}
		\begin{split}
		\mathscr{M}(t\wedge \tau_k)& \le C_4 \me \left[\int_0^{t\wedge \tau_k}  \lvert \v(s) \rvert^2_{\el^2} \lvert S(\v(s))  \rvert^2_{\mathcal{T}_2} \right]^\frac{p}2+C_4  \me \left[\int_0^{t\wedge \tau_k}  \lvert \nabla \d(s) \rvert^2_{\el^2} \lvert \d(s) \times \nabla \bh \rvert^2_{\elb^2} \right]^\frac{p}2\\
		& \le C_4  t^\frac{p-2}{2} \me \int_0^{t\wedge \tau_k} \rvert \mathcal{E}[\v,\d](s)\rvert^{\frac{p}{2}} [\lvert S(\v(s))  \rvert^2_{\mathcal{T}_2} + \lvert \d(s) \times \nabla \bh \rvert^2_{\elb^2}  ]^\frac{p}{2} ds\\
%		&\le \frac12 \me \sup_{s\in [0,t]} \me \rvert \mathcal{E}[\v,\d](s\wedge \tau_k)\rvert^{p} +C_5 t^{p-2} \me  \int_0^{t\wedge \tau_k} [\lvert S(\v(s))  \rvert^2_{\mathcal{T}_2} + \lvert \d(s) \times \nabla \bh \rvert^2_{\elb^2}  ]^p ds\\
		&\le \frac12 \me \sup_{s\in [0,t]} \me \rvert \mathcal{E}[\v,\d](s\wedge \tau_k)\rvert^{p} +C_5 t^{p-2} \me  \int_0^{t\wedge \tau_k} [1+ \lvert\v(s)\rvert^2_{\el^2} + \lvert \d(s) \rvert^2_{\elb^4} \lvert  \nabla \bh \rvert^2_{\elb^4}  ]^p ds.
		\end{split}
		\end{equation*}
		The last line and $\bH^1\hookrightarrow \elb^4$  imply \eqref{Eq:EstPowerPofMart}. This  completes the proof of Proposition \ref{EST1}.
	\end{proof}
	\subsection{Proof of Proposition \ref{STRONGER-NORM}}
	Before giving the promised proof we firstly state and prove two important lemmata. The first one is inspired  by  \cite[Lemma 6.3]{ZB+MO}.
	\begin{lem}\label{Lem:Abstract-Lem-L(x)}
		Let $V_1, H_1, \tilde{V}_1$ be there separable Hilbert spaces such that the embeddings $V_1\hookrightarrow H_1\hookrightarrow \tilde{V}_1$  are dense and continuous. Let $\mathcal{A}: V_1 \to \tilde{V}_1$ be a bounded linear map and  $\mathfrak{f}: [0,T] \to H_1$ and $g: [0,T] \to V_1$ measurable and progressively measurable respectively  such that
		\begin{equation}\label{Eq:Integrabilitoffandg}
		\mathbb{E}\int_0^T[ \lvert \mathfrak{f}(t) \rvert^2_{\tilde{V}_1} + \lvert g(t) \rvert^2_{H_1} ]dt <\infty.
		\end{equation}
Let $x: [0,T]\times \Omega \to V_1$ be a progressively measurable and $H_1$-continuous process  such that 
	\begin{align}
	&	\mathbb{E}\int_0^T \lvert x(s) \rvert^2_{V_1}ds<\infty,\label{Eq:Integrabilityofx}\\
&	x(t) =x(0) + \int_0^t \mathcal{A}x(s)\, ds + \int_0^t \mathfrak{f}(s)\, ds + \int_0^t g(s) dW_2(s)\, 	\text{ for all $t$  $\mathbb{P}$-a.s..} 	\label{Eq:Equationforx}
	\end{align}
	Now, let $V_2, H_2$ be two separable Hilbert spaces and $V_2^\ast$ the dual of $V_2$. We identify $H_2$ with its dual and we assume that the embeddings $V_2\hookrightarrow H_2\hookrightarrow V_2^\ast $ are continuous and dense. Let $\mathcal{B}: V_2 \to V_2^\ast$ be a bounded linear map. Let $L: \tilde{V}_1 \to V_2^\ast$ be a twice Fr\'echet differentiable map such that:
	\begin{enumerate}
		\item $L(V_1)\subset V_2$ and $L(H_1)\subset H_2$,
		\item there exists $\mathcal{H}: V_1 \to H_2$ such that for every $z\in V_1$
		\begin{equation}\label{Eq:LprimeAz=Bz+H(z)}
		L^\prime(z)[\mathcal{A}z]=\mathcal{B}L(z) +\mathcal{H}(z).
		\end{equation}
		\item The map $L^{\prime\prime}$ is  bounded on balls.
	\end{enumerate}
	
	Then for every $t\in [0,T]$, $\mathbb{P}$-a.s. the following identity holds in $V_2^\ast$
	\begin{equation}\label{Eq:IdentityforL(x)inDualV1}
\begin{split}
	L(x(t)) =L(x(0)) +\int_0^t\mathcal{B}L(x(s))\, ds + \int_0^t\Bigl( L^\prime(x(s))[\mathfrak{f}(s)] +\mathcal{H}(x(s)) \Bigr)\, ds \\+ \frac12 \int_0^t L^{\prime\prime}(x(s))[g(s), g(s)] ds+ \int_0^t L^\prime(x(s))[g(s)]dW_2(s).
\end{split}
	\end{equation}
	\end{lem}
\begin{proof}
	Let $t\in [0,T]$ and $\varphi \in V_2$. Let $z\ni L_\varphi:H_1 \mapsto L_\varphi(z):= {}_{V_2^\ast}\langle L(z), \varphi \rangle_{V_2}\in \mathbb{R}$.
	By the assumptions on $L$, $L_\varphi$ is twice Fr\'echet differentiable, $L_\varphi$ and $L^\prime_\varphi$ are continuous on $H_1$, $L_\varphi$, $L_\varphi^\prime$ and $L_\varphi^{\prime\prime}$ are locally bounded. Hence, by It\^o formula, see \cite[Theorem 3.2]{Pardoux}, we infer that  $\mathbb{P}$-a.s.
\begin{equation*}
\begin{split}
L_\varphi(x(t))=L_\varphi(x(0)) + \int_0^t {}_{V_2^\ast}\langle L^\prime(x(s))[\mathcal{A}(x(s)) +\mathfrak{f}(s)   ], \varphi \rangle_{V_2} + \int_0^t {}_{H_2}\langle L^\prime(x(s)) [g(s)],\varphi \rangle_{H_2} dW_2(s)\\+ \frac12\int_0^t {}_{H_2}\langle L^{\prime\prime}(x(s))[g(s), g(s)], \varphi\rangle_{H_2} ds
\end{split}
\end{equation*}
Using \eqref{Eq:LprimeAz=Bz+H(z)} yields
\begin{equation*}
\begin{split}
L_\varphi(x(t))=L_\varphi(x(0)) + \int_0^t {}_{V_2^\ast}\langle \mathcal{B}(x(s)), \varphi \rangle_{V_2} ds  +\int_0^t {}_{V_2^\ast}\langle L^\prime(x(s))[ \mathfrak{f}(s)]+H(x(s)) , \varphi \rangle_{V_2}ds \\+ \frac12\int_0^t {}_{H_2}\langle L^{\prime\prime}(x(s))[g(s), g(s)], \varphi\rangle_{H_2} ds +
\int_0^t {}_{H_2}\langle L^\prime(x(s)) [g(s)],\varphi \rangle_{H_2} dW_2(s) .
\end{split}
\end{equation*}
This completes the proof of \eqref{Eq:IdentityforL(x)inDualV1}.
\end{proof}
	\begin{lem}\label{Lem:IdentityforDd+f(d)}
	Let $\bh\in \bH^2$ and $\{y(t): t\in [0,\tilde{\tau}_\infty)\}$ be the local process defined by
	\begin{equation}\label{Eq:Def-Dd+f(d)}
	y(t):=
	 \rA_1 \bd(t) - f(\bd(t)), \;\text{  } t\in [0,\tilde{\tau}_\infty).
	\end{equation}
	Then, for all $t\in [0,T]$, $k \in \mathbb{N}$, $\mathbb{P}$-a.s. the following equation holds in $(\bH^1)^\ast$
	\begin{equation}\label{Eq:IdentityforDd+f(d)}
\begin{split}
&	y(t\wedge \tau_k)+ \int_0^{t\wedge \tau_k} \Bigl(\rA_1 y(r) + (\rA_1-f^\prime(\bd(r)))[\bv(r) \cdot \nabla \bd(r)] \Bigr)\, dr  -\int_0^{t\wedge \tau_k} (\rA_1-f^\prime(\bd(r)))[G(\bd(r))] dW_2\\
&= y(0) +\frac12 \int_0^{t\wedge \tau_k} \Bigl(2f^\prime(\bd(r))y(r) +(\rA_1-f^\prime(\bd(r)))[G^2(\d(r))] - f^{\prime \prime}(\d(r))[G(\d(r)), G(\d(r))] \Bigr)\, dr.
\end{split}
	\end{equation}
	\end{lem}
%\begin{Rem}
%	Since for all $k $  $\mathbb{P}$-a.s. $\{ \bd(t\wedge \tau_k): t\in [0,T]  \}\in C([0,T]; D(\rA_1)) \cap L^2(0,T; D(\rA_1^\frac32))$, the identity \eqref{Eq:IdentityforDd+f(d)} should be understood as the equality of two $(\bH^1)^\ast$-valued random variables.
%\end{Rem}
\begin{proof}[Proof of Lemma \ref{Lem:IdentityforDd+f(d)}]
	Let $\bh\in \bH^2$. 
	Let us put $V_1=D((I+\rA_1)^\frac32)$, $H_1=D(\rA_1)$, $\tilde{V}_1=\bH^1$ and $\mathcal{A}=\mathcal{B}=-\rA_1$. We also set $V_2=\bH^1$, $H_2=\elb^2$, $V_2^\ast=(\bH^1)^\ast$. The map
	$L: \bH^1 \ni z \mapsto L(z):= \rA_1z-f(z) \in (\bH^1)^\ast$satisfies the assumptions of Lemma \ref{Lem:Abstract-Lem-L(x)}. In particular, if we set $\mathcal{H}(z)= \rA_1f(z) -f^\prime(z)[\rA_1z], \,z \in V_1 $, then 
	$$ L^\prime(z)[\rA_1z]= \rA_1[ \rA_1 z - f(z)] +\rA_1f(z) -f^\prime(z)[\rA_1z] = \rA_1 L(z) +\mathcal{H}(z)\in (\bH^1)^\ast. $$
 Now, let $k\in \mathbb{N}$  and
 \begin{align}
 &\mathfrak{f}= \mathds{1}_{[0,\tau_k)} [ -\bv \cdot \nabla \bd + f(\bd) +\frac12G^2(\bd) ],\label{Eq:DefFrakf}\\
 &g= \mathds{1}_{[0,\tau_k)} G(\bd).\label{Eq:Defsmallg}
 \end{align}
By Lemmata \ref{Local-LIP-Lem-2} and\ref{Local-LIP-Lem-3}, and the definition of $\tau_k$ we infer that there exists $C>0$  such that
 \begin{align}
& \mathbb{E} \int_0^{t} \lvert \mathfrak{f}(r) \rvert^2_{\bH^1}\le  \mathbb{E}\int_0^{t\wedge \tau_k} \Bigl(\lvert \bv(r) \cdot \nabla \bd(r)\rvert^2_{\bH^1} +\lvert f(\bd(r))\rvert^2_{\bH^1}+\lvert \frac12 (\bd(r) \times \bh) \times \bh \rvert^2_{\bH^1}\Bigr)\, dr\le C,\label{Eq:Proof of IdDd+f(d)-1}\\
& \mathbb{E}\int_0^{t\wedge \tau_k} \lvert G(\bd(r)) \rvert^2_{\bH^2} dr =  	\mathbb{E}\int_0^{t\wedge \tau_k} \lvert \bd(r) \times \bh \rvert^2_{\bH^2} dr \le C.\label{Eq:Proof of IdDd+f(d)-2}
 \end{align}
 These mean that $\mathfrak{f}$ and $g$ satisfy \eqref{Eq:Integrabilitoffandg}.  Because the local strong solution $\y=(\bv,\bd)$ of \eqref{ABSTRACT-LC} satisfies  \eqref{eq-locsol_01-a}, the process $x(t)=\bd(t\wedge \tau_k)$ satisfies   \eqref{Eq:Integrabilityofx} and \eqref{Eq:Equationforx} with $\mathfrak{f}$ and $g$ as defined above.   By setting
 $y(t\wedge \tau_k)=L(\bd(t\wedge \tau_k)), t\ge 0,$  and applying Lemma \ref{Lem:Abstract-Lem-L(x)} we obtain
 \begin{equation}\label{Eq:AbsFormofAd+f(d)}
 \begin{split}
 y(t\wedge \tau_k)=y(0) +\int_0^{t\wedge \tau_k} \Bigl( -\rA_1y(r) + L^\prime(\bd(r)) [\mathfrak{f}(r)] + \mathcal{H}(\bn(r))   \Bigr)\, dr \\+ \frac12\int_0^{t\wedge \tau_k} L^{\prime \prime}(\bd(r))[g(r), g(r)] dr + \int_0^{t\wedge \tau_k} L^\prime(\bd(r)) [g(r) ]dW_2(r).
 \end{split}
 \end{equation}
We complete the proof of the lemma by  taking  into account the last line and the following identity 
\begin{equation}\label{Eq:1stDer+2ndDerofL}
  L^\prime(z)=\rA_1 -f^\prime(z) \text{ and }  L^{\prime\prime } (z)= - f^{\prime\prime}(z) \text{ for every } z\in \bH^1.
\end{equation}
 \end{proof}
	We now give the promised  proof of Proposition \ref{STRONGER-NORM}.
	\begin{proof}[Proof of Proposition \ref{STRONGER-NORM}]
		Throughout, $L(z)=\rA_1 z-f(z)$ and $\mathcal{H}(z)= \rA_1f(z) - f^\prime(z)[\rA_1 z]$ be defined as in the proof of Lemma \ref{Lem:IdentityforDd+f(d)}. 
	Keeping in mind the notations of  \ref{Lem:IdentityforDd+f(d)}, in particular \eqref{Eq:DefFrakf}, \eqref{Eq:Defsmallg} and \eqref{Eq:1stDer+2ndDerofL},
		we set
		\begin{equation}
		\mathfrak{v}= -\mathds{1}_{[0,\tau_k]} \rA_1 L(\bd) + L^\prime(\bd)[\mathfrak{f} ] +\mathds{1}_{[0,\tau_k]}\mathcal{H}(\bd) +\frac12 L^{\prime\prime}(\bd)[g, g], \, k \in \mathbb{N}.
		\end{equation}
Then,  for all $k \in \mathbb{N}$ and $F\in L^2(\Omega\times[0,t]; \elb^2)$, $t\ge 0$, 
			\begin{equation}\label{Eq:Proof of IdDd+f(d)-3-b}
		\mathbb{E} \int_0^{t\wedge \tau_k} \lvert f^\prime(\bd(r))[F(r)] \rvert^2_{\elb^2} dr \le c_1^2(1+k^2) \mathbb{E}\int_0^{t\wedge \tau_k } \lvert F(r) \rvert^2_{\elb^2}dr.
		\end{equation}
	In fact, from Remark \ref{REM-H2}\eqref{Itemi:REM-H2}, the embedding $\bH^2\hookrightarrow \elb^\infty$ and \eqref{Eq:Proof of IdDd+f(d)-1} we infer that
		\begin{align*}
		\mathbb{E} \int_0^{t\wedge \tau_k} \lvert f^\prime(\bd(r))[F(r)] \rvert^2_{\elb^2}\le&  \mathbb{E}\int_0^{t\wedge \tau_k} \int_\mo \lvert  f^\prime(\bd(r,x))[F(r,x)]\rvert^2 dxdr\\
		& \le c_1^2 \mathbb{E} \int_0^{t\wedge \tau_k} (1 + \lvert \bd(r,x) \rvert^{4N})\lvert F(r,x)\rvert^2 dxdr\\
		&\le c_1^2 \mathbb{E}\biggl([1+\sup_{t\in [0,T]} \lvert \bd(t\wedge \tau_k) \rvert^{4N}_{\el^\infty}]\int_0^{\tau_k} \lvert F(r) \rvert^2_{\elb^2} dr  \biggr) \\
		&	\le c_1^2(1+k^{2N}) \mathbb{E}\int_0^{t\wedge \tau_k } \lvert F(r) \rvert^2_{\elb^2}dr.
		\end{align*}
		In a similar way, we can prove that  for all $k \in \mathbb{N}$ and  $g=1_{[0,\tau_k]} \bd\times \bh $
		\begin{equation}\label{Eq:Proof of IdDd+f(d)-4}
		\mathbb{E} \int_0^{t\wedge \tau_k} \lvert f^{\prime\prime}(\bd(r))[g(r)   , g(r)    ]\rvert^2_{\elb^2} \le c^2_2 \lVert \bh \rVert^4_{2}(1+k^{2N} )\mathbb{E}\int_0^{t\wedge \tau_k } \lvert \bd(r) \rvert^2_{\elb^2}dr .
		\end{equation}
	From the continuity of the linear map $\rA_1:\bH^1 \to (\bH^1)^\ast $, the embedding $\bH^1\hookrightarrow \elb^2$,  \eqref{Eq:Proof of IdDd+f(d)-3-b}, \eqref{Eq:Proof of IdDd+f(d)-4} along with \eqref{Eq:Proof of IdDd+f(d)-1} and \eqref{Eq:Proof of IdDd+f(d)-2} we infer that there exits $C=C(k)>0$ such that
		\begin{equation}\label{Eq:DriftsatisfiesPardouxConditions}
\mathbb{E}\int_0^{t} \lvert \mathfrak{v}(r) \rvert^2_{(\bH^1)^\ast} \le C \mathbb{E}\int_0^{t\wedge \tau_k}\Bigl( \lvert L(\bd(r)) \rvert^2_{\el^2} + \lvert \mathfrak{f}(r) \rvert^2_{\bH^1} +\lvert f(\bd(r) ) \rvert^2_{\bH^1}
+ \lvert f^{\prime \prime}(\bd(r))[g(r), g(r) ] \rvert^2_{\elb^2} \Bigr)dr <\infty.
		\end{equation}
		
		Next,  let $\{ y(t): t\in [0,\tilde{\tau}_\infty)  \}$ be the process defined in \eqref{Eq:Def-Dd+f(d)}. The process $\{ y(t\wedge \tau_k): t \ge 0  \}$ is an $\elb^2$-valued process and satisfies the equivalent equations \eqref{Eq:IdentityforDd+f(d)} and \eqref{Eq:AbsFormofAd+f(d)}. Hence, by  \eqref{Eq:DriftsatisfiesPardouxConditions} we can apply the  It\^o formula to $ \frac12\lvert y(t\wedge \tau_k)\rvert^2_{\elb^2}=\frac12\lvert \rA_1 \mathbf{d} -f(\mathbf{d}) \rvert^2_{\elb^2}=\Psi_1(\mathbf{d})$ (see \cite[Theorem 3.2]{Pardoux}), and use the fact
		\begin{equation*}
		-(\rA_1-f^\prime(\bd))[\bv \cdot \nabla \bd]+ (f^\prime(\bd)y - (\rA_1-f^\prime(\bd))[\frac12G^2(\d)] - \frac12f^{\prime \prime}(\d)[G(\d), G(\d)] \in \elb^2,
		\end{equation*}
	to infer that for all $k \in \mathbb{N}$ and $t\ge 0$
%		\begin{equation}
%		\begin{split}
%		&	\frac12 \lvert y(t\wedge \tau_k)\rvert^2+ \int_0^{t\wedge \tau_k} {}_{(\bH^1)^\ast}\langle \rA_1 y(r), y(r) \rangle_{\bH^1} dr=	\frac12 \lvert y(t\wedge \tau_k)\rvert^2+ \int_0^{t\wedge \tau_k} \lvert \nabla y(r)\lvert^2_{\elb^2} dr\\	
%	&=\frac12 \lvert y(0) \rvert^2_{\elb^2}+ \int_0^{t\wedge \tau_k}
%		\Bigl\langle f^\prime(\bd(r))[y(r)] -(\rA_1- f^\prime(\d(r)))[G^2(\bd(r)) -\bv(r) \cdot \nabla \bd(r) ] , y(r)   \Bigr\rangle dr\\
%	&	\qquad +\frac12\int_0^{t\wedge \tau_k} \Bigl( \lvert (\rA_1- f^\prime(\bd(r)) )[G(\bd(r))] \rvert^2_{\el^2} - \langle f^{\prime\prime}[G(\bd(r)),G(\bd(r))], y(r) \rangle \Bigr)dr\\
%	& \qquad + \int_0^{t\wedge \tau_k} \Bigl \langle (\rA_1- f^\prime(\bd(r)) )[G(\bd(r))], y(r) \Bigr\rangle dW_2(r).
%		\end{split}
%		\end{equation}
%		Using the functional $=\frac12\lvert \rA_1 \mathbf{d} -f(\mathbf{d}) \rvert^2_{\elb^2}$ one can recast the last line on the following form 
		\begin{equation} \label{Eq:Psi(d)Formnula-1}
		\begin{split}
	&	\Psi_1(\bd(t\wedge \tau_k)) + \int_0^{t\wedge \tau_k} \lvert \nabla (\rA_1\bd(r) -f(\bd(r))  ) \rvert^2_{\elb^2} dr~+\int_0^{t\wedge \tau_k}\Psi^\prime_1(\d(r))[ \v(r)\cdot \nabla \d(r)]dr\\
	&	=
		\Psi_1(\bd_0) +\int_0^{t\wedge \tau_k}\Psi^\prime_1(\d(r))[ \frac12 G^2(\d(r))]dr + \frac12 \int_0^{t\wedge \tau_k} \Psi_1^{\prime\prime}(\bd(r))[G(\bd(r)), G(\bd(r)) ] dr\\
	&	- \int_0^{t\wedge \tau_k} \langle\rA_1\bd(r) -f(\bd(r)), f^\prime(\bd(r))[\rA_1\bd(r) -f(\bd(r))] \rangle dr+ \int_0^{t\wedge \tau_k} \Psi_1(\bd(r)) [G(\bd(r))]dW_2(r).
		\end{split}
		\end{equation}
		In preparation of our next step, let us set
		\begin{equation*}
		\mathrm{M}(t\wedge \tau_k)= \int_0^{t\wedge \tau_k} \Phi(s) \Psi_2(\v(s)) \circ S(\v(s)) dW_1(s) + \int_0^{t\wedge \tau_k} \Phi(s) \Psi_1(\d(s))[G(\d(s))]  dW_2(s),\, k \in \mathbb{N}, t\ge 0. 
		\end{equation*}
	With \eqref{Eq:Psi(d)Formnula-1} and the definitions of $\Psi_2, \Psi$ and $\Phi$ (see \eqref{Eq:DefPsi2forv}, \eqref{Eq:DefFunctionForIto} and \eqref{Eq:WeightToCancelbadterms}) in mind, we	apply the It\^o formula to $\Upsilon(t\wedge \tau_k) =\Phi(t\wedge \tau_k) \Psi(\bu,\bd)(t\wedge \tau_k) $ and obtain that for all $k \in \mathbb{N}$ and $t\ge 0$
		\begin{equation*}
		\begin{split}
		\Upsilon(t\wedge \tau_k) -\Upsilon(0)=& \mathrm{M}(t\wedge \tau_k)+ \int_0^{t\wedge \tau_k} \Phi(s) \Psi^\prime_2(\v(s)) [-B(\v(s),\v(s )  ) -M(\d(s), \d(s))    ] ds\\
		&+\int_0^{t\wedge \tau_k}  \Phi(s) \Psi^\prime_1(\d(s))[ -\v(s)\cdot \nabla \d(s) +\frac12 G^2(\d(s))   ]ds\\
		& +\frac12 \int_0^{t\wedge \tau_k} \Phi(s) \Psi_1^{\prime\prime}(\d(s))[G(\d(s)), G(\d(s))] ds+\int_0^{t\wedge \tau_k} \frac{d}{ds}\Phi(s) \Psi(\v(s), \d(s)))\, ds\\
		&- \int_0^{t\wedge \tau_k} \Phi(s) \langle\rA_1\bd(s) -f(\bd(s)), f^\prime(\bd(s))[\rA_1\bd(s) -f(\bd(s))] \rangle ds\\
		&	-\int_0^{t\wedge \tau_k}\Phi(s)\left(\lvert \nabla (\rA_1\bd(s))-f(\bd(s)) \rvert^2_{\elb^2}+ \lvert \rA \v(s) \rvert^2 - \lvert S(\v(s)) \rvert^2_{\mathcal{T}_2(\rK_1; \ve)} \right)ds.
		\end{split}
		\end{equation*}
		By Assumption \ref{HYPO-ST}, \eqref{Eq:Hypo-ST-Rem},  Lemma \ref{Lem:FirstDerAppliedtoVdotNabalaD}-\ref{Lem:CouplingTermIto} and the facts $\lvert \Phi \rvert\le 1$ and
		\begin{equation*}
		\begin{split}
		\frac{d}{ds}\Phi(s) \Psi(\v(s), \d(s))) = &-\left[ ((\kappa_1 +\kappa_4) \lve \d(s) \rve^2_2 + \kappa_2) (1+ \lve \d(s) \rve^2_1) + \kappa_3 \lvert \v(s) \rvert^2_{\el^2}   \lvert \nabla \v(s) \rvert^2_{\el^2}   \right] \\
		& \qquad \qquad \times \Phi(s)\Psi(\v(s), \d(s)),
		\end{split}
		\end{equation*}
		we infer that there exists $\kappa_7>0$ such that for all $k \in \mathbb{N}$ and $t\ge 0$
		\begin{equation}\label{Eq:ItoresultStrongnormMart}
		\begin{split}
		&\me \sup_{s\in [0,t]} \Upsilon(s\wedge \tau_k) + \int_0^{t\wedge \tau_k} \Phi(s) \left[\frac14 \lvert \rA \v(s) \rvert^2_{\el^2}  + \frac12 \lvert\nabla (\rA_1 \d(s)  +f(\d(s))      )  \rvert^2_{\el^2}     \right] ds \\
		&  \le \Upsilon(0) + \kappa_7 T+ \kappa_7 \me \int_0^{t\wedge\tau_k} \Upsilon(s)\, ds + \kappa_5\me  \int_0^{t\wedge \tau_k} (1+ \lve \d(s) \rve^{4N}_1) \lve \d(s) \rve^2_{2} ds + \me \sup_{s\in [0,t]}\lvert \mathrm{M}(s\wedge \tau_k) \rvert.
		\end{split}
		\end{equation}
		\noindent Next, by applying the BDG inequality, taking into account Assumption \ref{HYPO-ST}, \eqref{Eq:Psi1BDGNeeded} and the fact $\lvert \Phi \rvert\le 1 $ we infer that there exists $\kappa_8>0$ such that for all $k \in \mathbb{N}$ and $t\ge 0$
		\begin{equation*}
		\begin{split}
	&	\me \sup_{s\in [0,t]}\lvert \mathrm{M}(s\wedge \tau_k) \rvert \\
	& \le \kappa_8\me \left(\int_0^{t\wedge \tau_k} \Phi(s) \Psi_2(s)[1+ \Psi_2(s)] \Phi(s)\, ds   \right)^\frac12  + \kappa_8
		\me \left(\int_0^{t\wedge \tau_k}   \Phi^2(s) \lvert \Psi_1^\prime(\d(s))[G(\d(s))]\rvert^2 ds  \right)^\frac12 \\
		& \le \frac18 \me \sup_{s\in [0,t]} [\Phi(s)\Psi_2(\v(s)) ]+ \kappa_8 T +\kappa_8 \me \left[\int_0^{t\wedge \tau_k} [\Phi(s)\Psi_1(\d(s))]^2  ds \right]^\frac12\\
		&  \qquad \qquad + \kappa_8 \me \left[\int_0^{t\wedge \tau_k} \left(1 + \lve \d(s) \rve^{8N+4}_1 + [\lve \d(s) \rve_1 \lve \d(s)\rve_2]^2    \right)\, ds \right]^\frac12 +\kappa_8 \me\int_0^{t\wedge \tau_k} \Phi(s) \Psi_2(\v(s))ds  \\
		&\le  \frac14 \me \sup_{s\in [0,t]} \Upsilon(s)+ \kappa_8 T + \kappa_8 \me\int_0^{t\wedge \tau_k} \Upsilon(s)\, ds  + \kappa_8 \me \left[\int_0^{t\wedge \tau_k} \left(1 + \lve \d(s) \rve^{8N+2}_1 + [\lve \d(s) \rve_1 \lve \d(s)\rve_2]^2    \right)\, ds \right]^\frac12.
		\end{split}
		\end{equation*}
		Using the last inequality and absorbing the term $\frac14 \me \sup_{s\in [0,t]} \Upsilon(s)$   in the LHS of  \eqref{Eq:ItoresultStrongnormMart}, and applying Gronwall inequality imply that there exist an increasing function $\psi:[0,\infty) \to (0,\infty)$ and a constant $\kappa_9>0$ such that for all $k \in \mathbb{N}$ and $T\ge 0$ 
		\begin{equation}\label{Eq:FinalStepEstinStrongNorm}
		\begin{split}
		\me \sup_{s\in [0,T]} \Upsilon(s\wedge \tau_k) +\me \int_0^{T\wedge \tau_k} \Phi(s) \left[\frac14 \lvert \rA \v(s) \rvert^2_{\el^2}  + \frac12 \lvert\nabla (\rA_1 \d(s)  +f(\d(s))      )  \rvert^2_{\el^2}     \right] ds\\
		\le
		\psi(T) \biggl(1+ \Upsilon(0)  + \me \left[\int_0^{T\wedge \tau_k} \left(1 + \lve \d(s) \rve^{8N+2}_1 + [\lve \d(s) \rve_1 \lve \d(s)\rve_2]^2    \right)\, ds \right]^\frac12\\
		+\me  \int_0^{T\wedge \tau_k} (1+ \lve \d(s) \rve^{4N}_1) \lve \d(s) \rve^2_{2} ds   \biggr)
		\end{split}
		\end{equation}
		Using the estimate \eqref{Eq:ESTofVDinWeakerNorm} we easily conclude that there exists a constant $\tilde{\kappa}_0>0$ such that
		\begin{equation*}
		\begin{split}
		& \me \left[\int_0^{t\wedge \tau_k} \left(1 + \lve \d(s) \rve^{8N+2}_1 + [\lve \d(s) \rve_1 \lve \d(s)\rve_2]^2    \right)\, ds \right]^\frac12
		+\me  \int_0^{t\wedge \tau_k} (1+ \lve \d(s) \rve^{4N}_1) \lve \d(s) \rve^2_{2} ds \\
		&  \qquad \qquad \le  \tilde{\kappa}_0 (\mathfrak{C}_0+1),
		\end{split}
		\end{equation*}
		which along with \eqref{Eq:FinalStepEstinStrongNorm} complete the proof of Proposition \ref{STRONGER-NORM}
	\end{proof}
	% % % % % % % % % % % % % % % % % % % % % %
	% % % % % % % % % % % % % % % % % % % % % % %
	% % % % % % % % % % % % % % % % % % % % % % % % % % %
	%%%%%%%%%%%%%%%%%%%%%%%%%%%%%%%
	%%%%%%%%%%%%%%%%%%%%%%%%%%%%%
	\section{Strong solution for an abstract stochastic equation}\label{ABST-STRONG}
By a fixed point method we prove in this section general results about the existence and uniqueness of maximal local solution to stochastic evolution equations (SEEs) with 
	Lipschitz coefficients. 
	\subsection{Notations and Preliminary}\label{abstract-framework}
	Let $V$,  $E$ and $H$ be separable Banach spaces such that
	$E\hookrightarrow V$. We denote the norm in $V$ by $\Vert
	\cdot \Vert$ and for $a,b\in [0,\infty)$ with $a<b$ we put
	\begin{equation}\label{eqn-X_T}
	X_{a,b}:= C([a,b];V) \cap L^2(a,b;E)
	\end{equation}
	with the norm $\lvert \cdot \rvert_{X_{a,b}}$ defined by 
	\begin{equation}\label{eqn-X_T-norm}
	\vert u\vert_{X_{a,b}}^2= \sup_{s \in [a,b]} \Vert u(s)\Vert^2+\int_a^b \vert u(s) \vert_E^2\, ds.
	\end{equation}
	If $a=0$ we simply write $X_{a,b}=X_b$. If $a=b$ then the space $X_{a,a}$ is isomorphic to $V$.

Suppose that $\delta\in [0,T]$ and  $\ba \in X_\delta$. Define
\begin{align}
\hat{X}_{\delta,T,\ba}:=\{ u \in X_T: u|_{[0,\delta]}=\ba \}.
\end{align}
 Obviously,  $\hat{X}_{\delta,T,\ba}$ is a closed subspace  $X_T$ and hence a Banach space with the norm of  $\lvert \cdot \rvert_{X_T}$.
 Note that if $\delta=0$, $\ba \in X_0$ can be identified with  $=\ba (0)\in V$ and so  $\hat{X}_{0,T,\ba}=\{ u \in X_T: u(0)=\ba(0) \} $.

	Let $F$ and $G$ be two nonlinear mappings satisfying the following
	sets of conditions.
	\begin{assum}\label{assum-F}
		Suppose that $F: E \to H$ is such that
		$F(0)=0$ and there exist $N\in\mathbb{N}$ and  $p_i \geq 1$,  $\alpha_i \in [0,1)$, $i=1,\cdots, N$,   and $C>0$
		such that 		for all $x, y\in E$.
		\begin{equation}\label{eqn-local Lipschitz-F}
		\begin{split}
		\vert F(y)-F(x) \vert_H \leq C \sum_{i=1}^N \Big[ \Vert y-x\Vert \Vert
		y\Vert^{p_i-\alpha_i} \vert y\vert_E^{\alpha_i} + \vert y-x\vert_E^{\alpha_i}
		\Vert y-x\Vert^{1-\alpha_i} \Vert x\Vert^{p_i}\Big].
%\\		+ C \Big[ \Vert y-x\Vert \Vert		y\Vert^{q-\gamma} \vert y\vert_E^\gamma + \vert y-x\vert_E^\gamma \Vert y-x\Vert^{1-\gamma} \Vert x\Vert^q\Big],
		% +C \Vert y-x\Vert,
		\end{split}
		\end{equation}
	\end{assum}
	\begin{assum}\label{assum-G}
		Assume that $G: E \to V$  such that $G(0)=0$ and there exists $k
		\geq 1$, $\beta \in [0,1)$ and $C_G>0$ such that
		\begin{equation}\label{eqn-local Lipschitz-G}
		\Vert G(y)-G(x) \Vert \leq C_G \Big[ \Vert y-x\Vert \Vert
		y\Vert^{k-\beta} \vert y\vert_E^\beta + \vert y-x\vert_E^\beta
		\Vert y-x\Vert^{1-\beta} \Vert x\Vert^k\Big],
		\end{equation}
		for all $x, y\in E$.
	\end{assum}
	Let $(\Omega,
	\mathcal{F}, \mathbb{P})$ be a complete probability space equipped
	with a filtration $\mathbb{F}=\{\mathcal{F}_t: t\geq 0\}$
	satisfying the usual hypothesis.   By $\mathscr{M}^2(X_T)$ we denote the Banach space of all $E$-valued processes $u$ that are progressively measurable and with trajectories belonging to $X_T$ $\mathbb{P}$-a.s., with the norm
	\begin{equation}\label{eqn-M^2X_T}
	\vert u \vert_{ \mathscr{M}^2(X_T)} = \left(\mathbb{E}\Big[ \sup_{s \in [0,T]}
	\Vert u(s)\Vert^2+\int_0^T \vert u(s) \vert_E^2\, ds\Big].\right)^\frac12
	\end{equation}
	Let us also formulate the following assumptions.
	\begin{assum}\label{assum-01}
		Suppose that $E\hookrightarrow V \hookrightarrow H$. Consider (for
		simplicity) a one-dimensional Wiener process $W(t)$.\\
		Assume that $S(t)$, $t\in [0,\infty)$, is a family of bounded linear operators on the space $H$ such that:   the following properties are satisfied. 
		\begin{trivlist}
				\item[(i)] For every $T>0$, the linear map $$L^2(0,T;H) \ni f \mapsto \{S\ast f(t) =\int_0^{t}  S(t-r) f(r)\, dr; t\in [0,T] \} \in X_T, $$ is continuous. 
%				, i.e., there exists $C_2>0$ such that  
%				\begin{equation}\label{ineq-dc}
%						\vert u \vert _{X_T}\leq C_1 \vert f \vert _{L^2(0,T;H)}, f\in L^2(0,T;H).
%						\end{equation}
%				\begin{equation}\label{ineq-sc}
%						\vert u \vert _{\mathscr{M}^2(X_T)}\leq C_2 \vert \xi \vert _{\mathscr{M}^2(0,T;V)},\, .
%						\end{equation}
				
\item[(ii) ]For every $T>0$, the linear map 
				$$\mathscr{M}^2(0,T; V) \ni \xi \mapsto \{ S \diamond \xi(t):=  \int_0^t S(t-r) \xi(r)\, dW(r); t \in [0,T]  \}\in \mathscr{M}^2(X_T), $$  is continuous. 
				
		\item[(iii)] For every $T>0$, the linear map $$V\ni u_0 \mapsto \{[0,T]\ni t \mapsto Su_0(t):=S(t)u_0 \} \in X_T,$$ is continuous.
%			process $u=S \diamond \xi$ defined by
%			\[ u(t)=\int_0^t S(t-r) \xi(r)\, dW(r), \;\; t \in [0,T]\]
%			belongs to $\mathscr{M}^2(X_T)$ and
%			\begin{equation}\label{ineq-sc}
%			\vert u \vert _{\mathscr{M}^2(X_T)}\leq C_2 \vert \xi \vert _{\mathscr{M}^2(0,T;V)}.
%			\end{equation}
		\end{trivlist}
%	
%%		  a function
%%		$u=S\ast f$ defined by
%%		\[ u(t)=\int_0^t S(t-r) f(r)\, dr, \;\; t \in [0,T]\]
%%		belongs to $X_T$ and
%%		\begin{equation}\label{ineq-dc}
%%		\vert u \vert _{X_T}\leq C_1 \vert f \vert _{L^2(0,T;H)}.
%%		\end{equation}
%		(ii) For every $T>0$ and every process  $\xi\in \mathscr{M}^2(0,T;V)$ a
%		process $u=S \diamond \xi$ defined by
%		\[ u(t)=\int_0^T S(t-r) \xi(r)\, dW(r), \;\; t \in [0,T]\]
%		belongs to $\mathscr{M}^2(X_T)$ and
%		\begin{equation}\label{ineq-sc}
%		\vert u \vert _{\mathscr{M}^2(X_T)}\leq C_2 \vert \xi \vert _{\mathscr{M}^2(0,T;V)}.
%		\end{equation}
%		(iii) For every $T>0$ and every $u_0\in V$, a function $u=Su_0$
%		defined by
%		\[ u(t)= S(t)u_0,  \;\; t \in [0,T]\]
%		belongs to $X_T$. Moreover, for every $T_0>0$ there exist $C_0>0$
%		such that for all $T\in (0,T_0]$,
%		\begin{equation}\label{ineq-dc-2}
%		\vert u \vert _{X_T}\leq C_0 \Vert u_0 \Vert.
%		\end{equation}
	\end{assum}
	Now let us consider a semigroup $S(t)$, $t\in [0,\infty)$ as above
	and the abstract SEE
	\begin{equation}\label{ABS-SPDE-1}
	u(t)=S(t)u_0+\int_0^t S(t-s) F(u(s))\, ds+\int_0^t S(t-s) G(u(s))
	dW(s),\;\; \mbox{ for all }t>0
	\end{equation}
	
	which is a mild version of the problem
	
	\begin{equation}\label{ABS-SPDE-strong}
	\left\{\begin{array}{rl} du(t)&= Au(t)\,dt+  F\big(u(t)\big)\, dt+
	G\big(u(t)\big)
	dW(t),\;\;t>0,\\
	u(0)&=u_0.
	\end{array}
	\right.
	\end{equation}
	\begin{Def}\label{def-local solution-2}
		Assume that a $V$-valued  $\mathcal{F}_0$ measurable random variable  $u_0$ is given. A local
		solution to problem \eqref{ABS-SPDE-strong} (with the initial time
		$0$) is a pair $(u,\tau)$ such that
		\begin{enumerate}
			\item $\tau$ is an accessible stopping time, \item
			$u: [0,\tau)\times \Omega \to V$ is an admissible\footnote{This
				also follows from  condition (3) below.} process, \item there
			exists a sequence  $(\tau_m)_{m\in \mathbb{N}}$ of
			 finite stopping times  such that $\tau_m \toup \tau$
			$\mathbb{P}$-a.s. and, for every $m\in \mathbb{N}$ and $t\ge 0$,
			we have
			\begin{eqnarray}\label{eq-locsol_00}
			&&\mathbb{E}\Big(  \sup_{s\in [0,t\wedge \tau_m]} \Vert
			u(s)\Vert^2 +\int_0^{t\wedge \tau_m} \vert u(s)\vert_E^2 \,
			ds\Big)<\infty,
			\\
			\label{eq-locsol_00-b} u(t\wedge \tau_m)&=&S(t\wedge
			\tau_m)u_0+\int_0^{t\wedge \tau_m}
			S(t\wedge\tau_m-s) F(u(s))\, ds\\
			\nonumber &+&\int_0^t
			\mathds{1}_{[0,\tau_m)}S(t-s)G(u(s\wedge\tau_m))\,  dW(s).
			\end{eqnarray}
		\end{enumerate}
		Along the lines of \cite{Brz+Elw_2000}, we said that a
		local solution $u(t)$, $t < \tau$ is called global iff
		$\tau=\infty$ $\mathbb{P}$-a.s.
	\end{Def}
	Let us first formulate the following useful result.
	\begin{prop}
		\label{prop-local solution}  Assume that a pair $(u,\tau)$ is a
		local solution to problem \eqref{ABS-SPDE-strong}.Then for
		every finite stopping time $\sigma$, a pair $(u_{\vert [0, \tau
			\wedge \sigma)\times \Omega}, \tau \wedge \sigma)$ is also a local
		mild solution to problem \eqref{ABS-SPDE-strong}.
	\end{prop}

	Secondly, we state the following lemma  result which is a generalisation of   \cite[Lemmata III 6A and
	6B]{Elw_1982}.
	%\vspace{1mm}
	\begin{lem} \label{lem-amalgamation} \textbf{(The Amalgamation			Lemma) }
\begin{trivlist}
\item[(1)]
Let $\Delta$ be a  family of accessible stopping
		times taking  values in $[0, \infty ]$. Then a supremum  of $\Delta$, i.e., 
		$ \tau := \sup  \, {\Delta}  $, 
	 is an accessible stopping time with values in $[0, \infty ]$   and
		there exists an $\Delta $-valued  increasing sequence $ \{
		\alpha_{n} \}_{ n= 1 }^\infty $ such that
		$ \tau(\omega) = \lim_{n \to \infty} \alpha_n(\omega)$, for all $\omega \in \Omega$.
\item[(2)]
		Assume also that for each  $ \alpha \in \Delta $, $  I_{
			\alpha } : [ 0, \alpha ) \times \Omega \to V $ is an admissible
		process such that for all $ \alpha , \beta \in \Delta $ and
		every $t>0$,
		\begin{equation} \label{eqn-amalgamation_01} I_{ \alpha }(t)= I_{ \beta } (t) \mbox{ $\mathbb{P}$-a.s. on } \Omega_t( \alpha \wedge \beta) .
		\end{equation}
		Then, there exists an admissible process $ \mathbf{I}: [0, \tau )
		\times \Omega \to V $,
		such that every $t>0$,
		\begin{equation} \label{eqn-amalgamation_02}
		\mathbf{I}(t)= I_{ \alpha } (t) \; \mbox{$\mathbb{P}$-a.s. on} \; \Omega_t( \alpha ).
		\end{equation}
\item[(3)]		Moreover,\hdoubtz{This statement is not in Elworthy's book. Ay least I cannot find it.} %\begin{trivlist}\item[{\rm (i)}]
		if ${\tilde I} :[0,\tau) \times \Omega \to X$ is any process
		satisfying \eqref{eqn-amalgamation_02} then the process $\tilde I$
		is a version of the process $I$, i.e. for all $t\in [0,\infty) $
		\begin{equation} \label{eqn-amalgamation_03}
		\mathbb{ P}\left(\left\{\omega \in \Omega: t < \tau(\omega) , \;
		I(t,\omega)\not=  {\tilde I}(t,\omega) \right\} \right) =0 .
		\end{equation}
		In particular, if in addition ${\tilde I}$ is an admissible
		process,  then
		\begin{equation} \label{eqn-amalgamation_03'}
		\mathbf{I}=  {\tilde I}.
		\end{equation}
\end{trivlist}		
		%\item[{\rm (ii)}]\end{trivlist}
		
	\end{lem}
	\begin{Rem}\label{rem-amalgamation-equiv}
		Let us note that because both processes $ \mathbf{I}: [0, \tau )
		\times \Omega \to V $ and $ {I}_\alpha: [0, \alpha ) \times \Omega
		\to V $ are admissible (and hence with almost sure continuous
		trajectories), and since $\alpha \leq \tau$,  condition
		\eqref{eqn-amalgamation_02} is equivalent to the following one:
		\begin{equation} \label{eqn-amalgamation_02'}
		\mathbf{I}_{\vert [0,\alpha)\times \Omega }= I_{ \alpha } .
		\end{equation}
		Similarly,  condition  \eqref{eqn-amalgamation_01} is equivalent
		to the following one
		\begin{equation} \label{eqn-amalgamation_01'} {I_{ \alpha }}_{\vert [0,\alpha \wedge \beta )\times \Omega } = {I_{ \beta }}_{\vert [0,\alpha \wedge \beta )\times \Omega }.
		\end{equation}
			\end{Rem}
\begin{proof}[Proof of Lemma \ref{lem-amalgamation}] Let $\Delta$ be the family of accessible stopping
		times with values in $[0, \infty ]$. This set satisfies the assumptions of Lemma \cite[Lemma III.6A]{Elw_1982}, where the set $\Delta$ is denoted by $A$. Indeed,
by Remark \ref{rem-predictable stopping time},  the supremum of every finite subset of $\Delta$ belongs to $\Delta$. Therefore, there exists an $\mathcal{F}$-measurable function $\tau:\Omega\to [0,\infty]$ such that
\begin{trivlist}
\item[(i)] if $\sigma \in \Delta$, then $\tau\geq \sigma$,  $\mathbb{P}$-a.s.;
\item[(ii)] if a random variable $\eta:\Omega\to [0,\infty]$ satisfies $\tau\geq \sigma$,  $\mathbb{P}$-a.s., for all $\sigma \in \Delta$, then
$\eta\geq \tau$,  $\mathbb{P}$-a.s..
\item[(iii)] there exists a sequence $(\alpha_n)_{n \in \mathbb{N}}$ of elements of $\Delta$ such that 	 for all $\omega \in \Omega$, $\alpha_n(\omega) \leq \alpha_{n+1}(\omega) \leq \tau(\omega)$ for all $n\in \mathbb{N}$ and
 	$ \tau(\omega) = \lim_{n\to \infty} \alpha_n(\omega)= \sup_{n\in \mathbb{N}} \alpha_n(\omega)$.
\end{trivlist}
Moreover, $\tau$ is unique in the sense that if $\hat \tau$ satisfies the above conditions (i) and (ii), then $\hat \tau \geq \tau$,  $\mathbb{P}$-a.s..
Hence, since for every  $n \in \mathbb{N}$, $\alpha_n$,   is an accessible stopping time,  by \cite[Proposition III.5B]{Elw_1982}, Remark \ref{rem-predictable stopping time} and \cite[Proposition 4.11]{Metivier_1982} we infer that $\tau$ an accessible stopping time. This proves part (1) of Lemma \ref{lem-amalgamation}.

The proof of parts (2) and (3)  is the same as the proof of \cite[Lemma III 6 B]{Elw_1982}, so we omit it.
\end{proof}
		\begin{Def}\label{Def-maxsol}
		Consider  a family  $\mathcal{ LS}$ of all local solution
		$(u,\tau)$ to  the problem \eqref{ABS-SPDE-strong}. For two
		elements $(u,\tau), (v,\sigma) \in \mathcal{ LS} $ we write that
		$(u,\tau)\preceq (v,\sigma)$ iff $\tau \leq \sigma$ $\mathbb{P}$-a.s. and
		$v_{\vert [0,\tau)\times \Omega} \sim u$. Note that if
		$(u,\tau)\preceq (v,\sigma)$ and $(v,\sigma)\preceq (u,\tau)$,
		then  $(u,\tau)\sim (v,\sigma)$. We write $(u,\tau)\prec
		(v,\sigma)$ iff $(u,\tau)\preceq (v,\sigma)$ and $(u,\tau)\not\sim
		(v,\sigma)$. Then, the pair $(\mathcal{ LS},\preceq)$ is 
		partially ordered. 
		Each maximal element $(u,\tau)$ in the set $(\mathcal{
			LS},\preceq)$
		is called a maximal local solution to  the problem  \eqref{ABS-SPDE-strong}. The existence of an upper bound of  every non-empty chain of  $(\mathcal{
			LS},\preceq)$ is justified by  Amalgamation
		Lemma \ref{lem-amalgamation}. \\
		If $(u, \tau)$ is a  maximal local solution to equation
		\eqref{ABS-SPDE-strong}, the stopping time $\tau$ is called its
		lifetime.
	\end{Def}
	
		A priori, there may be many maximal elements in $(\mathcal{
			LS},\preceq)$ and hence many maximal local solutions    to
		the problem \eqref{ABS-SPDE-strong}. However, if  the uniqueness of local solutions holds,  then the
		uniqueness of the maximal local solution will follow.
	%%%%%%%%%%%%%%%%%%%%%%%%%%%%%%%%%%%%%%%
	\begin{Def}\label{Def-uniq} A local solution $(u,\tau)$
		to problem \eqref{ABS-SPDE-strong}  is unique iff for all other
		local solution $(v,\sigma)$  to \eqref{ABS-SPDE-strong} the
		restricted processes $u_{[0, \tau\wedge \sigma)\times \Omega}$ and
		$v_{[0, \tau\wedge\sigma)\times \Omega}$ are equivalent.
	\end{Def}
	
	\begin{prop}\label{prop-loc-implies-max}
		Suppose that $u_0$ is a $V$-valued random variable and $\mathcal{F}_0$-measurable. Assume that the following two conditions are satisfies:
\begin{trivlist}
\item[(i)]  there
		exist at least one local solution $(u^0,\tau^0)$ to problem
		\eqref{ABS-SPDE-strong}
\item[(ii)]
if  $(u^1,\tau^2)$ and $(u^2,\tau^2)$  are  local solutions,  then for  every $t>0$,
		\begin{equation} \label{eqn-amalgamation_04} u^1(t)= u^2 (t) \mbox{ $\mathbb{P}$-a.s. on } \Omega_t( \tau^1 \wedge \tau^2 ).
		%\{ t< \tau^1 \wedge \tau^2\} .
		\end{equation}
\end{trivlist}
Then, problem
problem \eqref{ABS-SPDE-strong} has a unique maximal  local solution { $(\hat{u},\hat{\tau})$}  satisfying  $(u^0,\tau^0)\preceq (\hat{u},\hat{\tau}) $. 
	\end{prop}
	
	\begin{Rem}\label{rem--loc-implies-max-equiv}
		Let us note that similarly to Remark \ref{rem-amalgamation-equiv},
		because both the local solutions $u^1$ and $u^2$ are admissible
		(hence with almost sure continuous trajectories),   condition
		\eqref{eqn-amalgamation_04} is equivalent to 
		\begin{equation} \label{eqn-amalgamation_04'}
		u^1_{\vert [0,\tau^1 \wedge \tau^2)\times \Omega }= u^2_{\vert
			[0,\tau^1 \wedge \tau^2)\times \Omega } .
		\end{equation}
	\end{Rem}
	\begin{proof}[Proof of Proposition \ref{prop-loc-implies-max}]
		Let us choose and fix a local solution $(u^0,\tau^0)$ to problem
		\eqref{ABS-SPDE-strong} and let us consider  the family  $\mathcal{ LS}$
of all local solution
		$(u,\tau)$ to  the problem \eqref{ABS-SPDE-strong} such that $(u^0,\tau^0)\preceq (u,\tau) $.
By assumptions this set is non-empty. Due to the assumptions (i) and (ii) of Proposition \ref{prop-loc-implies-max}, by the Amalgamation Lemma \ref{lem-amalgamation} we infer that there exists an accessible stopping time
			\[
		\hat{\tau} := \sup\left\{ \tau: (u,\tau)\in \mathcal{ LS}\right\}
		\]
		and an admissible process $\hat{u}: [0, \hat{\tau} )
		\times \Omega \to V $,
		such that  for all $(u,\tau)\in \mathcal{LS}$ and for $t>0$,
		\begin{equation} \label{eqn-amalgamation_02-a}
		\hat{u}(t)= u (t) \; \mbox{$\mathbb{P}$-a.s. on} \; \Omega_t( \tau).
		\end{equation}
Moreover, there exists  an increasing sequence $(\tau_n)$  of  accessible stopping times such that  $ \tau(\omega) = \lim_{n \to \infty} \tau_n(\omega)$, for all $\omega \in \Omega$.

In order to complete the proof of the existence of a maximal lcoal solution, we shall prove that $(\hat{u},\hat{\tau})\in \mathcal{LS}$. For this aim,  we closely follow the proof of \cite[Theorem 2.26]{Brz+Elw_2000}.
 Let us  define an auxiliary  process 	$\hat{\eta}=	\bigl(\hat{\eta}(t)\bigr),$ $t\in [0,\tau)$,  such that for each $n\in \mathbb{N}$ and $t\geq 0$, the following equality holds $\mathbb{P}$-a.s.
		\begin{equation}\label{eqn-proof-02}
	\begin{split}
	\hat{\eta}(t \wedge \tau_n)=S(t\wedge \tau_n)u_0 +\int_0^{t\wedge \tau_n} S(t-s) F(\hat{u}(s\wedge \tau_n) )\, ds +I_{\tau_n}(t\wedge \tau_n),
	\end{split}
		\end{equation}
where  $I_{\tau_n}$ is a continuous $V$-valued process process defined by
			\begin{equation}\label{eqn-proof-03}
			\begin{split}
I_{\tau_n}(t):= \int_0^t\mathds{1}_{[0,\tau_n)}(s){S}(t-s) G(\hat{u}(s \wedge \tau_m))\, 			d{W}(s),\;\; t\geq 0.
			\end{split}
			\end{equation}
Assume that $(u,\tau)\in \mathcal{LS}$.  	 Define  a  process $\eta=\bigl(\eta(t)\bigr)$, $t\in [0,\tau)$ by the above formulae \eqref{eqn-proof-02}-\eqref{eqn-proof-03} with
$\hat{u}$ replaced by $u$ and the announcing sequence $(\tau_n)$ of the accessible stopping time $\hat{\tau}$ replaced  by  announcing sequence of the accessible stopping time $\tau$.
Because $(u,\tau)$ is a local solution, we infer that  the process
$\eta(t)$, $t\in [0,\tau)$
 is a version of the process  $u(t)$, $t \in [0,\tau)$. Since $\hat{u}$ satisfies \eqref{eqn-amalgamation_02-a} and assumption (ii) of Proposition \ref{prop-loc-implies-max}  is satisfied,
 {we infer that
\begin{equation} \label{eqn-amalgamation_02-b}
		\hat{\eta}(t)= u (t) \; \mbox{$\mathbb{P}$-a.s. on} \; \Omega_t( \tau).
		\end{equation} }
 Hence, by the   part (3) of Lemma \ref{lem-amalgamation}, we infer that the process $\hat{\eta}(t)$, $t\in [0,\hat{\tau})$, is a version of the process $\hat{u}(t)$, $t\in [0,\hat{\tau})$ and therefore we can replace $\hat{\eta}$ by $\hat{u}$  on the LHS of \eqref{eqn-proof-02}. Therefore, we deduce that  $(\hat{u}, \hat{\tau})\in \mathcal{LS}$. This completes the existence of a local maximal solution.

As a byproduct of the above proof of the existence of   a local maximal solution  $(\hat{u}, \hat{\tau})$ we showed that $(\hat{u}, \hat{\tau})\in \mathcal{LS}$. This, in conjunction with the definition of 
	$\mathcal{LS}$ implies that  $(u^0,\tau^0)\preceq (\hat{u},\hat{\tau}) $.

It remains to prove the uniqueness of the local maximal solutions. For this aim let us suppose that $(u^1,\tau^1)$ and $(u^2,\tau^2)$ are two local maximal solutions. Let us put $\tilde{\tau} = \tau^1\vee \tau^2$.  Then, by part (ii) of Remark \ref{rem-predictable stopping time} $\tilde{\tau}$ is an accessible stopping time with announcing sequence  $(\tilde{\tau}_n:=\tau^1_n \vee \tau^2_n)_{n \in \mathbb{N}}$, where
$(\tau^i_n)_{n \in \mathbb{N}}$, $i=1,2$ is an announcing sequence of $\tau^i$.  By the uniqueness assumption (ii), we infer that
\begin{equation}\label{eqn-equivalence}
({u_1}_{\lvert _{[0,\tau^1\wedge \tau^2) }},\tau^1\wedge \tau^2) \sim ({u_2}_{\lvert _{[0,\tau^1\wedge \tau^2) } },\tau^1\wedge \tau^2).
\end{equation}
 We shall now  prove     that $\tau^1= \tau^2$ $\mathbb{P}$-a.s..  Suppose by contradiction  that
		$\mathbb{P} (\{\tau^1\neq \tau^2 \} )>0. $
		Let  $\Omega_1:=  \{\tau^1 \ge  \tau^2 \}$ and $\Omega_2:=\{ \tau^2 > \tau^1 \}$.
		We define a process $(\tilde{u}, \tilde{\tau})$ by the following formula
		\begin{equation}\label{eqn-Def-tildeu}
	\tilde{u} (t,\omega)=
		\begin{cases}
		u^1(t,\omega) \text{ if } \omega \in \Omega_1 \text{ and }  t\in [0, \tau^1(\omega) )\\
		u^2(t,\omega) \text{ if } \omega \in \Omega_2 \text{ and } t\in [0,\tau^2(\omega)). \\
		\end{cases}
		\end{equation}
We  now claim that the process  $(\tilde{u}, \tilde{\tau})$ is a local solution to Problem \eqref{ABS-SPDE-strong}. Let us fix $n \in \mathbb{N}$ and $t\ge 0$. By symmetry, we can assume that $\tau^1_n(\omega)\le \tau^2_n(\omega)$ for all  $\omega \in \Omega$ and  $n\in \mathbb{N}$.  Firstly, the proof of the admissibility of $\tilde{u}(t),\; t \in [0,\tilde{\tau})$ is very similar to the proof in \cite[Corollary 2.28]{Brz+Elw_2000}.
Secondly, let us also observe that on $ \Omega_1$ we have $\tilde{\tau}_n< \tau^1\wedge \tau^2$. Hence, we deduce from \eqref{eqn-Def-tildeu} and \eqref{eqn-equivalence} that
\begin{equation*}
\begin{split}
\tilde{u}(t\wedge\tilde{\tau}_n)=&\tilde{u}(t\wedge \tau^2_n)=u^1(t\wedge \tau^2_n)=u^2(t\wedge\tau^2_n)\\
=&S(t\wedge\tau^2_n)u_0 +\int_0^{t\wedge\tau^2_n} S(t\wedge\tau^2_n-s) F(u^2(s))\, ds
+ I^2_{\tau_n^2}(t \wedge \tau^2_n)\;\; t\geq 0,
\end{split}
\end{equation*}
where
\[
			I^2_{\tau_n^2}(t):= \int_0^t\mathds{1}_{[0,\tau_n^2)}(s){S}(t-s) G(u^2(s \wedge \tau_n^2))\, 			d{W}(s),\;\; t\geq 0.
			\]

The last equality follows from the fact that $(u^2,\tau^2)$ is a local solution. Since $\tau_n^2=\tilde{\tau}_n$, by  using \eqref{eqn-equivalence} and  \cite[Proposition 2.10]{Brz+Elw_2000} we deduce that
\begin{equation*}
\begin{split}
\tilde{u}(t\wedge\tilde{\tau}_n)
=&S(t\wedge\tilde{\tau}_n)u_0 +\int_0^{t\wedge\tilde{\tau}_n} S(t\wedge\tilde{\tau}_n-s) F(\tilde{u}(s))\, ds
+ \tilde{I}_{\tilde{\tau_n}}(t\wedge \tilde{\tau_n}), \;\; t\geq 0,
\end{split}
\end{equation*}
where
\[
\tilde{I}_{\tilde{\tau_n}}(t):=
\int_0^t 1_{[0,\tilde{\tau}_n)} S(t-s) G(\tilde{u}(s\wedge\tilde{\tau}_n)) dW(s), \;\; t\geq 0.
\]
Hence, $(\tilde{u},\tilde{\tau})$ satisfies equation \eqref{eq-locsol_00-b} on $\Omega_1$. In a similar way, we can also show that `$(\tilde{u},\tilde{\tau})$ satisfies \eqref{eq-locsol_00-b} on $\Omega_2$. Hence, $(\hat{u}, \hat{\tau})\in \mathcal{LS}$.
%\begin{equation*}
%\begin{split}
%\tilde{u}(t\wedge\tilde{\tau}_n)=&u^2(t\wedge\tau^2_n)=S(t\wedge\tau^2_n)u_0 +\int_0^{t\wedge\tau^2_n} S(t\wedge\tau^2_n-s) F(u^2(s))\, ds + I^2_{\tau_n^2}(t \wedge \tau^2_n) ,\\
%=&S(t\wedge\tilde{\tau}_n)u_0 +\int_0^{t\wedge\tilde{\tau}_n} S(t\wedge\tilde{\tau}_n-s) F(\tilde{u}(s))\, ds
%+ \tilde{I}_{\tilde{\tau_n}}(t\wedge \tilde{\tau_n}), \;\; t\geq 0.
%\end{split}
%\end{equation*}
%We infer that  $(\tilde{u},\tilde{\tau})$ satisfies \eqref{eq-locsol_00-b} on $\Omega_2$. Hence, $(\hat{u}, \hat{\tau})\in \mathcal{LS}$.

Now, by construction we have $(u^i, \tau^i) \preceq (\tilde{u}, \tilde{\tau})$ for $i\in \{1,2\}$ and there exists $i_0 \in \{1,2\}$ such that
$ (u^{i_0},\tau^{i_0}) \not \sim (\tilde{u},\tilde{\tau} )$. This contradicts the maximality of $(u^{i_0},\tau^{i_0})$ and completes the proof of Proposition \ref{prop-loc-implies-max}.

	\end{proof}
As a byproduct of the proof of the above Proposition \ref{prop-loc-implies-max}  we deduce the following general result.
\begin{cor}\label{cor-max of two local solutions} Let $(\mathbf{x},\sigma)$ and  $(\y,\tau)$
be two  local
		solution to problem \eqref{ABS-SPDE-strong}  such that   for  every $t>0$,
		\begin{equation} \label{eqn-amalgamation_05} \y(t)= \mathbf{x} (t) \mbox{ $\mathbb{P}$-a.s. on } \Omega_t( \tau \wedge \sigma ).
	\end{equation}
Then the process 		 $(\z, \sigma \vee \tau )$ defined by the following formula
		\begin{equation}\label{eqn-Def-tildeu-2}
	\z (t,\omega)=
		\begin{cases}
		\mathbf{x}(t,\omega), \text{ if }  \sigma(\omega)\geq \tau(\omega) \text{ and }  t\in [0, \sigma(\omega) ),\\
		\y(t,\omega), \text{ if }  \sigma(\omega)< \tau(\omega) \text{ and } t\in [0,\tau(\omega)),
		\end{cases}
		\end{equation}
is local
		solution to problem \eqref{ABS-SPDE-strong}. The process $(\z, \sigma \vee \tau )$ is called  supremum of $(\mathbf{x},\sigma)$ and  $(\y,\tau)$.
\end{cor}
	\subsection{An abstract result}
	In this subsection we prove by a fixed point method some results about  the 
	existence and uniqueness of maximal local mild solution to \eqref{ABS-SPDE-1}. 
%	 the following abstract SPDEs
%	\begin{equation}\label{ABS-SPDE}
%	u(t)=S(t)u_0+\int_0^t S(t-s) F(u(s))\, ds+\int_0^t S(t-s) G(u(s))
%	dW(s),\;\; \mbox{ for all }t>0,
%	\end{equation}
%	where $S(t)$, $t\in [0,\infty)$ is a semigroup, $F$ and $G$ are
%	nonlinear map satisfying Assumption \ref{assum-01}, Assumption
%	\ref{assum-F}, and Assumption \ref{assum-G}, respectively.
	
	Let
	$\theta:\mathbb{R}_+\to [0,1]$ be a ${\mathcal C}^\infty_c$ non
	increasing function
	such that
	\begin{equation}\label{eqn-theta} \inf_{x\in\mathbb{R}_+}\theta^\prime(x)\geq -1, \quad \theta(x)=1\;
	\mbox{\rm  iff } x\in [0,1]\quad \mbox{\rm  and } \theta(x)=0 \;
	\mbox{\rm  iff } x\in [2,\infty).
	\end{equation}
	and for $n\geq 1$ set  $\theta_n(\cdot)=\theta(\frac{\cdot}{n})$.
	Note that if $h:\mathbb{R}_+\to\mathbb{R}_+$ is a non decreasing
	function, then
	\begin{equation}\label{ineq-theta}
	\theta_n(x)h(x) \leq h(2n),\quad
	%\\\label{ineq-Lip-theta}
	\vert \theta_n(x)-\theta_n(y)\vert \leq \frac1n |x-y|, \text{ for every $x,y\in {\mathbb R}$}. %, \;\; \mbox{ for all } x_1,x_ 2\in\mathbb{R}.
	\end{equation}

 %We claim that
%		\begin{eqnarray}\label{eqn-global Lipschitz-F-new}
%		\vert \Phi_T^n(u_1)-\Phi_T^n(u_2)\vert_{L^2(0,T;H)} &\leq & C\Big[   2n C +1\Big]
%\\
%	&&	\Big[(2n)^{p+1} (T-\delta)^{(1-\alpha)/2} +(2n)^{q+1} (T-\delta)^{(1-\gamma)/2} \Big]\vert  u_1 -u_2
%		\vert_{X_T}.
%\nonumber		\end{eqnarray}
%for all 		$u_1,u_2 \in \hat{X}_{\delta,T,\ba}$.
%
%}

	\begin{prop}\label{prop-global Lipschitz-F}
		Let $F$ be a mapping satisfying Assumption
		\ref{assum-F}. Assume that  $\delta\in [0,T]$, $\ba \in X_\delta$. Then the  map
		\[
		\Phi_{{\delta,T,\ba}}^n: \hat{X}_{\delta,T,\ba} \ni u \mapsto \theta_n( \vert u
		\vert_{X_\cdot}) F(u) \in L^2(0,T;H).
		\]
	is globally Lipschitz and moreover, for all
		$u_1,u_2 \in \hat{X}_{\delta,T,\ba}$,
		\begin{equation}\label{eqn-global Lipschitz-F}
		\vert \Phi_{\delta,T,\ba}^n(u_1)-\Phi_{\delta,T,\ba}^n(u_2)\vert_{L^2(0,T;H)} \leq   		C( C +1)
\sum_{i=1}^N (2n)^{p_i+2} (T-\delta)^{(1-\alpha_i)/2}
\vert  u_1 -u_2
		\vert_{X_T}.
%\nonumber
		\end{equation}
	
In particular, the Lipschitz constant of $\Phi_{{\delta,T,\ba}}^n$ is independent of $\ba$.
	\end{prop}

	The proof is based on a proof from
	\cite{Brz+Millet_2012} which in turn was based on a proof from \cite{deBouard+Deb_1999,deBouard+Deb_2003}.
	For simplicity of notation, below we will write $\Phi_{T}$ instead of $\Phi_{\delta,T,\ba}^n$.

	\begin{proof}[Proof of Proposition \ref{prop-global Lipschitz-F}] Wlog we can assume that $N=1$ and we will use notation $\alpha=\alpha_1$ and $p=p_1$. In this case,  the inequality
\eqref{eqn-global Lipschitz-F} takes the following form. For every $n\in \mathbb{N}$ there exists $C(n)>0$ such that for all $T>\delta\geq 0$, all $\ba \in X_\delta$ and
all $u_1,u_2 \in \hat{X}_{\delta,T,\ba}$,
		\begin{equation}\label{eqn-global Lipschitz-F-simple}
		\vert \Phi_{\delta,T,\ba}^n(u_1)-\Phi_{\delta,T,\ba}^n(u_2)\vert_{L^2(0,T;H)} \leq   		C(n)  (T-\delta)^{(1-\alpha)/2}
\vert  u_1 -u_2
		\vert_{X_T}.
		\end{equation}
In what follows we will prove \eqref{eqn-global Lipschitz-F-simple}.
Let us fix $n\in \mathbb{N}$, $T>\delta\geq 0$, $\ba \in X_\delta$ and  $u_1,u_2 \in \hat{X}_{\delta,T,\ba}$,

Note that $\Phi_T(0)=0$. Assume that $u_1,u_2 \in X_T$. Denote,
		for $i=1,2$,
		\[
		\tau_i= \inf\{t \in [0,T]: \vert u_i \vert_{X_t} \geq 2n\}.
		\]
		Note that if the set on the RHS above is empty, i.e. $\vert u_i \vert_{X_t} < 2n$ for all $t\in [0,T]$,
		then $\tau_i=T$.
		
		Wlog we can assume that $\tau_1 \leq
		\tau_2$. Because for $i=1,2, \;\; \theta_n( \vert u_i
		\vert_{X_t}) =0 \mbox{ for } t \geq \tau_2$, we  have 
	\begin{eqnarray*}
		%	\vert \Phi_T(u_1)-\Phi_T(u_2)\vert_{L^2(0,T;H)} &=& \Big[ \int_0^T \vert  \theta_n( \vert u_1 \vert_{X_t}) F(u_1(t))-\theta_n( \vert u_2 \vert_{X_t}) F(u_2(t))\vert_H^2\,dt\Big]^{1/2}
		%	\\
		\vert \Phi_T(u_1)-\Phi_T(u_2)\vert_{L^2(0,T;H)} 	&=& \Big[ \int_0^{\tau_2} \vert  \theta_n( \vert u_1 \vert_{X_t})
			F(u_1(t))-\theta_n( \vert u_2 \vert_{X_t})
			F(u_2(t))\vert_H^2\,dt\Big]^{1/2}
			\\
%			&&\hspace{-5truecm}\lefteqn{= \Big[ \int_0^{\tau_2} \vert \big[
%				\theta_n( \vert u_1 \vert_{X_t}) - \theta_n( \vert u_2
%				\vert_{X_t}) \big] F(u_2(t))+ \theta_n( \vert u_1
%				\vert_{X_t})\big[  F(u_1(t))- F(u_2(t)) \Big]
%				\vert_H^2\,dt\Big]^{1/2} }
%			\\
			& &\hspace{-3truecm}\lefteqn{ \leq
				\Big[ \int_0^{\tau_2} \vert  \big[ \theta_n( \vert u_1 \vert_{X_t}) - \theta_n( \vert u_2 \vert_{X_t}) \big] F(u_2(t))\vert_H^2 \,dt \Big]^{1/2} }\\
			&&\hspace{-2truecm}\lefteqn{+ \Big[ \int_0^{\tau_2} \vert
				\theta_n( \vert u_1 \vert_{X_t})\big[  F(u_1(t))- F(u_2(t)) \big]
				\vert_H^2\,dt\Big]^{1/2} =:A+B}
		\end{eqnarray*}
		Next, since $\theta_n$ is Lipschitz with Lipschitz constant $2n$ and ${u_1}_{\lvert_{[0,\delta]}}={u_2}_{\lvert_{[0,\delta]}}=\ba$
		we have
		\begin{eqnarray*}
		%	A^2&=& \int_0^{\tau_2} \vert  \big[ \theta_n( \vert u_1 \vert_{X_t}) - \theta_n( \vert u_2 \vert_{X_t}) \big] F(u_2(t))\vert^2 \,dt\\
A^2 &=& \int_0^{\delta \wedge \tau_2} \vert  \big[ \theta_n( \vert u_1 \vert_{X_t}) - \theta_n( \vert u_2 \vert_{X_t}) \big] F(u_2(t))\vert^2 \,dt
+ \int_{\delta \wedge \tau_2}^{ \tau_2} \vert  \big[ \theta_n( \vert u_1 \vert_{X_t}) - \theta_n( \vert u_2 \vert_{X_t}) \big] F(u_2(t))\vert^2 \,dt
\\
&=& \int_{\delta \wedge \tau_2}^{ \tau_2} \vert  \big[ \theta_n( \vert u_1 \vert_{X_t}) - \theta_n( \vert u_2 \vert_{X_t}) \big] F(u_2(t))\vert^2 \,dt
\le 4n^2 C^2 \int_{\delta \wedge \tau_2}^{ \tau_2}  \big[ \vert
			\vert u_1 \vert_{X_t} -  \vert u_2 \vert_{X_t} \vert\big]^2 \vert F(u_2(t))\vert_H^2 \,dt\\
		%	&&\mbox{by Minkowski inequality}\\
			&\leq& 4n^2 C^2 \int_{\delta \wedge \tau_2}^{ \tau_2}  \vert
			u_1 -u_2 \vert_{X_t}^2 \vert F(u_2(t))\vert_H^2 \,dt \leq  4n^2 C^2\vert
			u_1 -u_2 \vert_{X_T}^2 \int_{\delta \wedge \tau_2}^{ \tau_2}   \vert F(u_2(t))\vert_H^2 \,dt.
%			&\leq & 4n^2 C^2 \vert
%			u_1 -u_2 \vert_{X_T}^2 \int_{\delta \wedge \tau_2}^{ \tau_2}    \vert F(u_2(t))\vert_H^2 \,dt
		\end{eqnarray*}
		Next, by assumptions and some elementary calculations 
		\begin{align*}
		\int_{\delta \wedge \tau_2}^{ \tau_2}   \vert F(u_2(t))\vert_H^2 \,dt  \leq & C^2  \int_{\delta \wedge \tau_2}^{ \tau_2}  \Vert  u_2(t)\Vert^{2p+2-2\alpha}  \vert  u_2(t)\vert_E^{2\alpha} \,dt\\ 
		\leq & C^2 \sup_{t \in [\delta\wedge \tau_2,\tau_2]} \Vert u_2(t)\Vert^{2p+2-2\alpha}
		\big( \int_{\delta \wedge \tau_2}^{ \tau_2}  \vert u_2(t)\vert_E^{2} \,dt\big)^\alpha
		(\tau_2-\delta\wedge \tau_2)^{1-\alpha}\\
		\leq & C^2  (T-\delta)^{1-\alpha} \vert u_2 \vert_{X_{\delta\wedge\tau_2, \tau_2}}^{2p+2} \leq  C^2  (T-\delta)^{1-\alpha}(2n)^{2p+2}.
		\end{align*}
		Therefore,
		\begin{equation*}
			A\le C  (T-\delta)^{(1-\alpha)/2}(2n)^{p+2} \vert  u_1 -u_2
			\vert_{X_T} .
		\end{equation*}
	Since $\tau_1 \leq \tau_2$, $ \theta_n( \vert u_1 \vert_{X_t}) =0$, $ t \geq
		\tau_1$,  $\theta_n( \vert u_1 \vert_{X_t}) \leq 1,\, t\in [0,\tau_1)$ and  ${u_1}_{\lvert_{[0,\delta]}}={u_2}_{\lvert_{[0,\delta]}}=\ba$, we have
		\begin{align*}
			B= \Big[ \int_0^{\tau_2} \vert   \theta_n( \vert u_1 \vert_{X_t})\big[  F(u_1(t))- F(u_2(t)) \big] \vert_H^2\,dt \Big]^{1/2}
		%	=& \Big[ \int_{\delta\wedge \tau_1}^{\tau_1} \vert   \theta_n( \vert u_1 \vert_{X_t})\big[  F(u_1(t))- F(u_2(t)) \big] \vert_H^2\,dt\Big]^{1/2}\\
			\leq& \Big[\int_{\delta\wedge \tau_1}^{\tau_1}  \vert     F(u_1(t))- F(u_2(t))  \vert_H^2\,dt\Big]^{1/2}
			\leq  \tilde{B}_{p},
		\end{align*}
		where
		\begin{equation*}
		\begin{split}
		\tilde{B}_{p}:=
		C \Big[ \int_{\delta\wedge \tau_1}^{\tau_1}  \vert u_1(t)-u_2(t)\vert_E^{2\alpha} \Vert u_1(t)-u_2(t)\Vert^{2-2\alpha} \Vert u_2(t)\Vert^{2p} \,dt\Big]^{1/2}\\
		+C \Big[ \int_{\delta\wedge \tau_1}^{\tau_1}     \Vert u_1(t)-u_2(t)\Vert^2 \Vert u_1(t)\Vert^{2p-2\alpha} \vert u_1(t)\vert_E^{2\alpha}   \,dt\Big]^{1/2}.
		\end{split}
		\end{equation*}
%		and
%		\begin{equation*}
%		\begin{split}
%		\tilde{B}_{,q}:=
%		C \Big[ \int_{\delta\wedge \tau_1}^{\tau_1}   \vert u_1(t)-u_2(t)\vert_E^{2\gamma} \Vert u_1(t)-u_2(t)\Vert^{2-2\gamma} \Vert u_2(t)\Vert^{2q} \,dt\Big]^{1/2}\\
%		+C \Big[\int_{\delta\wedge \tau_1}^{\tau_1}   \Vert u_1(t)-u_2(t)\Vert^2 \Vert u_1(t)\Vert^{2p-2\gamma} \vert u_1(t)\vert_E^{2\gamma}   \,dt\Big]^{1/2}.
%		\end{split}
%		\end{equation*}
This term can be estimated as follows
		\begin{align*}
%			\tilde{B}_{p} \leq& C \sup_{t\in [\delta\wedge \tau_1,\tau_1]} \Vert u_1(t)-u_2(t)\Vert \Vert
%			u_1(t)\Vert^{p-\alpha} \Big[ \int_{\delta\wedge \tau_1}^{\tau_1} \vert
%			u_1(t)\vert_E^{2\alpha}   \,dt\Big]^{1/2}\\
%			+& C \sup_{t\in [\delta\wedge \tau_1,\tau_1]} \Vert u_1(t)-u_2(t)\Vert^{1-\alpha}
%			\Vert u_2(t)\Vert^p \Big[ \int_{\delta\wedge \tau_1}^{\tau_1} \vert
%			u_1(t)-u_2(t)\vert_E^{2\alpha} \,dt\Big]^{1/2}
%			\\
		\tilde{B}_{p} 	\leq&  C \sup_{t\in [0,T]} \Vert u_1(t)-u_2(t)\Vert \sup_{t\in
				[\delta\wedge \tau_1,\tau_1]} \Vert u_1(t)\Vert^{p-\alpha} \Big[ \int_{\delta\wedge \tau_1}^{\tau_1}
			\vert u_1(t)\vert_E^{2}   \,dt\Big]^{\alpha/2}
			(\tau_1-\delta\wedge \tau_1)^{(1-\alpha)/2}
			\\
			+& C \sup_{t\in [0,T]} \Vert u_1(t)-u_2(t)\Vert^{1-\alpha}
			\sup_{t\in [\delta\wedge \tau_1,\tau_1]}  \Vert u_2(t)\Vert^p \Big[ \int_{\delta\wedge \tau_1}^{\tau_1}
			\vert u_1(t)-u_2(t)\vert_E^{2}   \,dt\Big]^{\alpha/2}
			(\tau_1-\delta\wedge \tau_1)^{(1-\alpha)/2}
			\\
			\leq&  C \vert u_1-u_2\vert_{X_T}  \vert u_1 \vert_{X_{\tau_1}}^p (T-\delta)^{(1-\alpha)/2} +C \vert u_1-u_2\vert_{X_T}  \vert u_2 \vert_{X_{\tau_1}}^p (T-\delta)^{(1-\alpha)/2}\\
		%	&\mbox{ because} \vert u_1 \vert_{X_{\tau_1}} \leq 2n \mbox{ and } \vert u_2 \vert_{X_{\tau_1}} \leq \vert u_2 \vert_{X_{\tau_2}}\leq 2n\\
			\leq&   C (T-\delta)^{(1-\alpha)/2} \vert u_1-u_2\vert_{X_T} \Big[
			\vert u_1 \vert_{X_{\tau_1}}^p  +   \vert u_2
			\vert_{X_{\tau_1}}^p\Big] \leq C (2n)^{p+1} (T-\delta)^{(1-\alpha)/2}
			\vert u_1-u_2\vert_{X_T}.
		\end{align*}
		Summing up,  we proved the following inequality 
		\begin{eqnarray*}
			\vert \Phi_T(u_1)-\Phi_T(u_2)\vert_{L^2(0,T;H)}\le C\Big[   2n C +1\Big] (2n)^{p+1}  \tau_2^{(1-\alpha)/2} \vert
			u_1 -u_2 \vert_{X_T},
%			&\leq & C^2[
%			\tau_2^{(1-\alpha)/2}(2n)^{p+2}+\tau_2^{(1-\gamma)/2}(2n)^{q+2}] \vert  u_1 -u_2 \vert_{X_T}\\
%			&& +C [(2n)^{p+1} \tau_1^{(1-\alpha)/2} +(2n)^{q+1} \tau_1^{(1-\gamma)/2} ]\vert u_1-u_2\vert_{X_T}\\
%			&=&
		\end{eqnarray*}
		which competes the proof of the proposition.
	\end{proof}
	The following result is a special case of Proposition
	\ref{prop-global Lipschitz-F} with $H=V$.
	\begin{cor}\label{cor-global-lip-G}
		Let $G$ be a nonlinear mapping satisfying Assumption
		\ref{assum-G}. Assume that $n\in \mathbb{N}$, $T>0$, $\delta \in [0,T]$ and  $\ba\in X_{0,\delta}$.
 Define a map $\hat{\Phi}_{{\delta,T,\ba}}^n$ by
		\begin{equation}\label{eqn-Phi_G}
		\hat{\Phi}_{{\delta,T,\ba}}^n: \hat{X}_{\delta,T,\ba} \ni u \mapsto \theta_n(
		\vert u \vert_{X_\cdot}) G(u) \in L^2(0,T;V).
		\end{equation}
		Then $\hat{\Phi}_{{\delta,T,\ba}}^n$ is globally Lipschitz and moreover, for all
		$u_1,u_2 \in \hat{X}_{\delta,T,\ba}$,
		\begin{eqnarray}\label{eqn-global LipschitzG}
		\vert \hat{\Phi}_{{\delta,T,\ba}}^n(u_1)-\hat{\Phi}_{{\delta,T,\ba}}^n(u_2)\vert_{L^2(0,T;V)} &\leq
		&  (2n)^{k+2}C_G(C_G +1)  {(T-\delta)^{(1-\beta)/2} }\vert  u_1
		-u_2 \vert_{X_T}.
		\end{eqnarray}
{In particular, the Lipschitz constant of $\hat{\Phi}_{{\delta,T,\ba}}^n$ is independent of $\ba$.}
	\end{cor}

	\begin{prop}\label{prop-Psi_T-Lipschitz}
	Assume that Assumptions \ref{assum-F} and  \ref{assum-01} hold.  Assume that $n\in \mathbb{N}$, $T>0$, $\delta \in [0,T]$ and  $\ba\in \mathscr{M}^2(X_{0,\delta})$.
	Then the map $\Psi^n_{\delta,T,\ba}$ defined by
		\begin{equation}\label{eqn-Psi_T}
\Psi_{\delta,T,\ba}^n: \mathscr{M}^2(\hat{X}_{\delta,T, \ba}) \ni u \mapsto  [S(\cdot)](\ba(0)) + S \ast \Phi_{\delta, T,\ba}^n
		(u)+ S\diamond \hat{\Phi}_{\delta,T,\ba}^n(u) \in \mathscr{M}^2(X_{T}),
		\end{equation}
		 is globally Lipschitz and moreover, for all $u_1,u_2
		\in \mathscr{M}^2(\hat{X}_{\delta,T, \ba})$,
%\hdoubtz{ZB2102: I changed the $\frac12$ power below. I hope without a mistake. }
		\begin{eqnarray}\label{eqn-global Lipschitz}
		\vert \Psi_{\delta,T,\ba}^{n}(u_1)-\Psi_{\delta,T,\ba}^{n}(u_2)\vert_{\mathscr{M}^2(X_T)}  \leq \hat{C}(n) \Big[\max_{1\le i \le N}{(T-\delta)^{1-\alpha_i}} \vee (T-\delta)^{1-\beta} \Big]^{\frac12}\vert  u_1 -u_2
		\vert_{\mathscr{M}^2(X_T)},
\nonumber
		\end{eqnarray}
		where $\hat{C}(n)$ is dependent only on $n$ and is given by, for some $D>1$,
		\[ \hat{C}(n)= C_1 C_F (C_F +1) \sum_{i=1}^N(2n)^{p_i+2} + C_2C_G (2n)^{k+2}(  C_G +1)
\leq C_3n^{D}.
		\]
In particular, the Lipschitz constant  of $\Psi_{\delta,T,\ba}^{n}$ is independent of $\ba$.
	\end{prop}
	\begin{proof}[Proof of Proposition \ref{prop-Psi_T-Lipschitz}]

For simplicity of notation we will write
		$\Psi_{T}$ instead of $\Psi_{\delta,T,\ba}^n$. We will also write $\Phi_F$ (resp. $\hat{\Phi}_G$) instead of $\Phi^n_{\delta,T,\ba}$ (resp. $\hat{\Phi}^n_{\delta,T,\ba}$).
		Obviously in view of Assumption \ref{assum-01} the map $\Psi_T$ is
		well defined. Let us fix  $u_1,u_2 \in \mathscr{M}^2(X_{\delta,T,\ba})$. Then by the Fubini Theorem, Assumption \ref{assum-01}, Proposition \ref{prop-global Lipschitz-F} and Corollary \ref{cor-global-lip-G} we infer that
		\begin{eqnarray*}
			\vert \Psi_T(u_1)-\Psi_T(u_2)\vert_{\mathscr{M}^2(X_T)} &\leq & \vert S\ast \Phi_F(u_1)-S\ast \Phi_F(u_2)\vert_{\mathscr{M}^2(X_T)}
			+ \vert S \diamond  \Phi_G(u_1)-S\diamond \Phi_G(u_2)\vert_{\mathscr{M}^2(X_T)}\\
			&\leq& C_1 \vert  \Phi_F(u_1)- \Phi_F(u_2)\vert_{\mathscr{M}^2(0,T;H)}+C_2
			\vert  \Phi_G(u_1)-\Phi_G(u_2)\vert_{\mathscr{M}^2(X_T)}
			\\
%			&\leq&\Big[ C_1 C_F (C_F +1)\sum_{i=1}^N (2n)^{p_i+2} (T-\delta)^{(1-\alpha_i)/2}\\
%			&& \qquad + C_2C_G (2n)^{k+2}(   C_G +1) (T-\delta)^{(1-\beta)/2} \Big] \vert  u_1 -u_2
%			\vert_{\mathscr{M}^2(X_T)}
%			\\
			&\leq& \hat{C}(n)\Big[\max_{1\le i \le N}{(T-\delta)^{1-\alpha_i}} \vee (T-\delta)^{1-\beta} \Big]^{\frac12}\vert  u_1 -u_2
			\vert_{\mathscr{M}^2(X_T)}.
		\end{eqnarray*}
		The proof is complete.
	\end{proof}
Since our method is based on finding  fixed points of $\Psi_{\delta,T,\ba}^n$, the following auxiliary result is useful.
\begin{lem}\label{lem-fixed points}
Assume that $n\in \mathbb{N}$ and  $T>S>  0$ and $\bx \in V$.   Assume that $\ba\in \mathscr{M}^2(\hat{X}_{0,S,\bx})$ is a fixed point of
$\Psi_{0,S,\bx}^n$. Then  $\Psi_{S,T,\ba}^n$ maps  $\mathscr{M}^2(\hat{X}_{S,T, \ba})$ into itself.
\end{lem}
\begin{proof}[Proof of Lemma \ref{lem-fixed points}] Let us choose and fix $n\in \mathbb{N}$,  $T>S>  0$, $\bx \in V$ and  $\ba\in \mathscr{M}^2(\hat{X}_{0,S,\bx})$,  a fixed point of
$\Psi_{0,S,\bx}^n$.
We will show that $\Psi^n_{S,T,\ba }$ maps $\mathscr{M}^2(\hat{X}_{S,T,\ba})$ into itself.
Take an arbitrary   $u\in \hat{X}_{S,T,\ba}$. Since by Proposition \ref{prop-Psi_T-Lipschitz}, $v:=\bigl[\Psi_{S,T,\ba}^n \bigr](u) \in \mathscr{M}^2(X_T)$ we only need to show that
$v_{|[0,S]}=\ba$. For this aim let us observe that by Definition \ref{eqn-Psi_T}, we have for $t\in [0,S]$,
\begin{align*}
  v(t)= S(t)(\ba(0)) &+ \int_0^t S(t-r) \theta_n(\lvert u\rvert_{X_r}) F(u(r))\, dr
  + \int_0^t S(t-r) \theta_n( \lvert u \rvert_{X_r}) G(u(r))\, dW(r).
 \end{align*}
Because  $u_{|[0,S]}=\ba$, $\ba(0)=\mathbf{x}$ and,  by assumptions $\Psi_{0,S,\bx}^n(\ba)=\ba$ in $\mathscr{M}^2(\hat{X}_{0,S, \mathbf{x}})$, we infer that
\begin{align*}
  v(t)= S(t)\bx &+ \int_0^t S(t-r) \theta_n(\lvert a\rvert_{X_r}) F(a(r))\, dr
  + \int_0^t S(t-r) \theta_n( \lvert a \rvert_{X_r}) G(a(r))\, dW(r)\\
  &=[\Psi_{0,S,\bx}^n(\ba)](t)=\ba(t), \;\;\;  t\in [0,S].
 \end{align*}
This completes the proof of Lemma \ref{lem-fixed points}.
 \end{proof}

\begin{comment}
Before proceeding further, we state the following remark.
\begin{Rem}\label{Rem-Global-Lipshitz}
	\begin{trivlist}
		\item[(i)] Let $\ba_0\in V$. Then, the results in Proposition  \ref{prop-global Lipschitz-F} and Corollary \ref{cor-global-lip-G} remain true with the space $\hat{X}_{\delta,T,\ba}$ replaced by
		$\hat{X}_{0,T,\ba_0}$.
		\item[(ii)] Let $\ba_0\in L^2(\Omega,\mathcal{F}_0;V)$ and $n \in \mathbb{N}$. Then, it follows from  item (i) and the proof of Proposition \ref{prop-Psi_T-Lipschitz} that the  map  $\Psi_{0,T,\ba_0}^n$ defined by
		\begin{equation}\label{eqn-Psi_T-0}
		\Psi_{0,T,\ba_0}:=\Psi_{0,T,\ba_0}^n: \mathscr{M}^2(\hat{X}_{0,T, \ba_0}) \ni u \mapsto  S\ba_0 + S \ast \Phi_{0, T,\ba_0}^n
		(u)+ S\diamond \hat{\Phi}_{0,T,\ba_0}^n(u) \in \mathscr{M}^2(X_{T}),
		\end{equation}
		is globally Lipschitz with the Lipschitz constant being independent of  $\ba_0$.
	\end{trivlist}
	In both items (i) and (ii), inequalities \eqref{eqn-global Lipschitz-F}, \eqref{eqn-global LipschitzG}  and \eqref{eqn-global Lipschitz} hold with $T-\delta$ replaced by  $T$.
	
	\end{Rem}
\end{comment}

	%%%%%%%%%%%
	%%%%%%%%%%
	%%%%%%%%%%%%
	The 	first two main results of this subsection are  given in the following two
	theorems.
	\begin{thm}\label{Thm:LocalUniqueness}
		Suppose that Assumption \ref{assum-F}-Assumption \ref{assum-01} hold. Let $u_0$ be a	$\mathcal{F}_{0}$-measurable $V$-valued square integrable random
		variable $u_0$ and $(u,\tau) $ and $(v,\sigma)$  two local solutions of \eqref{ABS-SPDE-1}. Then,
	\begin{equation}
	(u_{\lvert_{[0,\sigma \wedge \tau)\times \Omega } }, \sigma \wedge \tau)\sim (v_{\lvert_{[0,\sigma \wedge \tau)\times \Omega}},\sigma \wedge \tau).
	\end{equation}
	\end{thm}
\begin{proof}[Proof of Theorem \ref{Thm:LocalUniqueness}]
	Wlog we can  assume  in Assumptions \ref{assum-F} and \ref{assum-G} that $N=1$ and $p_i=k$ and will use the notations $\alpha_1=\alpha$, $p_1=p=k$.
	Let $(u,\tau) $ and $(v,\sigma)$ be  two local solutions of \eqref{ABS-SPDE-1}. Let $(\tau_n)_{n \in \mathbb{N}}$ and $(\sigma_{n \in \mathbb{N}})$ be the announcing sequences of $\tau$ and $\sigma$, respectively. By  \cite[Propositions 4.3 \& 4.11 and Theorem 6.6]{Metivier_1982} the stopping time $\varrho:=\tau \wedge \sigma$ is accessible and it is easy to show that $(\varrho_n)_{n \in \mathbb{N}}:= (\tau_n \wedge \sigma_n)_{n \in \mathbb{N}}$ is an announcing sequence of $\varrho$.
	
	Hereafter we fix $n\in \mathbb{N}$. Since $(v,\sigma)$ is a local solution to  \eqref{ABS-SPDE-strong} and $\varrho_n \le \sigma_n$, by Corollary \ref{cor-A.2} we infer that for all $t\ge 0$ $\mathbb{P}$-a.s.
	\begin{align}
	v(t\wedge \varrho_n)= & S_{t\wedge \varrho_n} u_0 + \int_0^{t\wedge \varrho_n} S_{t\wedge \varrho_n-r}F(v(r)) dr + I_{\sigma_n}(t\wedge \varrho_n) \nonumber \\
	=&  S_{t\wedge \varrho_n} u_0 + \int_0^{t\wedge \varrho_n} S_{t\wedge \varrho_n-r}F(v(r)) dr + I_{\varrho_n}(t\wedge \varrho_n),\label{Eq:vstopped is a local sol}
	\end{align}
	where $$  I_{\sigma_n}(t):=\int_0^t 1_{[0,\sigma_n] }(r) S_{t-r}G(v(r))dW(r),\; t\ge 0 .$$
	The identity \eqref{Eq:vstopped is a local sol} proves that $(v,\sigma\wedge \tau)$ is a local solution to \eqref{ABS-SPDE-strong}. In a similar way, we prove that $(u, \sigma\wedge \tau)$ is a local solution  to \eqref{ABS-SPDE-strong} as well.
	
	Thirdly,  for $k \in \mathbb{N}$ we put $\varrho_{n,k}=\tau_n \wedge \sigma_n\wedge \tilde{\tau}_k \wedge \tilde{\sigma}_k$, where
	\begin{align}
\tilde{\tau}_k=\inf\{t \in [0,T]: \lvert v \rvert_{X_t} \ge k  \} \wedge \tau \text{ and  } 	\tilde{\sigma}_k=\inf\{t \in [0,T]: \lvert v \rvert_{X_t} \ge k  \} \wedge \sigma, \; k \in \mathbb{N}.\nonumber
	\end{align}
We observe that for all $k\in \mathbb{N}$ $\varrho_{n,k}\le \varrho_{n, k+1}\le \sigma_n \wedge \tau_n$ $\mathbb{P}$-a.s. and $\varrho_{n,k} \toup \sigma_n \wedge \tau_n$ $\mathbb{P}$-a.s. if  $k \to \infty$.  Let us now fix $k \in \mathbb{N}$.  Arguing as in the proof of \eqref{Eq:vstopped is a local sol} we show that for all $t\ge 0$, $\mathbb{P}$-a.s.
	\begin{align}
	v(t\wedge \varrho_{n,k})
	=&S_{t\wedge \varrho_{n,k}} u_0 + \int_0^{t\wedge \varrho_{n,k}} S_{t\wedge \varrho_{n,k}-r}F(v(r)) dr + I_{\varrho_{n,k}}(t\wedge \varrho_{n,k}).
	\end{align}
	In a similar way, we prove that  the same identity holds with $v$ replaced by $u$. Hence, setting $w=u-v$ we infer that for all $t\ge 0$, $\mathbb{P}$-a.s.
	\begin{equation}
	w(t\wedge \varrho_{n,k})
	= \int_0^{t\wedge \varrho_{n,k}} S_{t\wedge \varrho_{n,k}-r}[F(u(r))- F(v(r))] dr + \tilde{I}_{\varrho_{n,k}}(t\wedge \varrho_{n,k}).
	\end{equation}
	where
	$$  \tilde{I}_{\varrho_{n,k}}(t):=\int_0^t 1_{[0,\varrho_{n,k}] }(r) S_{t-r}[G(u(r))-G(v(r))]dW(r),\; t\ge 0 .$$
	
	Hereafter, $c>0$ denotes an universal constant (independent of $n$ and $k$) which may change from one term to the other.
Following the lines of \cite[Proof of Lemma 3.8]{Brz+Gat_99} or \cite[Page 134]{Brz+Masl+Seidler_2005} and using Assumptions \ref{assum-01},  \ref{assum-F} and \ref{assum-G} we infer that for all $t\ge 0$
\begin{align}
\mathbb{E} \lvert w \rvert^2_{X_{t\wedge \varrho_{n,k}} }\le & c \mathbb{E} \int_0^{t\wedge \varrho_{n,k}} \Big[ \lvert F(u(r)) -F(v(r))\rvert^2_H +  \lVert G(u(r)) -G(v(r))\rVert^2\Big]dr\nonumber\\
	\le & c \mathbb{E} \int_0^{t\wedge \varrho_{n,k}} \left(\Vert w(s) \Vert^2 \Vert u(s) \Vert^{2(p-\alpha)} \lvert u(s) \rvert^{2\alpha}_E  \right) ds+
	c \mathbb{E}\int_0^{t\wedge \varrho_{n,k}} \left(\vert w(s) \vert^{2\alpha}_E \Vert w(s) \Vert^{2(1-\alpha)} \lVert v(s) \rVert^{2p} \right)ds\nonumber \\
	&=:c\mathbb{E}I_1+c\mathbb{E} I_2.\label{Eq:Uniq-Step-0}
\end{align}
The H\"older inequality and the definition of the stopping time $\varrho_{n,k}$ imply that for all $t\ge 0$, $\mathbb{P}$-a.s.
\begin{align}
 I_1\le &  c \left(\int_0^{t\wedge \varrho_{n,k}} \Vert w(s)\Vert^{\frac{2}{1-\alpha}} \Vert u(s)\Vert^{\frac{2(p-\alpha)}{1-\alpha}}  ds \right)^{1-\alpha} \left(\int_0^{t\wedge \varrho_{n,k}}  \vert u(s) \vert^{2}_E  ds \right)^{\alpha} \\
%&\le c \left(\int_0^{t\wedge \varrho_{n,k}} \Vert w(s)\Vert^{\frac{2}{1-\alpha}} \Vert u(s)\Vert^{\frac{2(p-\alpha)}{1-\alpha}}  ds \right)^{1-\alpha} \lvert u \rvert^{2\alpha}_{X_{t\wedge \varrho_{n,k}}}\\
%\le & c R^{2\alpha} \left(\int_0^{t\wedge \varrho_{n,k}} \Vert w(s)\Vert^{\frac{2}{1-\alpha}} \Vert u(s)\Vert^{\frac{2(p-\alpha)}{1-\alpha}}  ds \right)^{1-\alpha}\\
\le & c R^{2\alpha}\sup_{s\in [0,t\wedge \varrho_{n,k}]} \Vert u(s) \Vert^{2(p-\alpha)}\left(\int_0^{t\wedge \varrho_{n,k}} \Vert w(s)\Vert^{\frac{2}{1-\alpha}}  ds \right)^{1-\alpha}\\
%\le & c R^{2p} \sup_{s\in [0,t\wedge \varrho_{n,k}] } \lVert w(s) \rVert^{2\alpha} \left(\int_0^{t\wedge \varrho_{n,k}} \Vert w(s)\Vert^2   ds \right)^{1-\alpha}\\
\le &c R^{2p} \lvert w \rvert_{X_{t\wedge \varrho_{n,k}}}^{2\alpha} \left(\int_0^{t\wedge \varrho_{n,k}} \Vert w(s)\Vert^2   ds \right)^{1-\alpha}
\le \frac14 \lvert w \rvert_{X_{t\wedge \varrho_{n,k}}}^{2} + c R^{\frac{2p}{1-\alpha}} \int_0^{t\wedge \varrho_{n,k}} \Vert w(s)\Vert^2   ds .\label{Eq:Uniq-I_1}
\end{align}
In a similar way one can prove that for all $t\ge 0$, $\mathbb{P}$-a.s.
\begin{align}
 I_2\le &  c \left(\int_0^{t\wedge \varrho_{n,k}} \Vert w(s)\Vert^2 \Vert v(s)\Vert^{\frac{2p}{1-\alpha}}  ds \right)^{1-\alpha} \left(\int_0^{t\wedge \varrho_{n,k}}  \vert w(s) \vert^{2}_E  ds \right)^{\alpha} \\
\le & \frac14 \lvert w \rvert_{X_{t\wedge \varrho_{n,k}}}^{2} + c R^{\frac{2p}{1-\alpha}} \int_0^{t\wedge \varrho_{n,k}} \Vert w(s)\Vert^2   ds.\label{Eq:Uniq-I_2}
%	\le & c R^{2\alpha} \left(\int_0^{t\wedge \varrho_{n,k}} \Vert w(s)\Vert^{\frac{2}{1-\alpha}} \Vert u(s)\Vert^{\frac{2(p-\alpha)}{1-\alpha}}  ds \right)^{1-\alpha}\\
%	\le & c R^{2\alpha}\sup_{s\in [0,t\wedge \varrho_{n,k}]} \Vert u(s) \Vert^{2(p-\alpha)}\left(\int_0^{t\wedge \varrho_{n,k}} \Vert w(s)\Vert^{\frac{2}{1-\alpha}}  ds \right)^{1-\alpha}\\
%	\le & c R^{2p} \sup_{s\in [0,t\wedge \varrho_{n,k}] } \lVert w(s) \rVert^{2\alpha} \left(\int_0^{t\wedge \varrho_{n,k}} \Vert w(s)\Vert^2   ds \right)^{1-\alpha}\\
%	\le & c R^{2p} \lvert w \rvert_{X_T}^{2\alpha} \left(\int_0^{t\wedge \varrho_{n,k}} \Vert w(s)\Vert^2   ds \right)^{1-\alpha}\\
%	\le & \frac18 \lvert w \rvert_{X_T}^{2} + c R^{\frac{2p}{1-\alpha}} \int_0^{t\wedge \varrho_{n,k}} \Vert w(s)\Vert^2   ds .
\end{align}
Hence, plugging \eqref{Eq:Uniq-I_1} and \eqref{Eq:Uniq-I_2} in \eqref{Eq:Uniq-Step-0} and using $ \Vert w(t\wedge \varrho_{n,k} )\Vert^2\le \lvert w\rvert^2_{X_{t\wedge \varrho_{n,k}}},\;$  we infer that
%\begin{equation}
%\frac12 \mathbb{E} \lvert w\rvert^2_{X_{t\wedge \varrho_{n,k}}} \le  c R^{\frac{2p}{1-\alpha}} \mathbb{E} \int_0^{t\wedge \varrho_{n,k}} \Vert w(s)\Vert^2   ds.
%\end{equation}
%Since  for all $t\ge 0$, $\mathbb{P}$-a.s., we infer that
\begin{equation}
\mathbb{E} \lVert w(t\wedge \varrho_{n,k} )\rVert^2 \le 2 c R^{\frac{2p}{1-\alpha}}  \int_0^{t\wedge \varrho_{n,k}} \Vert w(s\wedge \varrho_{n,k} ) \Vert^2   ds, \;\forall t\ge0.
\end{equation}
This along the Gronwall lemma implies that for all $t\ge 0$, $
\mathbb{E} \Vert w(t\wedge \varrho_{n,k}) \Vert^2=0.$
Hence, by letting $k \to \infty$ we infer that for all $t\ge 0$, $
\mathbb{E} \Vert w(t\wedge \tau_n\wedge \sigma) \Vert^2=0,$
 which along with the continuity of $w$ completes the proof of the theorem.
\end{proof}

	\begin{thm}\label{thm_local} Suppose that Assumptions \ref{assum-F}-\ref{assum-01} are satisfied. Then
\begin{trivlist}
\item[(I)]
 for every
		$\mathcal{F}_{0}$-measurable $V$-valued square integrable random
		variable $u_0$    there exits  a  local process $u=\big(u(t),
		t\in[0,T_1) \big) $  which is the
		unique local solution to   problem \eqref{ABS-SPDE-strong},
\begin{comment}\item[(II)]
 for all  $R>0$ and  $\varepsilon >0$ there exists
		$\tau(\varepsilon,R)>0$,  such that for every
		$\mathcal{F}_0$-measurable $V$-valued random variable $u_0$
		satisfying  $\mathbb{E}\Vert u_0 \Vert^{2} \leq R^{2}$, one has
		\begin{equation}\label{ineq-positive probability}
{\mathbb P}\big(T_1\geq \tau(\varepsilon,R)\big) \geq
		1-\varepsilon.\end{equation}
\end{comment}
\item[(II)]
 if  $R>0$ and  $\varepsilon >0$ then there exists a number 
		$T^\ast(\varepsilon,R)>0$,  such that for every set $\Omega_1 \in \mathcal{F}_0$ and every
		$\mathcal{F}_0$-measurable $V$-valued random variable $u_0$ such that
		  \[\Vert u_0 \Vert \leq R \mbox{ $\mathbb{P}$-a.s. on } \Omega_1,\]  one has
		\begin{equation}\label{ineq-positive probability-2}
{\mathbb P}\big(\{T_1\geq T^\ast(\varepsilon,R)\} \cap \Omega_1 \big) \geq
		(1-\varepsilon)\mathbb{P}(\Omega_1).\end{equation}

\end{trivlist}
	\end{thm}
	\begin{proof}[Proof of Theorem \ref{thm_local}]

		Wlog we can assume that $N=1$ and we will use notation $\alpha=\alpha_1$ and $p=p_1$.
Let $u_0\in
		L^2(\Omega, \mathbb{P}; V)$. We also fix a natural number   $n\in \mathbb{N}$ in Steps \textbf{1}-\textbf{5}.
In the first part consisting of  Steps \textbf{1}-\textbf{7} we will prove the part (I) of the Theorem, i.e. the existence and uniqueness of a local solution to  problem \eqref{ABS-SPDE-strong}.
Part (II) of the Theorem will be proven in Steps \textbf{8}-\textbf{9}.

\noindent \textit{Proof of part (I)}
\begin{trivlist}
\item[\textbf{Step 1.}] Let us fix $n\in \mathbb{N}$ and $T>0$.
Let  $\Psi^n_{0,T,u_0}: \mathscr{M}^2(\hat{X}_{0,T,u_0}) \to \mathscr{M}^2(\hat{X}_{0,T,u_0}) $.
By Proposition \ref{prop-Psi_T-Lipschitz} the map  $\Psi^n_{0, T,u_0}$ is well defined
 and for sufficiently small $T=\delta_n$, and all $\ba_0$,
 it is an $\frac12$-contraction. 
			Thus, by the Banach Fixed Point Theorem,   there exists a unique $u^{[n,1]}\in \mathscr{M}^2(\hat{X}_{0,\delta_n,u_0})$ such that
			\[u^{[n,1]}=\Psi^n_{0,\delta_n,u_0}(u^{[n,1]}).\] We fix $u^{[n,1]}$ for the rest of the proof. We also put $M:=\frac{T}{\delta_n}\in \mathbb{N}$.		

\item[\textbf{Step 2.}]
By Lemma \ref{lem-fixed points} $\Psi^n_{\delta_n, 2 \delta_n, u^{[n,1]} }$ maps $\mathscr{M}^2(\hat{X}_{\delta_n, 2 \delta_n, u^{[n,1]}})$ into itself
and by Proposition \ref{prop-Psi_T-Lipschitz} and inequality \eqref{eqn-global Lipschitz}  it is an  $\frac 12$-contraction.  Therefore, we can find  a unique
 $u^{[n,2]} \in \mathscr{M}^2(\hat{X}_{\delta_n, 2\delta_n, u^{[n,1]}})$, which we fix for the rest of the proof,  such that \[u^{[n,2]}=\Psi^n_{\delta_n, 2\delta_n, u^{[n,1]}}(u^{[n,2]})\in \mathscr{M}^2(X_{\delta, 2\delta_n, u^{[n,1]}}).\]

\item[\textbf{Step 3.}]
			 By induction we can construct a sequence $(u^{[n,k]})_{k=1}^\infty$ such that
\[u^{[n,k]}=\Psi^n_{(k-1)\delta_n,k\delta_n, u^{[n,k-1]}}(u^{[n,k]})\in \mathscr{M}^2(X_{(k-1)\delta_n,k\delta_n, u^{[n,k-1]}}),\;\; k=2,\ldots .\]
			
Note that by construction, the restriction of $ u^{[n,k]} $ to interval $[0,(k-1)\delta_n]$ is equal to  $u^{[n,k-1]}$.

\item[\textbf{Step 4.}]
By \textbf{Step 3} we can define a process  $u^n\in \mathscr{M}^2(X_T)$  by $u^n(t)=u^{[n,k]}(t)$, if $t \in [0,k\delta_n]$. Moreover,  for every $t\in [0,T]$, $\mathbb{P}$-a.s.,
			\begin{equation}\label{Eq:truncated-equation}
			u^n(t)=S_t u_0+\int_0^t S_{t-r}[\theta_n (|u^n|_{X_r})F(u^n(r))]dr+\int_0^t
			S_{t-r}[\theta_n (|u^n|_{X_r})G(u^n(r))]dW(r).
			\end{equation}

\item[\textbf{Step 5.}]
Let $(\tau_n)_{n\in \mathbb{N}}$ be a  sequence of stopping
			times defined by
			\begin{equation}\label{eqn-stopping time tau_n}
			\tau_n=\inf\{t\in [0,\infty): |u^n|_{X_t}\ge n\}.
			\end{equation}

Let us fix $n\in \mathbb{N}$.
 By \cite[Lemma A.1]{Brz+Masl+Seidler_2005}, we infer from \eqref{Eq:truncated-equation} that for every $t\in [0,\infty)$, $\mathbb{P}$-a.s.
			\begin{equation}\label{Eq:truncated-equation-2}
			u^n(t \wedge \tau_n)=S_{t\wedge \tau_n} u_0+\int_0^{t\wedge \tau_n} S_{t\wedge \tau_n -r}[\theta_n (|u^n|_{X_r})F(u^n(r))]\,dr
+\tilde{I}^n_{\tau_n}(t \wedge \tau_n),
			\end{equation}
where $\tilde{I}^n_{\tau_n}$ is a continuous $V$-valued %\hdoubtz{ZB2101: Please check!}
\[
\tilde{I}^n_{\tau_n}(t):= \int_0^t \mathds{1}_{[0,\tau_n)}(s)
			S_{t-r}[\theta_n (|u^n|_{X_r})G(u^n(r))]dW(r),\;\; t\in [0,\infty).
\]
	By the definition of the function $\theta_n$ we infer that  $\theta_n(|u^n|_{X_{r}})=1$ for $r\in [0,t\wedge
			\tau_n)$. Hence
			$$\theta_n (|u^n|_{X_r})F(u^n(r)) = F(u^n(r)), r\in [0,t\wedge
			\tau_n), \; t\in [0,\infty).$$
Therefore, we deduce  that for every $t\in [0,\infty)$, $\mathbb{P}$-a.s.
\[
\tilde{I}^n_{\tau_n}(t)= \int_0^t \mathds{1}_{[0,\tau_n)}(s)
			S_{t-r}[G(u^n(r))]dW(r)=:I^n_{\tau_n}(t).
\]
Thus, we infer that
			$u^n$ satisfies, for every  $t\in [0,\infty)$, $\mathbb{P}$-a.s.
			\begin{equation}\label{Eq:truncated-equation-2-A}
			u^n(t \wedge \tau_n)=S_{t\wedge \tau_n} u_0+\int_0^{t\wedge \tau_n} S_{t\wedge \tau_n -r}[F(u^n(r))]\,dr
+I^n_{\tau_n}(t \wedge \tau_n).
			\end{equation}

\item[\textbf{Step 6.}]\label{lem-increasing stopping times}
 Arguing as in the proof of proof of \cite[Lemma 5.1]{Brz+Millet_2012} we  can  show that for every  $n \in \mathbb{N}$,
\begin{equation}
\label{eqn-increasing stopping times}
\tau_n < \tau_{n+1} \;\;\;\; \mathbb{P}\text{-a.s.}
\end{equation}
and
\begin{equation}
\label{eqn-u_n=u_n+1}
u^n(t)=u^{n+1}(t) \mbox{ if } t\in [0,\tau_n) \mbox{ and  } n \in \mathbb{N},\;\; \mathbb{P}\text{-a.s.}
\end{equation}
By taking appropriate modifications we can assume that \eqref{eqn-increasing stopping times}  is satisfied on the whole space $\Omega$.
Hence, the following limit exists

\begin{equation}\label{eqn-tau_infty=lim}
\tau_\infty(\omega)=\lim_{n\rightarrow \infty} \tau_n(\omega), \,\, \omega \in \Omega.
\end{equation}
Since our probability basis  satisfies the usual hypothesis, $\tau_\infty$ is an {accessible} stopping time, see \cite[Proposition 2.3 and Lemma 2.11]{Kar-Shr-96} with $(\tau_n)$ being the announcing sequence for $\tau_\infty$. The two claims made at the beginning of \textbf{Step 6} enable us  to define a   local process  $(u,\tau_\infty)$ in  the following way
\begin{equation}
\label{eqn-local process u}
 u(t,\omega)=u^n(t,\omega) \text{ if } t<\tau_n(\omega), \omega \in \Omega.
\end{equation}
\item[\textbf{Step 7.}]\label{Step7}  We claim that  $(u,\tau_\infty)$ is a local solution to problem \eqref{ABS-SPDE-strong}. \\
Indeed, arguing as in \textbf{Step 5}, in particular using \cite[Lemma A.1]{Brz+Masl+Seidler_2005}, we can show that
			$u^n$ satisfies, for every  $t\in [0,\infty)$, $\mathbb{P}$-a.s.
			\begin{equation}\label{eqn-local-3}
			u(t \wedge \tau_n)=S_{t\wedge \tau_n} u_0+\int_0^{t\wedge \tau_n} S_{t\wedge \tau_n -r}[F(u(r))]\,dr
+I_{\tau_n}(t \wedge \tau_n).
			\end{equation}	
where $I_{\tau_n}$ is an $V$-valued continuous process defined by
\[
I_{\tau_n}(t)= \int_0^t \mathds{1}_{[0,\tau_n)}(s)
			S_{t-r}[G(u(r))]\,dW(r), \;\; t\in [0,\infty).
\]
Since, as observed above, $\tau_\infty$ is an accessible stopping time with the announcing sequence $(\tau_n)$, by definition \ref{def-local solution} we infer that
the local			process $(u, \tau_\infty)$ is a local solution to problem \eqref{ABS-SPDE-strong}.

This also ends the proof of the first part of Theorem \ref{thm_local}.
			\end{trivlist}

%\begin{proof}[Proof of part (III)].

\noindent 	\textit{Proof of part (III)}
	
	Let us recall that $\Omega_1 \in \mathcal{F}_0$.  Let $i:\Omega_1 \embed \Omega$ be the natural embedding and $W^1$  a process given by 
	\[
	W^1(t):=W(t)\circ i,\;\; t\geq 0.
	\]
	We define $\mathcal{F}_t^1$ by
	\[
	\mathcal{F}_t^1:=\bigl\{ A\cap \Omega_1=i^{-1}(A): A \in \mathcal{F}_t \bigr\}, \;\;\; t\geq 0.
	\]
	Similarly we define $\mathcal{F}^1$.  We put $\mathbb{F}^1=\bigl(\mathcal{F}_t^1\bigr)_{t \geq 0}$.
	
	We also define a measure
	\[
	\mathbb{P}^1: \mathcal{F}^1 \ni A\cap \Omega_1 \mapsto \frac{\mathbb{P}(A\cap \Omega_1)}{\mathbb{P}(\Omega_1)} \in [0,1].
	\]
	It is easy to check that
	$
	\bigl(\Omega_1,\mathcal{F}^1,\mathbb{P}^1, \mathbb{F}^1\bigr)
	$
	is a filtered probability space satisfying the usual conditions. Moreover, since the Wiener process $W$ is independent of $\mathcal{F}_0$,
	$W^1$ is a Wiener process on $\bigl(\Omega_1,\mathcal{F}^1,\mathbb{P}^1, \mathbb{F}^1\bigr)$.
	
\noindent	Hence, it is sufficient to prove our result in the case when $\Omega_1=\Omega$. Indeed, the following holds true.
	\begin{lem}\label{lem-localization of the probability space}
		If
		a  local process $u=\big(u(t),
		t\in[0, \tau) \big) $  is a
		local solution to   problem \eqref{ABS-SPDE-strong} on  the original probability basis
		$\bigl(\Omega,\mathcal{F},\mathbb{P}, \mathbb{F}\bigr)$, then the
		process $u^1=\big(u(t),		t\in[0, \tau^1) \big) $ defined by 
		\[
		u^1(t,\omega_1):=u(t,i(\omega_1)), \;\; t \in  [0,\tau(i(\omega_1)),\;\; \omega_1 \in \Omega_1,
		\]
		and $\tau^1:=\tau\circ i$, 
		is a local solution to   problem \eqref{ABS-SPDE-strong} on  the new  probability basis
		$\bigl(\Omega_1,\mathcal{F}^1,\mathbb{P}^1, \mathbb{F}^1\bigr)$.
	\end{lem}
%	\coma{Maybe we need an argument similar to Applebaum book, p. 372-373.}
	
	Thus we assume that  $u_0$ is a $\mathcal{F}_0$-measurable $V$-valued random variable such that
	\[\Vert u_0 \Vert \leq R \mbox{ $\mathbb{P}$-a.s. on } \Omega.\]

	Our  proof  follows the lines of \cite[Theorem 5.3]{Brz+Millet_2012}. In order to simplify the notation we will write $\Psi^n_{T}$ instead of $\Psi^n_{0,T,u_0}$, We will also write $\Phi^n_F$ (resp. $\hat{\Phi}^n_G$) instead of $\Phi^n_{\delta,T,\ba}$ (resp. $\hat{\Phi}^n_{\delta,T,\ba}$).
	
	%\begin{comment}
	We modify the initial data $u_0$ by replacing it by $\tilde{u}_0=u_01_{\Omega_1}$. Then $ \Vert \tilde{u}_0 \Vert\leq R$ on $\Omega$.
	
	\begin{trivlist}
		\item[\textbf{Step 8.}] Let us  fix $\varepsilon
		>0$ and choose $M$
		such that $M\ge (C_0+1) \eps^{-\frac 12}$, where  $C_0$ is the constant appearing in inequality \eqref{ineq-C_0} below.
		Thanks to  Propositions \ref{prop-global Lipschitz-F} and \ref{prop-Psi_T-Lipschitz},  Corollary \ref{cor-global-lip-G} and Assumption
		\ref{assum-01} we can find $\tilde{D}_i(n)>0, i=1,2$, $n\in\mathbb{N}$  such that for all $u\in \mathscr{M}^2(\hat{X}_{0,T,u_0})$ we have
		\begin{align*}
		\lvert S\ast \Phi_F^n(u)\rvert_{\mathscr{M}^2(X_T)}
		\le \tilde{D}_1(n) \tilde{C}_1 \tilde{C}_2(2n \tilde{C}_2+1)(2n)^{p+2}T^{(1-\alpha)/2},\\
		\lvert S\diamond \hat{\Phi}^n_G(u)\rvert_{\mathscr{M}^2(X_T)}\le  \tilde{D}_2(n) \tilde{C}_3
		\tilde{C}_4 (2n)^{k+2} [2n \tilde{C}_4+1] T^{(1-\beta)/2}.
		\end{align*}
		Hence,		since $1-\alpha, 1-\beta >0$ we infer
		that there exists a sequence $(K_n(T))_n$ of numerical functions
		such that for all $n$,  $\lim_{T\rightarrow 0}  K_n(T)=0$ and
		\begin{equation*}
		\lvert S\ast \Phi^n_T(u)+ S\diamond
		\Phi^n_G(u)\rvert_{\mathscr{M}^2(X_T)}\le K_n(T),\;\; u\in \mathscr{M}^2(X_T).
		\end{equation*}
		Let us put $n=M R^2$ 
		and
		choose $\delta_{1}(\varepsilon,R) >0$ such that
		$K_{n}(\delta_{1}(\varepsilon,R)) \leq  R   $. Let $\Psi^n_T$ be the
		mapping defined by \eqref{eqn-Psi_T}. Since $\me \lve \tilde{u}_0\rve^2\leq
		R^2 $, we infer by the Assumption \ref{assum-01}  that
		\begin{align} \label{ineq-C_0}
		\lvert \Psi^n_T(u)\rvert_{\mathscr{M}^2(X_T)}\le & C_0 R + K_n(T)
		\le  (C_0+ 1)R ,
		\end{align}
		for all $T\le \delta_1(\eps,R)$. 
%		That is, for $T\le
%		\delta_1(\eps,R)$
%		the range of the map $\Psi_T^{n}$ is included in the ball (in the space ${M}^2(X_T)$) centered at $0$ and of radius  $(C_0+1)R $. 
		Furthermore, Propositions \ref{prop-Psi_T-Lipschitz}
		implies that there exists $C>0$ such that 
		\begin{equation*}
		\lvert \Psi^n_T(u_1)-\Psi^n_T(u_2)\rvert_{\mathscr{M}^2(X_T)}\le C_3
		M^{D}R^{D} \Big[T^{(1-\alpha)/2} \vee T^{(1-\beta)/2}\Big] \lvert
		u_1-u_2\rvert_{\mathscr{M}^2(X_T)}, \text{ for all $u_1,u_2\in
			\mathscr{M}^2(X_T)$}.
		\end{equation*}
		Since  $M\ge (C_0+1) \eps^\frac 12$  we infer that
		\begin{equation*}
		\lvert \Psi^n_T(u_1)-\Psi^n_T(u_2)\rvert_{\mathscr{M}^2(X_T)}\le C_3
		M^{D}R^{D} \Big[T^{(1-\alpha)/2} \vee T^{(1-\beta)/2}\Big] \lvert
		u_1-u_2\rvert_{\mathscr{M}^2(X_T)}.
		\end{equation*}

		Hence we can find $\delta_2(\eps,R)>0$ such that $\Psi^n_T$ is a
		strict contraction for all $T\le \delta_2(\eps,R)$. Thus if one
		puts $T^\ast(\varepsilon,R)=\delta_1(\eps,R)\wedge \delta_2(\eps,R)$, the
		mapping $\Psi^n_T$ has a unique fixed point ${\hat{u}^n}$ which satisfies
		\begin{equation}\label{ineq-moment-1}
		\me \lvert {\hat{u}^n}\rvert^2_{X_{T^\ast(\varepsilon,R)}}\le (C_0+1)^2 R^2.
		\end{equation}
		Similarly to \eqref{eqn-stopping time tau_n} we can  define a new  stopping time ${\hat{\tau}_n}$ by
		\[
		{\hat{\tau}_n}:= \inf\{t\in [0,\infty): |{\hat{u}^n}|_{X_t}\ge n\}.
		\]
		Arguing as in \textbf{Step 5}  we can show that 			${\hat{u}^n}$ satisfies, for every  $t\in [0,T]$, $\mathbb{P}$-a.s.
		\begin{equation}\label{Eq:truncated-equation-9-1}
		{\hat{u}^n}(t \wedge \tau_n)=S_{t\wedge \tau_n} u_0+\int_0^{t\wedge {\hat{\tau}_n}} S_{t\wedge {\hat{\tau}_n} -r}[F({\hat{u}^n}(r))]\,dr
		+{\hat{I}^n_{\hat{\tau}_n}}(t \wedge {\hat{\tau}_n}).
		\end{equation}
		where ${\hat{I}^n_{\hat{\tau}_n}}$ is a continuous $V$-valued process defined by
		\[
		{\hat{I}^n_{\hat{\tau}_n}}(t):=\int_0^t \mathds{1}_{[0,\tau_n)}(s)
		S_{t-r}[G({\hat{u}^n}(r))]dW(r), \;\; t\in [0,T].
		\]
		
		%\hdoubtz{ZB2102: I this that this claim will not be used. }		As in \cite[Proposition 5.1]{Brz+Millet_2012} we can show that { the local process }
		%		$(u^n(t), t\le \tau_n)$ is a local solution to problem
		%		\eqref{ABS-SPDE}.

		By the definition   of the stopping time ${\hat{\tau}_n}$,
		$\{{\hat{\tau}_n} \le T^\ast(\varepsilon,R)\} \subset \{
		\lvert u^n\rvert_{X_{T^\ast(\varepsilon,R)}}\ge n \}$. Therefore, by the Chebyshev
		inequality and inequality \eqref{ineq-moment-1} we infer that 		
		\begin{align*}
		\mathbb{P}({\hat{\tau}_n} \le T^\ast(\varepsilon,R))\le &\mathbb{P}(\lvert
		u^n\rvert^2_{X_{T^\ast(\varepsilon,R)}}\ge n )  \leq \frac{1}{n}\mathbb{E}\lvert
		u^n\rvert^2_{X_{T^\ast(\varepsilon,R)}}  \leq \frac{1}{n} (C_0+1)^2R^2.
		\end{align*}
		Since $n=NR^2$ and $N\ge (C_0+1)^2 \eps^{-\frac 12}$ we get
		\begin{align*}
		\mathbb{P}({\hat{\tau}_n} \le T^\ast(\varepsilon,R))\leq & (C_0+1)^2 N^{-2} \leq \eps.
		\end{align*}
		
		Hence, we have prove \eqref{ineq-positive probability-2}.
		
		\item[\textbf{Step 9.}]\label{Step9}
		
		To conclude the proof, we observe that in view of Remark \ref{rem-local solution}, it follows from equality \eqref{Eq:truncated-equation-9-1} that the process
		${\hat{u}^n}$ restricted to the open random interval $[0, \hat{\tau}^n)\times  \Omega$ is  a  local
		solution to problem \eqref{ABS-SPDE-strong}. On the other hand, in \textbf{Step 7}  we  proved that also   $(u,\tau_\infty)$ is a local solution to problem \eqref{ABS-SPDE-strong}.
		{Because local uniqueness holds for problem \eqref{ABS-SPDE-strong}}, see Theorem \ref{Thm:LocalUniqueness}, we infer by applying Corollary \ref{cor-max of two local solutions} that the supremum
		of $(u,\tau_\infty)$ and $({\hat{u}^n},\hat{\tau}^n)$  is another local
		solution to problem \eqref{ABS-SPDE-strong}.
		Hence,   the stopping time $T_1= \tau_\infty \vee \hat{\tau}^n$ satisfies the requirements of the theorem.
	\end{trivlist}
	This  concludes the proof of part (III) of Theorem \ref{thm_local} and thus of the whole theorem.
		
\end{proof}

 The next result is about the existence and uniqueness of a maximal solution and the characterization of its lifespan.
	\begin{thm}\label{thm_maximal-abstract}
	Let  $u_0\in L^2(\Omega,\mathcal{F}_0,V)$. Then, problem \eqref{ABS-SPDE-strong} has a unique maximal local solution $(\hat{u},\hat{\tau})$.
		Moreover,  $(u,\tau_\infty) \sim (\hat{u},\hat{\tau})$ and
	\begin{align}
	&	\lim_{t\toup \hat{\tau}}|\hat{u}|_{X_t}=\infty \;\;\; \mathbb{P}\mbox{-a.s. on
		} \{\hat{\tau}<\infty\},\label{eqn-t_infty}\\
&		\mathbb{ P}\Big(\{\omega \in\Omega :\hat{\tau} (\omega)< \infty \text{ and }
		\sup_{t\in [0, \hat{\tau} (\omega))} \lvert \hat{u}(t)(\omega)\rvert_V<\infty  \}\Big)=0. \label{eqn-t_infty-limit}
		\end{align}
		%%%%%%%%%%%%%%%%%%		
	\end{thm}

%{
	\begin{proof}
		Let us choose and fix $u_0\in L^2(\Omega,\mathcal{F}_0,V)$. Wlog we can assume that $\mathbb{P}(\{\hat{\tau}<\infty\})>0$.
		
		Firstly, we observe that it follows from the proof of  Theorem \ref{thm_local}  that the local process $(u,\tau_\infty)$ defined in  \eqref{eqn-local process u} is a local solution to \eqref{ABS-SPDE-strong}.
 In particular,  the set of local solutions to problem \eqref{ABS-SPDE-strong} is non-empty.  Since by Theorem \ref{Thm:LocalUniqueness} the local uniqueness holds for problem \eqref{ABS-SPDE-strong}, we infer by applying Proposition \ref{prop-loc-implies-max} that  there exists a unique maximal local solution $(\hat{u},\hat{\tau})$ to \eqref{ABS-SPDE-strong}. Moreover,
		$(\hat{u},\hat{\tau})$ satisfies the following
		\begin{equation}\label{eqn-similarity}
		\hat{\tau} \ge \tau_\infty \text{ $\mathbb{P}$-a.s. and } \hat{u}_{\lvert_{[0,\tau_\infty)\times \Omega} }= u.
		\end{equation}		
Secondly, suppose that $\mathbb{P} \bigl( \hat{\tau} > \tau_\infty \bigr)>0$. Let $\bigl(\hat{\tau}_n\bigr)$ be the announcing sequence of $\hat{\tau}$. Since $\hat{\tau}_n \toup \hat{\tau}$ $\mathbb{P}$-a.s., we infer that
there exists  $n \in \mathbb{N}$ such that $\mathbb{P} \bigl( \hat{\tau} > \hat{\tau}_n> \tau_\infty \bigr)>0$.  Thus we infer that there exists $t>0$ such that
$\mathbb{P} \bigl( t> \hat{\tau} > \hat{\tau}_n> \tau_\infty \bigr)>0$. \\
On the other hand,  by the definition  \eqref{eqn-stopping time tau_n} of the announcing sequence $(\tau_n)_{n \in \mathbb{N}}$ for $\tau_\infty$ and the definition, see  \eqref{eqn-local process u}, of the local process $(u,\tau_\infty)$ that the sequence $\vert u \vert_{X_{\tau_n}}$  converges to $\infty$ on $\{ \tau_\infty< \infty\}$. Hence we infer that $\lvert u\rvert_{X_{\tau_\infty}} = \infty $ $\mathbb{P}$-a.s. on
 $\bigl\{ t> \hat{\tau} > \hat{\tau}_n> \tau_\infty \bigr\}$. Since the probability of the last set is $>0$, this contradicts condition \eqref{eq-locsol_00} of Definition
  \ref{def-local solution-2} of a local solution.  This contradiction implies that $\mathbb{P} \bigl( \hat{\tau} > \tau_\infty \bigr)=0$, and  in view of \eqref{eqn-similarity}, we also have $(u,\tau_\infty) \sim (\hat{u},\hat{\tau})$.

		Thirdly,  we infer from the definition  \eqref{eqn-stopping time tau_n} of the announcing sequence $(\tau_n)_{n \in \mathbb{N}}$ for $\tau_\infty$ and the definition, see  \eqref{eqn-local process u}, of the local process $(u,\tau_\infty)$ that the sequence $\vert u \vert_{X_{\tau_n}}$  converges to $\infty$ on $\{\tau_\infty< \infty\}$.  Since
the function   $t \mapsto \vert u \vert_{X_{t}}$ is increasing, by  \eqref{eqn-similarity},  we infer that
\begin{equation}\label{eqn-blow-up similarity}
\dela{\lim_{t\toup \hat{\tau}} \lvert \hat{u} \rvert_{X_t} \ge
\lim_{t\toup \tau_\infty} \lvert \hat{u} \rvert_{X_t} =} \lim_{t\toup \tau_\infty}\lvert u \rvert_{X_t} = \lim_{n \toup \infty} \lvert u\rvert_{X_{\tau_n}} = \infty \text{ $\mathbb{P}$-a.s. on }  \{\hat{\tau}< \infty \}.
\end{equation}
This proves	\eqref{eqn-t_infty}.

Fourthly,  we shall prove  \eqref{eqn-t_infty-limit}. By contradiction, we assume that
there exists  $\eps>0$ such that
\[ {\mathbb P }\big(\{\tau_\infty <\infty\} \cap \{ \sup_{t\in [0,\tau_\infty)} |\hat{u}(t)|_{V}<\infty
\}\big)=4\eps>0.\]
Hence we can easily deduce that  there exists   $R>0$ such  that
\[ {\mathbb P }\big(  \{ |u(t)|_{V}<R \mbox{ for all } t\in [0,\tau_\infty )
\}\big) \geq 3\eps.\]
 Let us now choose
$\alpha$ such that  $\alpha= \frac 12 T^\ast(\varepsilon,R)$, where the number $T^\ast(\varepsilon,R)>0$ depending only on $\eps$ and $R$ comes from 
part (II) of  Theorem \ref{thm_local}.  By the definition of an announcing sequence 
 $\bigl(\tau_n\bigr)$  of $\tau_\infty$   we infer that
for arbitrary  $\delta>0$   there exists $n_0>0$ such that
$\mathbb{P}(\Omega_0)\geq (1-\delta)\mathbb{P}(\tilde{\Omega})$,
where $\Omega_0:=\{\omega \in \tilde{\Omega}: |u(t)|_{V}<R \mbox{ for all } t\in [0,\hat{\tau}) \mbox{ and }
\hat{\tau}-\hat{\tau}_{n_0}<\alpha\}.$ Choosing $\delta=\frac13$ we get
$\mathbb{P}(\Omega_0)\geq 2\eps$. Let
$y_0= u(\tau_{n_0})$. 
Note that $y_0$ is $\mathcal{F}_{\tau_{n_0}}$-measurable and  $\vert y_0\vert_{V}\leq R$ on $\Omega_0$.
With the previously chosen $R$,  $\eps$ and $T^\ast(\varepsilon,R)>0$, by applying part (II) of  Theorem \ref{thm_local} we find  a local solution $y(t)$, $t\in
[\tau_{n_0},\tau_{n_0} +T_1)$ to problem \eqref{ABS-SPDE-1} with the initial condition (starting at $\tau_{n_0}$) $y(\tau_{n_0})=y_0$ such that
$\mathbb{P}(T_1\geq T^\ast(\varepsilon,R))>1-\eps$.  
Also, let $\Omega_1:= \Omega_0 \cap \{ T_1\geq T^\ast(\varepsilon,R)\}.$ Since
$\mathbb{P}(\hat{\tau}-T_0< \frac 12 T^\ast(\varepsilon,R))\geq 2\eps$, we infer that
\[\mathbb{P}\big(\Omega_1  \big) \geq \eps>0.\]

By a generalization of   \cite[Corollary 2.28]{Brz+Elw_2000} to the case of SPDEs we infer that  a local stochastic process $v(t)$, $t\in [0,\tau_{n_0}(\omega) +T_1)$ defined by
\begin{equation*}
v(t,\omega)=\begin{cases}
u(t,\omega) \text{ if }  t \in [0,\tau_{n_0}(\omega)] ,\\
y(t,\omega) \text{ if } t\in [ \tau_{n_0}(\omega), \tau_{n_0}(\omega) +T_1 ).\\
\end{cases}
\end{equation*}
is a local solution to problem \eqref{ABS-SPDE-1} with the initial data $v(0)=u_0$. However, on the $\Omega_1$, we have 
\[
\tau_{n_0} +T_1-\tau_\infty =   T_1 -( \tau_\infty -\tau_{n_0} )\geq 2\alpha -\alpha=\alpha>0,
\]
what contradicts the maximality of the solution $(u,\tau_\infty)$ proved in the earlier part of the theorem.
This completes the proof of \eqref{eqn-t_infty-limit} as well as the theorem.

	\end{proof}
	\subsection*{Acknowledgments}
	This article is part of a project that is currently funded by the  European Union's Horizon 2020 research and innovation programme under the Marie Sk\l{}odowska-Curie grant agreement No. 791735 ``SELEs".
	The authors are also grateful for the support they received from the FWF-Austrian Science through the Stand-Alone project P28010.
	Z. Brze{\' z}niak presented a
	lecture based on a preliminary version of  this paper at the RIMS Symposium on Mathematical Analysis of Incompressible Flow held at Kyoto in  February 2013.
	He would like to thank Professor Toshiaki Hishida for the kind invitation. Razafimandimby is also very grateful to the organizers of the conference `` Nonlinear PDEs in Micromagnetism: Analysis, Numerics and Applications'', which was held in ICMS Edinburgh, UK, for their invitation to present a talk based on preliminary result of this paper at this meeting. He is very grateful for the financial support he received from the International Centre for Mathematical Sciences (ICMS) Edinburgh.
	Last, but not the least, the authors wish to thank Professor Guoli Zhou for pointing out some gaps in the previous version of this paper \cite{BHP-arxiv}.

\appendix
\section{On stopped stochastic convolution processes}
\label{sec-stopped sc}
Let $K$ and $E$ be two separable Hilbert spaces.  Suppose that $W$  is a  canonical cylindrical Wiener process on $\rK$ and that $\xi:[0,\infty) \to \gamma(\rK,E)$ is a progressively measurable process such that 
\begin{equation}
\label{eqn-A.1} \int_0^t \Vert \xi(s)\Vert^2_{\gamma(\rK,E)}\, ds<\infty, \mbox{for all $t\geq 0$, $\mathbb{P}$-almost surely.}
\end{equation}
Assume that $S=\bigl(S_t\bigr)_{t\geq 0}$ is a $C_0$ semigroup on $E$.
Let us define  a  process $I$ by
\begin{equation}
\label{eqn-A.2} I(t):= \int_0^t  S_{t-s} \xi(s)\,  dW(s), \;\; t\geq 0.
\end{equation}
Assume that $\tau$ is a finite stopping time. Let us define
a  process $I_\tau$ by
\begin{equation}
\label{eqn-A.3} I_\tau(t):= \int_0^t  \mathds{1}_{[0,\tau)}(s) S_{t-s} \xi(s)\,  dW(s), \;\; t\geq 0 .
\end{equation}

Let us observe that since $\tau$ is a stopping, the stochastic process $\mathds{1}_{[0,\tau)}(s)$, $s\in [0,\infty)$ is well-measurable, see \cite[Proposition 4.2]{Metivier_1982}. Therefore, since by \cite[Theorem 1.6]{Metivier_1982}, the $\sigma$-field of well measurable sets is smaller than the $\sigma$-field of progressively measurable sets, it follows
that the stochastic process $\mathds{1}_{[0,\tau)}$,  is progressively measurable. In particular, the integrand in \eqref{eq-locsol_01-c} is  progressively measurable.

If both processes $I$ and $I_\tau$ have continuous paths $\mathbb{P}$-a.s, then the next lemma  was proved in \cite{Brz+Masl+Seidler_2005}.
\begin{lem}\label{lem-A.1}
\begin{equation}
\label{eqn-A.4} S_{t-t \wedge \tau} I(t\wedge \tau) =I_\tau(t) \mbox{ for all $t\geq 0$, $\mathbb{P}$-a.s.,}
\end{equation}
and
\begin{equation}
\label{eqn-A.5}  I(t\wedge \tau) =I_\tau(t \wedge \tau) \mbox{ for all $t\geq 0$, $\mathbb{P}$-a.s.}
\end{equation}
\end{lem}

Let us observe that the process $I_\tau$ is well defined even if the integrand $\xi$ is only defined on the random interval $[0,\tau)\times \Omega$ and that it satisfies the following mofification of the condition
\eqref{eqn-A.1}, i.e.
\begin{equation}
\label{eqn-A.1'} \int_0^{t \wedge \tau} \Vert \xi(s)\Vert^2_{\gamma(\rK,E)}\, ds<\infty, \mbox{ for all $t\geq 0$, $\mathbb{P}$-a.s.}
\end{equation}
In particular, if
$\xi$ is defined on the random closed interval $[0,\tau]\times \Omega$ and that it satisfies
\eqref{eqn-A.1}, i.e.
\begin{equation}
\label{eqn-A.1''} \int_0^{\tau} \Vert \xi(s)\Vert^2_{\gamma(\rK,E)}\, ds<\infty, \mbox{ $\mathbb{P}$-a.s.}
\end{equation}

Let us now formulate a useful corollary of the above result.
\begin{cor}\label{cor-A.2}
 Under the above assumptions, if $\xi$ is a progressively measurable process  defined on  $[0,\tau)\times \Omega$ and $\xi$ satisfies condition  \eqref{eqn-A.1'} and
 $\sigma$ is another stopping time such that $\sigma \leq \tau$, then
\begin{equation}
\label{eqn-A.6}  I_\tau(t\wedge \sigma) =I_\sigma(t\wedge \sigma) \mbox{ for all $t\geq 0$, $\mathbb{P}$-a.s.}
\end{equation}
\end{cor}
\begin{proof}[Proof of Corollary \ref{cor-A.2}] Let us define a new process $\eta=\mathds{1}_{[0,\tau)}\xi$. Obviously, the process $\eta$  satisfies the assumptions of Lemma \ref{lem-A.1}. In particular, if we define
continuous processes $J$ and $J_\sigma$ by formulae \eqref{eqn-A.2}-\eqref{eqn-A.3} with $\tau$ replaced by $\sigma$ and $\xi$ replaced by $\eta$, i.e.
\begin{equation}
 J(t):= \int_0^t  S_{t-s} \eta(s)\,  dW(s) \text{ and } J_\sigma(t):= \int_0^t  \mathds{1}_{[0,\sigma)}(s) S_{t-s} \eta(s)\,  dW(s) \;\; t\geq 0, \label{eqn-A.8}
\end{equation}
%\begin{equation}
% J_\sigma(t):= \int_0^t  \mathds{1}_{[0,\sigma)}(s) S_{t-s} \eta(s)\,  dW(s), \;\; t\geq 0,
%\end{equation}
then by Lemma \ref{lem-A.1} we have
\begin{equation}
\label{eqn-A.9}  J(t\wedge \sigma) =J_\sigma(t \wedge \sigma) \mbox{ for all $t\geq 0$, $\mathbb{P}$-a.s.}
\end{equation}
On the other hand, we trivially  have the following identities,
\begin{equation}
\label{eqn-A.10}  J_\sigma(t) =I_\sigma(t ) \mbox{ for all $t\geq 0$, $\mathbb{P}$-a.s.,}
\end{equation}
\begin{equation}
\label{eqn-A.11}  J(t) =I_\tau(t ) \mbox{ for all $t\geq 0$, $\mathbb{P}$-a.s.}
\end{equation}
Therefore, by \eqref{eqn-A.10}, \eqref{eqn-A.8} and  \eqref{eqn-A.11} (in that order), we infer that
\[
I_\sigma(t \wedge \sigma ) = J_\sigma(t \wedge \sigma)=J(t\wedge \sigma)=I_\tau(t \wedge \sigma ) \mbox{ for all $t\geq 0$, $\mathbb{P}$-a.s..}
\]
what proves equality \eqref{eqn-A.6} and the corollary.
\end{proof}
\begin{Rem}
\label{rem-stopped sc} The approach from \cite{Brz+Masl+Seidler_2005} we  follow here  was used implicitly in several papers in particular, in the paper \cite{Brz+Gat_99}, but it seems to have been discussed explicitly for
the first time only in \cite{Carroll_1999}, section 4.3 (in a way different from the one presented above).
\end{Rem}

\end{document}